\documentclass[a4paper, 10pt, parskip=half]{scrartcl}

\usepackage[utf8]{inputenc}
\usepackage[T1]{fontenc}
\usepackage{lmodern}
\usepackage{csquotes}
\usepackage{amsmath}
\usepackage{amssymb}
\usepackage{amsthm}
\usepackage{amsfonts}
\usepackage{dsfont}
\usepackage{mathtools}
\usepackage{hyperref}
\hypersetup{%
  colorlinks=true,
  linkcolor=black,
  citecolor=black,
  urlcolor=blue,
  pdftitle={},
  pdfauthor={Mario Fuest},
  pdfkeywords={},
  bookmarksopen=true,
}

\usepackage[numbers, sort&compress]{natbib}

\RequirePackage{geometry}
\newcommand{\R}{\mathbb{R}}

\newcommand{\N}{\mathbb{N}}

\newcommand{\mc}[1]{\mathcal{#1}}

\newcommand{\ur}[1]{\mathrm{#1}}
\newcommand{\ure}{\ur{e}}

\ifdefined\labelenumi%
  \renewcommand{\labelenumi}{(\roman{enumi})}
  
\fi

\newcommand{\eps}{\varepsilon}

\newcommand{\gt}{>}
\newcommand{\lt}{<}

\DeclareMathOperator{\supp}{supp}

\DeclareMathOperator{\spanset}{span}

\newcommand{\defs}{\coloneqq}
\newcommand{\sfed}{\eqqcolon}

\newcommand{\ra}{\rightarrow}

\newcommand{\nea}{\nearrow}

\newcommand{\sea}{\searrow}

\newcommand{\rh}{\rightharpoonup}

\newcommand{\ol}{\overline}
\newcommand{\ul}{\underline}
\newcommand{\wt}{\widetilde}

\newcommand{\diff}{\,\mathrm{d}}

\newcommand{\dt}{\,\mathrm{d}t}

\newcommand{\dsigma}{\,\mathrm{d}\sigma}
\newcommand{\drho}{\,\mathrm{d}\rho}

\newcommand{\ddt}{\frac{\mathrm{d}}{\mathrm{d}t}}

\DeclareMathOperator{\sign}{sign}

\newcommand{\embed}{\hookrightarrow}

\newcommand{\hp}{\hphantom}
\newcommand{\pe}{\mathrel{\hp{=}}}

\newcommand{\tmaxedk}{T_{\max, \eps \delta k}}

\newcommand{\intom}{\int_\Omega}
\newcommand{\intnt}{\int_0^T}
\newcommand{\intnst}{\int_0^t}

\newcommand{\intntom}{\int_0^T \int_\Omega}
\newcommand{\intnstom}{\int_0^t \int_\Omega}
\newcommand{\intninfom}{\int_0^\infty \int_\Omega}

\newcommand{\Ombar}{\ol \Omega}
\newcommand{\loc}{\mathrm{loc}}

\newcommand{\leb}[2][\Omega]{\ensuremath{L^{#2}(#1)}}
\newcommand{\lebl}[1][\Omega]{\ensuremath{L\log L(#1)}}
\newcommand{\sob}[3][\Omega]{\ensuremath{W^{#2, #3}(#1)}}
\newcommand{\sobn}[3][\Omega]{\ensuremath{W_N^{#2, #3}(#1)}}

\newcommand{\con}[2][\Ombar]{\ensuremath{C^{#2}(#1)}}

\newcommand{\dual}[1]{\ensuremath{(#1)^\star}}
\newcommand{\dualn}[1]{\ensuremath{#1^\star}}
\newcommand{\Wt}{\ensuremath{W^{1, 2}((0, T); \sob12)}}
\newcommand{\Winf}{\ensuremath{W_{\loc}^{1, 2}([0, \infty); \sob12)}}

\newcommand{\ue}{u_\eps}

\newcommand{\uet}{u_{\eps t}}

\newcommand{\uej}{{u_{\eps_j}}}

\newcommand{\ued}{u_{\eps \delta}}
\newcommand{\uedj}{u_{\eps \delta_j}}
\newcommand{\uedjt}{u_{\eps \delta_j t}}

\newcommand{\uedt}{u_{\eps \delta t}}
\newcommand{\uedk}{u_{\eps \delta k}}
\newcommand{\uedkt}{u_{\eps \delta k t}}
\newcommand{\uedkj}{u_{\eps \delta k_j}}
\newcommand{\uedkjt}{u_{\eps \delta k_j t}}
\newcommand{\ve}{v_\eps}

\newcommand{\vet}{v_{\eps t}}

\newcommand{\vej}{{v_{\eps_j}}}

\newcommand{\ved}{v_{\eps \delta}}
\newcommand{\vedj}{v_{\eps \delta_j}}
\newcommand{\vedjt}{v_{\eps \delta_j t}}

\newcommand{\vedt}{v_{\eps \delta t}}
\newcommand{\vedk}{v_{\eps \delta k}}
\newcommand{\vedkt}{v_{\eps \delta k t}}
\newcommand{\vedkj}{v_{\eps \delta k_j}}
\newcommand{\vedkjt}{v_{\eps \delta k_j t}}
\newcommand{\uejdj}{{u_{\eps_j \delta_{j'}}}}
\newcommand{\vejdj}{{v_{\eps_j \delta_{j'}}}}

\newcommand{\ua}{u_\alpha}
\newcommand{\uaj}{u_{\alpha_j}}
\newcommand{\uat}{u_{\alpha t}}
\newcommand{\va}{v_\alpha}
\newcommand{\vaj}{v_{\alpha_j}}
\newcommand{\vat}{v_{\alpha t}}
\newcommand{\una}{u_{0 \alpha}}
\newcommand{\vna}{v_{0 \alpha}}
\newcommand{\Dia}{D_{i \alpha}}
\newcommand{\Doa}{D_{1 \alpha}}

\newcommand{\Dta}{D_{2 \alpha}}

\newcommand{\Sia}{S_{i \alpha}}
\newcommand{\Soa}{S_{1 \alpha}}
\newcommand{\Soaj}{S_{1 \alpha_j}}
\newcommand{\Sta}{S_{2 \alpha}}

\newcommand{\fia}{f_{i \alpha}}
\newcommand{\foa}{f_{1 \alpha}}
\newcommand{\foaj}{f_{1 \alpha_j}}
\newcommand{\fta}{f_{2 \alpha}}

\newcommand{\xiia}{\xi_{\alpha}}
\newcommand{\xioa}{\xi_{\alpha}}

\newcommand{\xita}{\xi_{\alpha}}

\newcommand{\Ba}{B_\alpha}

\newcommand{\Baj}{B_{\alpha_j}}

\newcommand{\Eed}{\mc E_{\eps \delta}}
\newcommand{\Ded}{\mc D_{\eps \delta}}
\newcommand{\Red}{\mc R_{\eps \delta}}

\newcommand{\Sod}{S_{1 \delta}}
\newcommand{\Sodj}{S_{1 \delta_j}}
\newcommand{\Sodjp}{S_{1 \delta_{j'}}}
\newcommand{\Std}{S_{2 \delta}}

\newcommand{\Sid}{S_{i \delta}}
\newcommand{\fod}{f_{1 \delta}}
\newcommand{\fodj}{f_{1 \delta_j}}
\newcommand{\ftd}{f_{2 \delta}}
\newcommand{\fid}{f_{i \delta}}
\newcommand{\God}{G_{1 \delta}}
\newcommand{\Godj}{G_{1 \delta_j}}
\newcommand{\Godjp}{G_{1 \delta_{j'}}}
\newcommand{\Gtd}{G_{2 \delta}}
\newcommand{\Gid}{G_{i \delta}}

\newcommand{\fo}{f_1}
\newcommand{\ft}{f_2}

\newcommand{\Goa}{G_{1 \alpha}}
\newcommand{\Gta}{G_{2 \alpha}}
\newcommand{\Gia}{G_{i \alpha}}

\newcommand{\cp}{C_{\mathrm P}}

\newcommand{\tops}{\texorpdfstring}

\makeatletter
\renewenvironment{proof}[1][\proofname]{\par
  \pushQED{\qed}%
  \normalfont \topsep0\p@\relax
  \trivlist
  \item[\hskip\labelsep\scshape
  #1\@addpunct{.}]\ignorespaces
}{%
  \popQED\endtrivlist\@endpefalse
}
\makeatother

\newtheorem{base}{Base}[section]
\numberwithin{equation}{section}

\newtheorem{theorem}[base]{Theorem} \newtheorem*{theorem*}{Theorem}
\newtheorem{lemma}[base]{Lemma} \newtheorem*{lemma*}{Lemma}
 \newtheorem*{prop*}{Proposition}
 \newtheorem*{cor*}{Corollary}

 \newtheorem*{remark*}{Remark}

\theoremstyle{definition}
\newtheorem{definition}[base]{Definition} \newtheorem*{definition*}{Definition}
 \newtheorem*{example*}{Example}
 \newtheorem*{cond*}{Condition}

\begin{document}
\setkomafont{title}{\normalfont\Large}
\title{Global weak solutions to fully cross-diffusive systems with nonlinear diffusion and saturated taxis sensitivity}
\author{%
Mario Fuest\footnote{fuestm@math.upb.de}\\
{\small Institut f\"ur Mathematik, Universit\"at Paderborn,}\\
{\small 33098 Paderborn, Germany}
}
\date{}

\maketitle

\KOMAoptions{abstract=true}
\begin{abstract}
\noindent
Systems of the type
\begin{align}\label{prob:abstract}\tag{$\star$}
  \begin{cases}
    u_t = \nabla \cdot (D_1(u) \nabla u - S_1(u) \nabla v) + f_1(u, v), \\
    v_t = \nabla \cdot (D_2(v) \nabla v + S_2(v) \nabla u) + f_2(u, v)
  \end{cases}
\end{align}
can be used to model pursuit--evasion relationships between predators and prey.
Apart from local kinetics given by $f_1$ and $f_2$,
the key components in this system are the taxis terms $-\nabla \cdot (S_1(u) \nabla v)$ and $+\nabla \cdot (S_2(v) \nabla u)$;
that is, the species are not only assumed to move around randomly in space but are also able to partially direct their movement
depending on the nearby presence of the other species.\\[0.5pt]
In the present article,
we construct global weak solutions of \eqref{prob:abstract} for certain prototypical nonlinear functions $D_i$, $S_i$ and $f_i$, $i \in \{1, 2\}$.
To that end, we first make use of a fourth-order regularization to obtain global solutions to approximate systems
and then rely on an entropy-like identity associated with \eqref{prob:abstract} for obtaining various a~priori estimates. \\[0.5pt]
 \textbf{AMS Classification (2020):} {35K51 (primary); 35B45, 35D30, 35K59, 92C17 (secondary)}
\end{abstract}

\section{Introduction}
Predator--prey relationships are often described by systems of differential equations.
While systems of ordinary differential equations essentially assume a spatially homogeneous setting,
the simplest way to account for nontrivial spatial behavior is to assume that the species move around randomly.
However, sufficiently intelligent predators and prey may also partially orient their movement
towards or away from higher concentrations of the other species.
In order to capture these abilities,
\cite{TsyganovEtAlQuasisolitonInteractionPursuit2003}
proposes the so-called \emph{pursuit--evasion model}
\begin{align}\label{prob:lin}
  \begin{cases}
    u_t = \nabla \cdot (d_1 \nabla u - \chi_1 u \nabla v) + \fo(u, v), \\
    v_t = \nabla \cdot (d_2 \nabla v + \chi_2 v \nabla u) + \ft(u, v),
  \end{cases}
\end{align}
where $u$, $v$ correspond to the predator and prey densities,
$d_1, d_2, \chi_1, \chi_2 \gt 0$ are given parameters
and $\fo$, $\ft$ relate to intrinsic growth and certain functional responses.
We discuss reasons for various choices of the later in a moment, but first motivate the fluxes present in \eqref{prob:lin}.

Crucially, both species are not only assumed to move randomly in their habitat (which is modelled by the diffusion terms $\nabla \cdot (d_1 \nabla u)$ and $\nabla \cdot (d_2 \nabla v)$
with diffusion strength characterised by the parameters $d_1$ and $d_2$),
but may also partially direct their movement in response to the presence of the other species.
More concretely, predators move towards higher prey concentrations and the prey seeks to avoid high predator concentrations.
These effects are called attractive prey- and repulsive predator-taxis and modelled by the terms $-\nabla \cdot (\chi_1 u \nabla v)$ and $+\nabla \cdot (\chi_2 v \nabla u)$, respectively,
again with the strength of the effects indicated by $\chi_1$ and $\chi_2$.

Similar terms are also present in the \emph{minimal Keller--Segel system}
\begin{align}\label{prob:ks}
  \begin{cases}
    u_t = \nabla \cdot (d_1 \nabla u - \chi_1 u \nabla v), \\
    v_t = d_2 \Delta v - v + u,
  \end{cases}
\end{align}
which has been proposed in \cite{KellerSegelInitiationSlimeMold1970}
to describe the behavior of the slime mold \emph{Dictyostelium discoideum} with density $u$,
which are attracted by the chemical substance with density $v$ they produce themselves.
A key feature of this organism is to spontaneously form structures,
which is reflected mathematically by the existence of solutions to \eqref{prob:ks} which aggregate so strongly that they blow up in finite time
\cite{HerreroVelazquezBlowupMechanismChemotaxis1997, WinklerFinitetimeBlowupHigherdimensional2013}.
For an overview of blow-up results and techniques for obtaining them, we refer to the recent survey \cite{LankeitWinklerFacingLowRegularity2019}.
In addition to questions of global well-posedness,
various other aspects of \eqref{prob:ks} and relatives thereof have been analyzed,
culminating in a huge body of mathematical literature
(see for instance the survey \cite{BellomoEtAlMathematicalTheoryKeller2015}).

In contrast to \eqref{prob:lin}, however, \eqref{prob:ks} only features a single taxis term---which already forms a huge obstacle for obtaining global classical solutions,
in that it can cause finite-time blow-up, as indicated above.
Thus it has to be expected that solutions to doubly cross-diffusive systems such as \eqref{prob:lin} are generally even less regular
or, at least, that constructing global (weak) solutions for such systems is more challenging.
Accordingly, global existence results regarding \eqref{prob:lin} are, up to now, quite limited.
Nontrivial unconditional a priori estimates can apparently only been derived by means of a certain entropy-like identity (which we discuss below in more detail),
which in the one-dimensional setting are barely sufficient to conclude the existence of global weak solutions \cite{TaoWinklerExistenceTheoryQualitative2020, TaoWinklerFullyCrossdiffusiveTwocomponent2021}.
Apart from that, only certain conditional functional inequalities, which fail to hold for general data, are known to exist.
In \cite{FuestGlobalSolutionsHomogeneous2020}, these have successfully been employed in order to prove existence of global classical solutions
under the assumption that the initial data are sufficiently close to homogeneous steady states.

We also remark that, apart from \eqref{prob:lin}, several other fully cross-diffusive systems have been examined,
of which the one proposed by Shigesada, Kawasaki and Teramoto to model spatial segregation \cite{ShigesadaEtAlSpatialSegregationInteracting1979},
which we henceforth call the SKT model,
is certainly one of the most famous.
Again, the cross-diffusive terms pose challenges for obtaining any global existence results
and, accordingly, global classical solutions are only known to exist in certain specific situations; see for instance \cite{DeuringInitialboundaryValueProblem1987, LouWinklerGlobalExistenceUniform2015}.
Moreover, a quite general global solution theory for such cross-diffusive systems has been developed,
both for weak \cite{JungelBoundednessbyentropyMethodCrossdiffusion2015}
and renormalized \cite{ChenJungelGlobalRenormalizedSolutions2019, FischerGlobalExistenceRenormalized2015} solutions,
which is in part applicable to the SKT model.
Unfortunately, however, the techniques employed there are not transferable to the system \eqref{prob:lin},
the main reasons being that stronger versions of the entropy-like inequality \eqref{eq:intro:entropy} below would be needed.
For a more thorough comparison between the SKT model and \eqref{prob:lin}, we refer to the introduction of \cite{TaoWinklerFullyCrossdiffusiveTwocomponent2021}.

\paragraph{Nonlinear diffusion and saturated taxis sensitivities.}
Over time, various modifications of the minimal Keller--Segel system \eqref{prob:ks} have been proposed,
see \cite{HillenPainterUserGuidePDE2009} for a (non-exhaustive) list.
Prominent examples include replacing the linear diffusion term with a quasilinear one and allowing for nonlinear taxis sensitivities.
While in part this has already been suggested by Keller and Segel in \cite{KellerSegelInitiationSlimeMold1970},
the need for these adjustments has been further emphasized in \cite{PainterHillenVolumefillingQuorumsensingModels2002}
(see also \cite{HillenPainterUserGuidePDE2009, WrzosekVolumeFillingEffect2010}).
A key reason is the desire to incorporate effects such as volume-filling;
that is, to take into account that the motility of bacteria may be impacted by the availability of free nearby space and thus reduced by a high presence of other bacteria.

Apart from biological motivations, suitable nonlinearities may also improve the regularity of the system.
When replacing the constants $d_1$ and $\chi_1$ in \eqref{prob:ks} with functions of $u$,
if the growth rate of $\frac{d_1}{\chi_1}$ is higher than a certain threshold,
solutions exist always globally in time and are bounded, even in situations where blow-up can occur for constant $d_1$ and $\chi_1$
\cite{HorstmannWinklerBoundednessVsBlowup2005, IshidaEtAlBoundednessQuasilinearKeller2014, TaoWinklerBoundednessQuasilinearParabolic2012}.
On the other hand, if the growth rate is below that threshold, solutions may be unbounded
\cite{HorstmannWinklerBoundednessVsBlowup2005, WinklerDoesVolumefillingEffect2009}.
That is, if the motility of bacteria is only slightly reduced in high-density environments, chemotaxis may still lead to overcrowding.
Determining when exactly these solutions fail to exist globally in time is still an open question,
although quite large classes of examples for both finite-time 
\cite{CieslakStinnerFinitetimeBlowupGlobalintime2012, CieslakStinnerFinitetimeBlowupSupercritical2014, CieslakStinnerNewCriticalExponents2015}
and infinite-time blow-up \cite{WinklerGlobalClassicalSolvability2019} have been detected.
Generally, it is conjectured that the system \eqref{prob:ks} with nonconstant $d_1, \chi_1$
possesses similar properties as parabolic--elliptic simplifications thereof,
for which a more complete answer to the question of global existence is available
\cite{LankeitInfiniteTimeBlowup2020, WinklerDjieBoundednessFinitetimeCollapse2010}.

In the present article, we adapt these ideas to the model \eqref{prob:lin} and thus consider the system
\begin{align}\label{prob:nonlin}\tag{P}
  \begin{cases}
    u_t = \nabla \cdot (D_1(u) \nabla u - S_1(u) \nabla v) + \fo(u, v)  & \text{in $\Omega \times (0, \infty)$}, \\
    u_t = \nabla \cdot (D_2(v) \nabla v + S_2(v) \nabla u) + \ft(u, v)  & \text{in $\Omega \times (0, \infty)$}, \\
    \partial_\nu u = \partial_\nu v = 0                                 & \text{on $\partial \Omega \times (0, \infty)$}, \\
    u(\cdot, 0) = u_0, v(\cdot, 0) = v_0                                & \text{in $\Omega$}
  \end{cases}
\end{align}
in smooth bounded domains $\Omega \subset \R^n$, $n \in \N$.
Although the methods established below would allow for more general choices,
mainly for the sake of clarity we confine ourselves to certain prototypical functions in \eqref{prob:nonlin};
that is, we set
\begin{align}\label{eq:intro:Di_Si}
  D_i(s) \defs d_i (s + 1)^{m_i - 1}
  \quad \text{and} \quad
  S_i(s) \defs \chi_i s (s + 1)^{q_i - 1}
\end{align}
for $s \ge 0$ and $i \in \{1, 2\}$,
and where the parameters therein are such that
\begin{align}\label{eq:intro:di_chii}
  d_1, d_2, \chi_1, \chi_2 \gt 0, \,
  m_1, m_2 \in \R, \,
  q_1, q_2 \in (-\infty, 1].
\end{align}
Moreover, we choose to either neglect zeroth order kinetics altogether
or assume a typical Lotka--Volterra-type predator--prey interaction,
meaning that apart from intrinsic growth the functions $f_1$ and $f_2$ should reflect that interspecies encounters are beneficial for the predator and harmful to the prey.
Again, while the techniques developed in this article could be employed for a variety of choices for $f_1$ and $f_2$ and thus for various functional responses,
we only consider the same zeroth order terms as in \cite{TaoWinklerExistenceTheoryQualitative2020} and \cite{FuestGlobalSolutionsHomogeneous2020},
namely
\begin{align}\label{eq:intro:fi}
  f_i(s_1, s_2) \defs \lambda_i s_i - \mu_i s_i^2 + (-1)^{i+1} a_i s_1 s_2
\end{align}
for $s_1, s_2 \ge 0$ and $i \in \{1, 2\}$,
where
\begin{align}
  \label{eq:intro:h1}\tag{H1}
    \text{either} &\qquad \lambda_1, \lambda_2, \mu_1, \mu_2, a_1, a_2 = 0  \\
  \label{eq:intro:h2}\tag{H2}
    \text{or} &\qquad \lambda_1, \lambda_2, \mu_1, \mu_2, a_1, a_2 \gt 0.
\end{align}

\paragraph{The entropy-like identity.}
Our goal is to construct global weak solutions of \eqref{prob:nonlin} for widely arbitrary initial data.
Thus, conditional estimates valid only as long $u$ and $v$ are close to certain steady states
(such as those derived in \cite{FuestGlobalSolutionsHomogeneous2020})
are evidently insufficient for our purposes.
Instead,
we will rely on the following unconditional entropy-like identity 
which has already been made use of in \cite{TaoWinklerExistenceTheoryQualitative2020, TaoWinklerFullyCrossdiffusiveTwocomponent2021} for related systems.
Setting
\begin{align*}
  G_i(s) \defs \int_1^s \int_1^\rho \frac{1}{S_i(\sigma)} \dsigma \drho
  \qquad \text{for $s \ge 0$ and $i \in \{1, 2\}$},
\end{align*}
a sufficiently smooth and positive global solution $(u, v)$ to \eqref{prob:nonlin} satisfies
\begin{align}\label{eq:intro:entropy}
  &\pe  \ddt \left( \intom G_1(u) + \intom G_2(v) \right)
        + \intom \frac{D_1(u)}{S_1(u)} |\nabla u|^2 + \intom \frac{D_2(v)}{S_2(v)} |\nabla v|^2 \notag \\
  &=    \intom \left( \frac{S_1(u)}{S_1(u)} - \frac{S_2(v)}{S_2(v)} \right) \nabla u \cdot \nabla v
        + \intom G_1'(u) \fo(u, v) + \intom G_2'(v) \ft(u, v)
  \qquad \text{in $(0, \infty)$}.
\end{align}
This functional inequality constitutes the main---if not essentially the only---source for a priori estimates.
In order to indeed gain any useful bounds from \eqref{eq:intro:entropy}, however, we have to control the right-hand side therein.
Evidently, the first term there just vanishes; the functions $G_1$ and $G_2$
have been chosen precisely to guarantee a cancellation of the cross-diffusive contributions.

Moreover, the last two summands on the right-hand side in \eqref{eq:intro:entropy}
also simply vanish if \eqref{eq:intro:h1} holds and
they can be easily controlled if there are $C_1, C_2 \gt 0$ such that
\begin{align}\label{eq:intro:cond_f1}\tag{F1}
      G_1'(s_1) \fo(s_1, s_2) + G_2'(s_2) \ft(s_1, s_2)
  \le - C_1 s_1^2 \ln s_1 - C_1 s_2^2 \ln s_2 + C_2
  \qquad \text{for all $s_1, s_2 \ge 1$}.
\end{align}
(We note that, while for bounding the right-hand side in \eqref{eq:intro:entropy} it would suffice to take $C_1 = 0$,
positive values of $C_1$ guarantee uniform integrability of $f_i(u, v)$
which in turn will allow us to undertake certain limit processes in approximative problems.)
Unfortunately, \eqref{eq:intro:cond_f1} cannot hold unconditionally.
Indeed, suppose $q_1 = q_2 = q \le 1$
and that \eqref{eq:intro:cond_f1} holds for $C_1 = 0$ and some $C_2 \gt 0$.
Taking $s_1 = s_2 = s \ge 1$ in \eqref{eq:intro:cond_f1} then implies
\begin{align*}
        C_2
  &\ge  G_1'(s) (\lambda_1 s - \mu_1 s^2 + a_1 s^2)
        + G_2'(s) (\lambda_2 s - \mu_2 s^2 - a_2 s^2) \\
  &\ge  \int_1^{s} \frac{(\sigma + 1)^{1-q}}{\sigma} \dsigma \left( \frac{-\mu_1 + a_1}{\chi_1} + \frac{-\mu_2 - a_2}{\chi_2} \right) s^2,
\end{align*}
where the right-hand side diverges to $\infty$ as $s \nea \infty$,
provided $\frac{a_1}{\chi_1} \gt \frac{\mu_1}{\chi_1} + \frac{\mu_2}{\chi_2} + \frac{a_2}{\chi_2}$.
Still, in the case of $q_1 = q_2 = q \le 1$,
Young's inequality shows that \eqref{eq:intro:cond_f1} holds
provided $a_1$ is sufficiently small or $\chi_1$ is sufficiently large compared to the other parameters, for instance.

Of course, instead of \eqref{eq:intro:cond_f1} one may also rely on the dissipative terms in \eqref{eq:intro:entropy}
for controlling the right-hand side in \eqref{eq:intro:entropy}
and this idea will allow us to derive another sufficient condition for bounding the right-hand side in \eqref{eq:intro:entropy}.
As integrating certain linear combinations of the first two equations in \eqref{prob:nonlin} provides us
with a locally uniform-in-time $\leb1$ bound for both $u$ and $v$,
combining the Gagliardo--Nirenberg and Young inequalities shows that requiring 
\begin{align}\label{eq:intro:cond_f2}\tag{F2}
  m_1 \gt \frac{2n-2}{n} + \frac{(3-q_2)(2-q_1) - (3-q_1)(2-q_2)}{2-q_2}
  \quad \text{or} \quad
  m_2 \gt \frac{2n-2}{n} + (q_2 - q_1)
\end{align}
suffices to estimate the right-hand side in \eqref{eq:intro:entropy} against the dissipative terms therein (cf.\ Lemma~\ref{lm:rhs_bdd_f2}).
We note that if $q_1 = q_2$, then \eqref{eq:intro:cond_f2} is equivalent to $\max\{m_1, m_2\} \gt \frac{2n-2}{n}$.

Next, one could discuss more refined approaches and for instance
also make use of the $L^2$ space-time bounds
(which in the case of \eqref{eq:intro:h2} result as a by-product when obtaining $\leb1$ bounds).
However, here we confine ourselves to the conditions \eqref{eq:intro:cond_f1} and \eqref{eq:intro:cond_f2},
mainly because treating the most general case possible would lead to several technical difficulties which we would like to rather avoid here.
Still, the important special cases that
either $a_1$ is small or $\chi_1$ is large (condition \eqref{eq:intro:cond_f1}) or $m_1$ or $m_2$ are large (condition \eqref{eq:intro:cond_f2})
are included in our analysis
and, as the examples above show, at least qualitatively, these conditions seem to be optimal.

\paragraph{Obtaining further a priori estimates.}
With the right-hand side of \eqref{eq:intro:entropy} under control, we then make use of (a corollary of) the Gagliardo--Nirenberg inequality
to obtain space-time bounds for $u, v, \nabla u$ and $\nabla v$.
That is, assuming
\begin{align}\label{eq:intro:mi_qi}
  m_i - q_i \gt - 1
  \qquad \text{for $i \in \{1, 2\}$},
\end{align}
we can obtain estimates in $L^{p_1}$, $L^{p_2}$, $L^{r_1}$ and $L^{r_2}$, respectively, where
\begin{align}\label{eq:intro:pi}
  p_i &\defs
  \begin{cases}
    \max\{m_i + 1 - q_i + \frac{2(2-q_i)}{n}, 2 - q_i\},  & \text{if \eqref{eq:intro:h1} holds} \\
    \max\{m_i + 1 - q_i + \frac{2(2-q_i)}{n}, 3 - q_i\},  & \text{if \eqref{eq:intro:h2} holds}
  \end{cases}
  \qquad \text{for $i \in \{1, 2\}$}
\intertext{and}\label{eq:intro:ri}
  r_i &\defs \min\left\{ \frac{2p_i}{p_i - (m_i - q_i - 1)}, 2 \right\},
  \qquad \text{for $i \in \{1, 2\}$},
\end{align}
see Lemma~\ref{lm:space_time_bdds_alpha} and Lemma~\ref{lm:gradient_space_time_bdds_alpha}.

Lacking any other sources of helpful a priori bounds,
these estimates need to be strong enough to inter alia assert convergence of the corresponding approximative terms to
\begin{align*}
  \intninfom S_1(u) \nabla v \cdot \nabla \varphi
  \quad \text{and} \quad
  \intninfom S_2(v) \nabla u \cdot \nabla \varphi,
  \qquad \varphi \in C_c^\infty(\Ombar \times [0, \infty)).
\end{align*}
This is the case when $p_i$ and $r_i$ are sufficiently large.
More precisely, we need to require
\begin{align}\label{eq:intro:main_cond}
  \begin{cases}
    \frac{1}{r_{3-i}} \lt 1,                    & q_i \le 0 , \\
    \frac{q_i}{p_i} + \frac{1}{r_{3-i}} \lt 1,  & 0 \lt q_i \lt 1, \\
    \frac{1}{p_i} + \frac{1}{r_{3-i}} \le 1,    & q_i = 1,
  \end{cases}
  \qquad \text{for $i \in \{1, 2\}$}
\end{align}
(In the case of $q_i = 1$, we obtain slightly stronger bounds than outlined above
so that equality in \eqref{eq:intro:main_cond} is sufficient for that case.)
We remark that if $m_i = m \in \R$ and $q_i = q \in (-\infty, 1]$ for $i \in \{1, 2\}$,
then $q \le 0$ implies \eqref{eq:intro:main_cond}
while for $q \in (0, 1)$ and if \eqref{eq:intro:h1} holds,
\eqref{eq:intro:main_cond} is equivalent to
\begin{align}\label{eq:intro:cond_m_h1}
  m \gt \min\left\{ \frac{(2n + 1)q - 2}{n}, 4q-1 \right\}
\end{align}
Moreover, in the case of \eqref{eq:intro:h2} (and again $q \in (0, 1)$),
\eqref{eq:intro:main_cond} is not only implied by \eqref{eq:intro:cond_m_h1} but also by $m \gt 4q - 2$.

\paragraph{Main results.}
Under these assumptions, we are then finally able to construct global weak solutions to \eqref{prob:nonlin}.
\begin{theorem}\label{th:ex_weak_nonlin}
  Let $\Omega \subset \R^n$, $n \in \N$, be a smooth, bounded domain.
  Suppose that \eqref{eq:intro:Di_Si}, \eqref{eq:intro:di_chii}, \eqref{eq:intro:fi}, \eqref{eq:intro:mi_qi}, 
  either \eqref{eq:intro:h1} or \eqref{eq:intro:h2},
  \eqref{eq:intro:cond_f1} or \eqref{eq:intro:cond_f2},
  as well as \eqref{eq:intro:main_cond} (with $p_i$ and $r_i$ as in \eqref{eq:intro:pi} and \eqref{eq:intro:ri}, respectively) hold
  and that
  \begin{align}\label{eq:intro:cond_init}
    u_0, v_0 \in 
    \begin{cases}
      \leb{2-q_i},  & q_i \lt 1, \\
      \lebl,        & q_i = 1
    \end{cases}
    \quad \text{are nonnegative}.
  \end{align}
  Then there exists a global nonnegative weak solution $(u, v)$ of \eqref{prob:nonlin} in the sense of Definition~\ref{def:weak_sol_main}.
\end{theorem}
In conclusion, under certain conditions we are able show the existence of global weak solutions for variants of the pursuit--evasion model.
As already mentioned above, the techniques employed in this paper should also be applicable for different choices of $D_i$, $S_i$ and $f_i$.
In particular, Theorem~\ref{th:weak_sol_nonlin} below may serve as a starting point for global existence results of related systems.

Moreover, Theorem~\ref{th:ex_weak_nonlin} can be seen as a prerequisite for further analysis.
That is, only if (global) solutions are known to exist, it sensible to ask questions such as:
Can patterns emerge at intermediate or large time scales?
Are certain homogeneous steady states globally attractive in the sense that the solutions constructed in Theorem~\ref{th:ex_weak_nonlin} converge towards them?
Can one show that taxis mechanisms are actually beneficial for the species subject to them,
perhaps by comparing qualitative and quantitative results for \eqref{prob:nonlin} with those for systems without predator- or prey-taxis?
However, all of these are out of scope for the current paper and thus left for further research.
 
\paragraph{Structure of the paper.}
A challenge not yet addressed is the construction of global solutions to certain approximative problems.
For systems similar to \eqref{prob:nonlin} but where either $S_1 \equiv 0$ or $S_2 \equiv 0$,
this is usually a straightforward task.
For the fully cross-diffusive system \eqref{prob:nonlin}, however,
even if all given functions are assumed to be bounded,
the question of global existence is already highly nontrivial, even for a weak solution concept.

Thus, Section~\ref{sec:approx} is devoted to the construction of so-called weak $W^{1, 2}$-solutions
to systems suitably approximating \eqref{prob:nonlin}.
The corresponding proof then relies on an additional approximation; we make use of fourth-order regularization terms.
The general strategy is described more thoroughly at the beginning of Section~\ref{sec:approx},
so we do not go into much more detail at this point.
However, it seems worth emphasizing that apart from obtaining these solutions,
we also prove a corresponding version of the entropy-like identity \eqref{eq:intro:entropy}.

Next, in Section~\ref{sec:approx_main},
we fix the final approximation functions used
and rely on the results in the preceding section to obtain a global weak $W^{1, 2}$-solution fulfilling a certain entropy-like inequality,
see Lemma~\ref{lm:ex_ua_va}.

Section~\ref{sec:lim_alpha_sea_0} then makes use of this inequality and the hypotheses of Theorem~\ref{th:ex_weak_nonlin}
in order to guarantee sufficiently strong convergence towards a function pair $(u, v)$,
which in Section~\ref{sec:proof_11} is then finally seen to be a weak solution of \eqref{prob:nonlin}.

\paragraph{Notation.}
Throughout the article, we fix $n \in \N$ and a smooth bounded domain $\Omega \subset \R^n$.
For $p \in (1, \infty)$, we set $\sobn 2p \defs \{\, \varphi \in \sob 2p : \partial_\nu \varphi = 0 \text{ in the sense of traces}\,\}$.

Additionally, we use the following notation for Sobolev spaces involving evolution triples.
For an interval $I \subset \R$ and an evolution triple $V \embed H \embed \dualn V$,
we set $W^{1, 2}(I; V, H) \defs \{\,\varphi \in L^2(I; V) : \varphi_t \in L^2(I; \dualn V)\,\}$
as well as $W_{\loc}^{1, 2}(I; V, H) \defs \bigcup_{[a, b] \subset I} W^{1, 2}([a, b]; V, H)$
and abbreviate $W_{(\loc)}^{1, 2}(I; \sob12) \defs W_{(\loc)}^{1, 2}(I; \sob12, \leb2)$.

Moreover, for a set $X$, a function $\varphi \colon X \to \R$ and $A \in \R$,
we abbreviate $\{\,x \in X : \varphi(x) \le A\,\}$ by $\{\varphi \le A\}$,
the set $X$ being implied by the context.
Similarly for other order relations.

\section{Global weak \tops{$W^{1, 2}$}{W12}-solutions to approximative systems}\label{sec:approx}
In this section, we prove the following quite general global existence theorem,
which we will then use in Section~\ref{sec:approx_main} to obtain solutions to certain approximate problems.
In contrast to the hypotheses of Theorem~\ref{th:ex_weak_nonlin},
here we also assume that all given functions are bounded.
That is, in this section, we do not need to assume any of the conditions introduced in the introduction
but instead require that \eqref{eq:weak_sol_nonlin:spaces_d}--\eqref{eq:weak_sol_nonlin:initial} below are fulfilled.
\begin{theorem}\label{th:weak_sol_nonlin}
  Suppose that, for $i \in \{1, 2\}$,
  \begin{align}
    D_i &\in C^0([0, \infty)) \cap L^\infty((0, \infty)), \label{eq:weak_sol_nonlin:spaces_d} \\
    S_i &\in C^1([0, \infty)) \cap W^{1, \infty}((0, \infty)) \quad \text{and} \label{eq:weak_sol_nonlin:spaces_s} \\
    f_i &\in C^0([0, \infty)^2) \cap L^\infty((0, \infty)^2) \label{eq:weak_sol_nonlin:spaces_f}
  \end{align}
  fulfill
  \begin{align}\label{eq:weak_sol_nonlin:cond_d_s}
    \inf_{s \in [0, \infty)} D_i(s) \gt 0, \quad
    \inf_{s \in (0, 1)} \frac{S_i(s)}{s} \gt 0, \quad
    \inf_{s \in [1, \infty)} S_i(s) \gt 0
    \quad \text{and} \quad
    S_i(0) = 0
  \end{align}
  as well as
  \begin{align}\label{eq:weak_sol_nonlin:cond_f}
    \lim_{s_1 \sea 0} \sup_{s_2 \ge 0} |f_1(s_1, s_2) \ln s_1| = 0
    \quad \text{and} \quad
    \lim_{s_2 \sea 0} \sup_{s_1 \ge 0} |f_2(s_1, s_2) \ln s_2| = 0
  \end{align}
  and assume that
  \begin{align}\label{eq:weak_sol_nonlin:initial}
    u_0, v_0 \in \con\infty
    \quad \text{are positive in $\Ombar$}.
  \end{align}
  Then there exists a global nonnegative weak $W^{1, 2}$-solution $(u, v)$ of \eqref{prob:nonlin},
  meaning that $u$ and $v$ belong to the space $\Winf$,
  satisfy
  \begin{align}\label{eq:weak_sol_nonlin:init_cond}
    u(\cdot, 0) = u_0
    \quad \text{as well as} \quad
    v(\cdot, 0) = v_0
    \qquad \text{a.e.\ in $\Omega$}
  \end{align}
  and fulfill
  \begin{align}\label{eq:weak_sol_nonlin:u_sol}
        \intninfom u_t \varphi
    &=  - \intninfom D_1(u) \nabla u \cdot \nabla \varphi
        + \intninfom S_1(u) \nabla v \cdot \nabla \varphi
        + \intninfom \fo(u, v) \varphi \\
    \intertext{as well as}\label{eq:weak_sol_nonlin:v_sol}
        \intninfom v_t \varphi
    &=  - \intninfom D_2(u) \nabla v \cdot \nabla \varphi
        - \intninfom S_2(u) \nabla u \cdot \nabla \varphi
        + \intninfom \ft(u, v) \varphi
  \end{align}
  for all $\varphi \in L_{\loc}^2([0, \infty); \sob12)$.
\end{theorem}
In what follows, we fix $D_i, S_i, f_i$, $i \in \{1, 2\}$ fulfilling \eqref{eq:weak_sol_nonlin:spaces_d}--\eqref{eq:weak_sol_nonlin:cond_d_s}
as well as $u_0, v_0$ as in \eqref{eq:weak_sol_nonlin:initial}.

As already alluded to in the introduction,
a cornerstone for gaining a~priori bounds for these solutions is the following theorem,
which shows that the solutions constructed in Theorem~\ref{th:weak_sol_nonlin} fulfill an inequality reminiscent of \eqref{eq:intro:entropy}.
\begin{theorem}\label{th:approx_entropy}
  Denote the weak $W^{1, 2}$-solution of \eqref{prob:nonlin} given by Theorem~\ref{th:weak_sol_nonlin} by $(u, v)$
  and let
  \begin{align*}
    G_i(s) \defs \int_1^s \int_1^\rho \frac{1}{S_i(\sigma)} \dsigma \drho
    \qquad \text{for $s \in \R$ and $i \in \{1, 2\}$}
  \end{align*}
  as well as
  \begin{align*}
            \mc E(t)
    &\defs  \intom G_1(u(\cdot, t))
            + \intom G_2(v(\cdot, t)), \\
            \mc D(t)
    &\defs  \intom \frac{D_1(u(\cdot, t))}{S_1(u(\cdot, t))} |\nabla u(\cdot, t)|^2
            + \intom \frac{D_2(v(\cdot, t))}{S_2(v(\cdot, t))} |\nabla v(\cdot, t)|^2 \quad \text{and} \\
            \mc R(t)
    &\defs  \intom G_1'(u(\cdot, t)) \fo(u(\cdot, t), v(\cdot, t))
            + \intom G_2'(v(\cdot, t)) \ft(u(\cdot, t), v(\cdot, t))
  \end{align*}
  for $t \in [0, \infty)$.
  (We remark that $\mc D$ and $\mc R$ are to be understood as functions in $L^0((0, \infty))$;
  that is, they are only well-defined up to modifications on null sets.)
  Then
  \begin{align}\label{eq:approx_entropy:entropy}
          \mc E(T) \zeta(T)
          + \intnt \mc D(t) \zeta(t) \dt
    &\le  \mc E(0) \zeta(0)
          + \intnt \mc R(t) \zeta(t) \dt
          + \intnt \mc E(t) \zeta'(t) \dt
  \end{align}
  for all $T \in (0, \infty)$ and $0 \le \zeta \in C^\infty([0, T])$.
\end{theorem}

Next, we describe our approach of proving the theorems above.
Similar to \cite{TaoWinklerExistenceTheoryQualitative2020, TaoWinklerFullyCrossdiffusiveTwocomponent2021},
where one-dimensional relatives of \eqref{prob:nonlin} have been studied,
our general approach is approximation by a fourth order regularization.
That is, for $\eps, \delta \in (0, 1)$, we will first construct global solutions to
\begin{align*}\label{prob:ped}\tag{$\textrm P_{\eps \delta}$}
  \begin{cases}
    \uedt = \nabla \cdot (
            - \eps \Sod(\ued) \nabla \Delta \ued
            + D_1(|\ued|) \nabla \ued
            - \Sod(\ued) \nabla \ved) 
            + \fod(\ued, \ved) &
            \text{in $\Omega \times (0, \infty)$}, \\
    \vedt = \nabla \cdot (
            - \eps \Std(\ved) \nabla \Delta \ved
            + D_2(|\ved|) \nabla \ved
            + \Std(\ved) \nabla \ued) 
            + \ftd(\ued, \ved) &
            \text{in $\Omega \times (0, \infty)$}, \\
    \partial_\nu \Delta \ued = \partial_\nu \ued = \partial_\nu \Delta \ved = \partial_\nu \ved = 0 &
            \text{on $\partial \Omega \times (0, \infty)$}, \\
    \ued(\cdot, 0) = u_0, \ved(\cdot, 0) = v_0 &
            \text{in $\Omega$},
  \end{cases}
\end{align*}
where
\begin{align}\label{eq:def_sid}
  \Sid(s) \defs S_i(|s|) + \delta
  \qquad \text{for $s \in \R$, $\delta \in (0, 1)$ and $i \in \{1, 2\}$}
\end{align}
and
\begin{align}\label{eq:def_fid}
  \fid(s_1, s_2) \defs f_i((s_1)_+, (s_2)_+)
  \qquad \text{for $s_1, s_2 \in \R$, $\delta \in (0, 1)$ and $i \in \{1, 2\}$}.
\end{align}
We note that \eqref{eq:weak_sol_nonlin:cond_f} entails $f_1(0, \cdot) \equiv 0$
and hence $\fod(\rho, \sigma) = 0$ for all $\rho \le 0$ and $\sigma \in \R$.
Likewise, $\ftd(\rho, \sigma) = 0$ for all $\rho \in \R$ and $\sigma \le 0$.

For convenience, let us introduce several abbreviations.
For $i \in \{1, 2\}$, we set
\begin{align*}
  \ol D_i \defs \|D_i\|_{L^\infty((0, \infty))}, \quad
  \ol S_i \defs \|S_i\|_{L^\infty((0, \infty))} + 1
  \quad \text{and} \quad
  \ol S'_i \defs \|S'_i\|_{L^\infty((0, \infty))}
\intertext{as well as}
  \ul D_i \defs \inf_{s \in [0, \infty)} D_i(s)
  \quad \text{and} \quad
  \ul S_i \defs \inf_{s \in (0, \infty)} S_i(s) [ (\tfrac1s - 1) \mathds 1_{(0, 1)}(s) + 1 ].
\end{align*}
Due to continuity of $S_i$ up to $0$,
the definition of $\ul S_i$ entails that $S_i(s) \ge \ul S_i s$ for all $s \in [0, 1)$, $i \in \{1, 2\}$.

The rest of this section is organized as follows.
The first step towards proving Theorem~\ref{th:weak_sol_nonlin} and Theorem~\ref{th:approx_entropy}
consists of constructing solutions to \eqref{prob:ped} and is achieved by a Galerkin approach.
To that end, non-degeneracy of the fourth order terms in \eqref{prob:ped} is of crucial importance,
which is the reason for introducing the parameter~$\delta$.

A general problem for equations of fourth-order is the lack of a maximum principle;
that is, $\ued, \ved$ might become negative even for strictly positive initial data.
Following \cite{GrunDegenerateParabolicDifferential1995}, however,
we see in Subsection~\ref{sec:delta_sea_0} that suitably constructed limit functions $\ue, \ve$ are indeed nonnegative.
Here, degeneracy for $\delta=0$ actually comes in handy.

In contrast to Section~\ref{sec:lim_alpha_sea_0}, where we aim to argue similarly but only assume the hypotheses of Theorem~\ref{th:ex_weak_nonlin},
the assumptions \eqref{eq:weak_sol_nonlin:spaces_f} and \eqref{eq:weak_sol_nonlin:cond_f} allow us to rather easily
obtain certain a~priori bounds from a version of the entropy-like identity \eqref{eq:intro:entropy}.
These allow us to so finally let $\eps \sea 0$ in Subsection~\ref{sec:eps_sea_0}
and then to prove Theorem~\ref{th:weak_sol_nonlin} and Theorem~\ref{th:approx_entropy}.

\subsection{The limit process \tops{$k \ra \infty$}{k to infty}: existence of weak solutions to \tops{\eqref{prob:ped}}{P\_(eps delta)} by a Galerkin method}\label{sec:k_to_infty}
To prepare the Galerkin approach used below for constructing solutions to \eqref{prob:ped},
we briefly state the well-known
\begin{lemma}\label{lm:ex_eigen_basis}
  There exists an orthonormal basis $\{\,\varphi_j : j \in \N\,\}$ of $\leb2$
  consisting of smooth eigenfunctions of $-\Delta$ with homogeneous Neumann boundary conditions.
\end{lemma}
\begin{proof}
  The existence of an orthonormal basis consisting of eigenfunctions of $-\Delta$ with homogeneous Neumann boundary conditions
  is given by~\cite[Theorem~1.2.8]{HenrotExtremumProblemsEigenvalues2006}
  and their smoothness is proved by iteratively applying \cite[Theorem~19.1]{FriedmanPartialDifferentialEquations1976}.
\end{proof}

For the Galerkin approach, we first construct local-in-time solutions to certain finite-dimensional problems.
\begin{lemma}\label{lm:ex_uedk}
  Let $(\varphi_j)_{j \in \N}$ be as in Lemma~\ref{lm:ex_eigen_basis}
  and set $X_k \defs \spanset\{\,\varphi_j: 1 \le j \le k\,\}$ for $k \in \N$.
  For $\eps, \delta \in (0, 1)$ and $k \in \N$, there exist $\tmaxedk \in (0, \infty]$
  and functions
  \begin{align}\label{eq:ex_uedk:smooth}
    \uedk,\vedk \in C^\infty(\Ombar \times [0, \tmaxedk))
  \end{align}
  with
  \begin{align}\label{eq:ex_uedk:boundary}
      \partial_\nu \uedk = \partial_\nu \Delta \uedk
    = \partial_\nu \vedk = \partial_\nu \Delta \vedk
    = 0
  \end{align}
  fulfilling
  \begin{align}\label{eq:ex_uedk:u_eq}
          \ddt \intom \uedk \psi
    &=    \eps \intom \Sod(\uedk) \nabla \Delta \uedk \cdot \nabla \psi
          - \intom D_1(|\uedk|) \nabla \uedk \cdot \nabla \psi \notag \\
    &\pe  + \intom \Sod(\uedk) \nabla \vedk \cdot \nabla \psi
          + \intom \fod(\uedk, \vedk) \psi
  \intertext{and}\label{eq:ex_uedk:v_eq}
          \ddt \intom \vedk \psi
    &=    \eps \intom \Std(\vedk) \nabla \Delta \vedk \cdot \nabla \psi
          - \intom D_2(|\vedk|) \nabla \vedk \cdot \nabla \psi \notag \\
    &\pe  - \intom \Std(\vedk) \nabla \uedk \cdot \nabla \psi
          + \intom \ftd(\uedk, \vedk) \psi
  \end{align}
  in $(0, \tmaxedk)$ for all $\psi \in X_k$
  as well as
  \begin{align}\label{eq:ex_uedk:init}
    \intom \uedk(\cdot, 0) \psi = \intom u_{0} \psi
    \quad \text{and} \quad
    \intom \vedk(\cdot, 0) \psi = \intom v_{0} \psi
    \qquad \text{for all $\psi \in X_k$.}
  \end{align}
  
  Additionally, if $\tmaxedk \lt \infty$, then
  \begin{align}\label{eq:ex_uedk:ex_crit}
    \limsup_{t \nea \tmaxedk} \left( \|\uedk(\cdot, t)\|_{L^2(\Omega)} + \|\vedk(\cdot, t)\|_{L^2(\Omega)} \right) = \infty.
  \end{align}
\end{lemma}
\begin{proof}
  We fix $\eps, \delta \in (0, 1)$ and $k \in \N$.
  \newcommand{\usum}{\left({\textstyle \sum_{j=1}^k} w_{j} \varphi_{j} \right)}%
  \newcommand{\usumabs}{\left(\left|{\textstyle \sum_{j=1}^k} w_{j} \varphi_{j} \right|\right)}%
  \newcommand{\usumbless}{{\textstyle \sum_{j=1}^k} w_{j} \varphi_{j}}%
  \newcommand{\usumblessabs}{\left|{\textstyle \sum_{j=1}^k} w_{j} \varphi_{j} \right|}%
  \newcommand{\vsum}{\left({\textstyle \sum_{j=1}^k} z_{j} \varphi_{j} \right)}%
  \newcommand{\vsumabs}{\left(\left|{\textstyle \sum_{j=1}^k} z_{j} \varphi_{j} \right|\right)}%
  \newcommand{\vsumbless}{{\textstyle \sum_{j=1}^k} z_{j} \varphi_{j}}%
  \newcommand{\vsumblessabs}{\left|{\textstyle \sum_{j=1}^k} z_{j} \varphi_{j} \right|}%
  For $w, z \in \R^k$, we define $F_1(w, z), F_2(w, z) \in \R^k$ by
  \begin{align*}
            (F_1(w, z))_i
    &\defs  \eps \intom \Sod \usum \nabla \Delta \usum \cdot \nabla \varphi_i
            - \intom D_1 \usumabs \nabla \usum \cdot \nabla \varphi_i \\
    &\pe    + \intom \Sod \usum \nabla \vsum \cdot \nabla \varphi_i
            + \intom \fod\left(\usumbless, \vsumbless \right) \varphi_i \\
    \intertext{and}
            (F_2(w, z))_i
    &\defs  \eps \intom \Std \vsum \nabla \Delta \vsum \cdot \nabla \varphi_i
            - \intom D_2 \vsumabs \nabla \vsum \cdot \nabla \varphi_i \\
    &\pe    - \intom \Std \vsum \nabla \usum \cdot \nabla \varphi_i
            + \intom \ftd\left(\usumbless, \vsumbless \right) \varphi_i
  \end{align*}
  for $i \in \{1, \dots, k\}$.
  
  As $F_1$ and $F_2$ are locally Lipschitz continuous, the Picard--Lindelöf theorem asserts the existence of
  $\tmaxedk \in (0, \infty]$ and $w, z \in C^0([0, \tmaxedk); \R^k) \cap C^1((0, \tmaxedk); \R^k)$ which solve
  \begin{align*}
    \begin{cases}
      w' = F_1(w, z) & \quad \text{in $(0, \tmaxedk)$}, \\
      z' = F_2(w, z) & \quad \text{in $(0, \tmaxedk)$}, \\
      w(0) = \intom u_{0\eps} \varphi, \\
      z(0) = \intom v_{0\eps} \varphi
    \end{cases}
  \end{align*}
  classically and, if $\tmaxedk \lt \infty$, then
  \begin{align}\label{eq:ex_uedk:ex_crit_w}
    \limsup_{t \nea \tmaxedk} \left( |w(t)| + |z(t)| \right) = \infty.
  \end{align}

  According to Lemma~\ref{lm:ex_eigen_basis}, the functions
  \begin{align*}
    \uedk(x, t) \defs \sum_{j=1}^k w_j(t) \varphi_j(x)
    \quad \text{and} \quad
    \vedk(x, t) \defs \sum_{j=1}^k z_j(t) \varphi_j(x),
    \qquad x \in \Ombar, t \in [0, \tmaxedk),
  \end{align*}
  satisfy \eqref{eq:ex_uedk:smooth} and \eqref{eq:ex_uedk:boundary}.
  Moreover, they fulfill
  \begin{align*}
      \ddt \intom \uedk \varphi_i
    = \ddt \intom \usum \varphi_i
    = \sum_{j=1}^k w_j' \intom \varphi_i \varphi_j
    = w_i'
    = (F_1(w, z))_i
    \qquad \text{in $(0, \tmaxedk)$}
  \end{align*}
  for $i \in \{1, \dots, k\}$.
  Thus, \eqref{eq:ex_uedk:u_eq} is fulfilled for $\psi = \varphi_i$ for all $i \in \{1, \dots, k\}$ and, due to linearity, 
  also for all $\psi \in X_k$, as desired.
  Likewise, we obtain that \eqref{eq:ex_uedk:v_eq} is also fulfilled for all $\psi \in X_k$.

  From $\intom \varphi_i \varphi_j = \delta_{ij}$ for $i, j \in \{1, \dots, k\}$,
  we further infer
  \begin{align*}
      \sum_{j=0}^k w_{\eps \delta k j}^2
    = \sum_{j=0}^k \intom w_{\eps \delta k j}^2 \varphi_j^2
    = \intom \left(\sum_{j=0}^k w_{\eps \delta k j} \varphi_j \right)^2
    = \intom \uedk^2
    \quad \text{in $(0, \tmaxedk)$}
  \end{align*}
  and, likewise,
  \begin{align*}
      \sum_{j=0}^k z_{\eps \delta k j}^2
    = \intom \uedk^2
    \quad \text{in $(0, \tmaxedk)$}.
  \end{align*}
  Thus, if \eqref{eq:ex_uedk:ex_crit} is not fulfilled, then  \eqref{eq:ex_uedk:ex_crit_w} is also not satisfied,
  implying $\tmaxedk = \infty$.
\end{proof}

In the following lemma, we show that the solutions $(\uedk, \vedk)$ constructed in Lemma~\ref{lm:ex_uedk} are global in time.
Moreover, in order to prepare the application of certain compactness theorems,
we also collect several $k$-independent a priori estimates.

As opposed to \cite{GrunDegenerateParabolicDifferential1995}, however, these bounds may depend on $\delta$,
the reason being that in our situation the terms stemming from the possibly nonlinear diffusion terms $D_1$ and $D_2$
can no longer be controlled independently of $\delta$,
at least not in all situations covered by Theorem~\ref{th:weak_sol_nonlin}.
This problem will then be circumvented by deriving appropriate $\delta$-independent estimates in Lemma~\ref{lm:est_g} below,
which are, however, weaker than those obtained in the present subsection.
\begin{lemma}\label{lm:dk_est}
  For all $\eps, \delta \in (0, 1)$ and $k \in \N$, let $(\ued, \ved)$ and $\tmaxedk$ be as given by Lemma~\ref{lm:ex_uedk}.
  Then $\tmaxedk = \infty$ for all $\eps, \delta \in (0, 1)$ and $k \in \N$
  and, moreover, for all $\eps, \delta \in (0, 1)$ and all $T \in (0, \infty)$, there exists $C \gt 0$ such that 
  for all $k \in \N$, the estimates
  \begin{align}
    \label{eq:dk_est:l2}
      &\sup_{t \in (0, T)} \intom \uedk^2(\cdot, t) + \sup_{t \in (0, T)} \intom \vedk^2(\cdot, t) \le C, \\
    \label{eq:dk_est:w12}
      &\sup_{t \in (0, T)} \intom |\nabla \uedk(\cdot, t)|^2 + \sup_{t \in (0, T)} \intom |\nabla \vedk(\cdot, t)|^2 \le C \quad \text{and}\\
    \label{eq:dk_est:l2_w32_m}
      &\intntom |\nabla \Delta \uedk|^2 + \intntom |\nabla \Delta \vedk|^2 \le C
  \end{align}
  hold.
\end{lemma}
\begin{proof}
  According to the Poincar\'e inequality (cf.\ \cite[Lemma~A.1]{FuestGlobalSolutionsHomogeneous2020}),
  there is $\cp \gt 0$ such that
  \begin{align}\label{eq:dk_est:poincare}
    \intom |\Delta \psi|^2 \le \cp \intom |\nabla \Delta \psi|^2
    \qquad \text{for all $\psi \in \sobn32$}.
  \end{align}
  We then fix $\eps, \delta \in (0, 1)$, take $\uedk$ as test function in \eqref{eq:ex_uedk:u_eq} and apply Young's inequality to obtain
  \begin{align*}
    &\pe  \frac12 \ddt \intom \uedk^2 \\
    &=    \eps \intom \Sod(\uedk) \nabla \Delta \uedk \cdot \nabla \uedk
          - \intom D_1(|\uedk|) |\nabla \uedk|^2 \\
    &\pe  + \intom \Sod(\uedk) \nabla \uedk \cdot \nabla \vedk
          + \intom \fod(\uedk, \vedk) \uedk \\
    &\le  \frac{\eps}{4} \intom \Sod(\uedk) |\nabla \Delta \uedk|^2
          + \left(\eps \ol S_1 - \ul D_1 + \frac{\ol S_1}{2} \right) \intom |\nabla \uedk|^2 \\
    &\pe  + \frac{\ol S_1}{2} \intom |\nabla \vedk|^2
          + \frac12 \intom \uedk^2
          + \frac{|\Omega| \|\fo\|_{L^\infty([0, \infty)^2)}^2}{2}
  \end{align*}
  in $(0, \tmaxedk)$ for all $k \in \N$.
  Moreover, as the Laplacian leaves the space $X_k$ defined in Lemma~\ref{lm:ex_uedk} invariant,
  we may also use $-\Delta \uedk \in X_k$ as a test function in \eqref{eq:ex_uedk:u_eq},
  which when combined with Young's inequality, \eqref{eq:def_fid}, \eqref{eq:dk_est:poincare} and \eqref{eq:def_sid} gives
  \begin{align*}
          \frac12 \ddt \intom |\nabla \uedk|^2
    &=    - \eps \intom \Sod(\uedk) |\nabla \Delta \uedk|^2
          + \intom D_1(|\uedk|) \nabla \Delta \uedk \cdot \nabla \uedk \\
    &\pe  - \intom \Sod(\uedk) \nabla \Delta \uedk \cdot \nabla \vedk
          - \intom \fod(\uedk, \vedk) \Delta \uedk \\
    &\le  - \frac{3\eps}{4} \intom \Sod(\uedk) |\nabla \Delta \uedk|^2
          + \frac{\eps \delta}{8} \intom |\nabla \Delta \uedk|^2
          + \frac{\eps \delta}{8 \cp} \intom |\Delta \uedk|^2 \\
    &\pe  + \frac{2 \ol D_1^2}{\eps \delta} \intom |\nabla \uedk|^2
          + \frac{\ol S_1}{\eps} \intom |\nabla \vedk|^2
          + \frac{2\cp|\Omega|}{\eps \delta} \|\fo\|_{L^\infty([0, \infty)^2)}^2 \\
    &\le  - \frac{\eps}{4} \intom \Sod(\uedk) |\nabla \Delta \uedk|^2
          - \frac{\eps \delta}{4} \intom |\nabla \Delta \uedk|^2 \\
    &\pe  + \frac{2 \ol D_1^2}{\eps \delta} \intom |\nabla \uedk|^2
          + \frac{\ol S_1}{\eps} \intom |\nabla \vedk|^2
          + \frac{2\cp|\Omega|}{\eps \delta} \|\fo\|_{L^\infty([0, \infty)^2)}^2
  \end{align*}
  in $(0, \tmaxedk)$ for all $k \in \N$.

  Along with analogous computations for the second equation,
  we see that there are $c_1, c_2 \gt 0$ such that for all $k \in \N$, the function
  \begin{align*}
    y(t) \defs \frac12 \intom \uedk^2 + \frac12 \intom |\nabla \uedk|^2 + \frac12 \intom \vedk^2 + \frac12 \intom |\nabla \vedk|^2,
    \qquad t \in [0, \tmaxedk),
  \end{align*}
  solves the ODI
  \begin{align*}
          y'(t)
    \le - c_1 \intom |\nabla \Delta \uedk|^2
        - c_1 \intom |\nabla \Delta \vedk|^2
        + c_2 y
        + c_2
    \qquad \text{in $(0, \tmaxedk)$}.
  \end{align*}
  
  According to Grönwall's inequality
  and as $y(0)$ is finite and bounded independently of $k$ by \eqref{eq:weak_sol_nonlin:initial},
  the estimates \eqref{eq:dk_est:l2}--\eqref{eq:dk_est:l2_w32_m} are then valid for all finite $T \in (0, \tmaxedk]$
  and certain $C \gt 0$ (depending on $\eps, \delta$ and $T$ but not on $k$).
  Due to the extensibility criterion \eqref{eq:ex_uedk:ex_crit},
  this then implies $\tmaxedk = \infty$ for all $k \in \N$
  and then that \eqref{eq:dk_est:l2}--\eqref{eq:dk_est:l2_w32_m} indeed hold for all $T \in (0, \infty)$
  (and corresponding $C \gt 0$).
\end{proof}

Having an application of the Aubin--Lions lemma in mind,
we next collect a priori estimates for the time derivatives.
\begin{lemma}\label{lm:uedk_time_est}
  For $\eps, \delta \in (0, 1)$ and $k \in \N$, we denote the solution given by Lemma~\ref{lm:ex_uedk} by $(\uedk, \vedk)$.
  For all $\eps, \delta \in (0, 1)$ and $T \in (0, \infty)$,
  there exist $C_1, C_2 \gt 0$ such that
  \begin{align}\label{eq:uedk_time_est:statement_1}
      \|\uedkt\|_{L^2((0, T); \dual{\sob12})}
      + \|\vedkt\|_{L^2((0, T); \dual{\sob12})}
    &\le C_1
  \intertext{and}\label{eq:uedk_time_est:statement_2}
      \|\nabla \uedkt\|_{L^2((0, T); \dual{\sobn22})}
      + \|\nabla \vedkt\|_{L^2((0, T); \dual{\sobn22})}
    &\le C_2
  \end{align}
  for all $k \in \N$.
\end{lemma}
\begin{proof}
  Let $\eps, \delta \in (0, 1)$ and $T \in (0, \infty)$.
  Letting $X_k$ be as in Lemma~\ref{lm:ex_uedk},
  we denote the orthogonal projection from $\sob12$ onto $X_k$ by $P_k$.
  Applying Lemma~\ref{lm:ex_uedk} and Hölder's inequality shows that
  \begin{align*}
    &\pe  \left| \intom \uedkt \varphi \right|
     =    \left| \intom \uedkt P_k \varphi \right| \\
    &\le  \eps \left| \intom \Sod(\uedk) \nabla \Delta \uedk \cdot \nabla P_k \varphi \right|
          + \left| \intom D_1(|\uedk|) \nabla \uedk \cdot \nabla P_k \varphi \right| \\
    &\pe  + \left| \intom \Sod(\uedk) \nabla \vedk \cdot \nabla P_k \varphi \right|
          + \left| \intom \fod(\ued, \ved) P_k \varphi \right| \\
    &\le  \left(
            \eps \ol S_1 \|\nabla \Delta \uedk\|_{\leb2}
            + \ol D_1 \|\nabla \uedk\|_{\leb2}
            + \ol S_1 \|\nabla \vedk\|_{\leb2}
            + ( \|\fo\|_{L^\infty([0, \infty)^2)} ) |\Omega|^\frac12
          \right) \|P_k \varphi\|_{\sob12}
  \end{align*}
  for all $\varphi \in \sob12$ and $k \in \N$.
  Upon integrating this inequality over $(0, T)$ and in conjunction with an analogous argument for $\vedkt$,
  we then infer \eqref{eq:uedk_time_est:statement_1} 
  from \eqref{eq:dk_est:l2_w32_m}, \eqref{eq:dk_est:w12} and \eqref{eq:weak_sol_nonlin:spaces_f}.

  Since for all $\varphi \in \sobn[{\Omega; \R^n}]22$ and $k \in \N$, we have
  \begin{align*}
          \left| \intom \nabla \uedkt \cdot \varphi \right|
    &=    \left| \intom \uedkt \nabla \cdot \varphi \right|
     \le  \|\uedkt\|_{\dual{\sob12}} \|\nabla \cdot \varphi\|_{\sob12}
     \le  \|\uedkt\|_{\dual{\sob12}} \|\varphi\|_{\sob[{\Omega; \R^n}]22}
  \end{align*}
  (and likewise for $\nabla \vedkt$),
  a consequence thereof is \eqref{eq:uedk_time_est:statement_2}.
\end{proof}

The bounds obtained above now allow us to obtain convergences of $\uedk$ and $\vedk$ along certain subsequences of $(k)_{k \in \N}$.
\begin{lemma}\label{lm:k_to_infty}
  For all $\eps, \delta \in (0, 1)$,
  there exist a subsequence $(k_j)_{j \in \N}$ of $(k)_{k \in \N}$ and functions
  \begin{align*}
    \ued, \ved \in W_{\loc}^{1, 2}([0, \infty); \sobn22, \sob12)
    \cap L_{\loc}^2([0, \infty); \sob32)
    \cap C^0([0, \infty); \sob12)
  \end{align*}
  such that
  \begin{alignat}{4}
    \uedkj &\ra \ued &\quad &\text{and} \quad & \vedkj &\ra \ved
      &&\qquad \text{pointwise a.e.}, \label{eq:k_to_infty:pw} \\
    \uedkj &\ra \ued &\quad &\text{and} \quad & \vedkj &\ra \ved
      &&\qquad \text{in $C^0([0, \infty); \leb2)$}, \label{eq:k_to_infty:l2} \\
    \nabla \uedkj &\ra \nabla \ued &\quad &\text{and} \quad & \nabla \vedkj &\ra \nabla \ved
      &&\qquad \text{in $L_{\loc}^2(\Ombar \times [0, \infty); \R^n)$}, \label{eq:k_to_infty:nabla_l2} \\
    \nabla \Delta \uedkj &\rh \nabla \Delta \ued &\quad &\text{and} \quad & \nabla \Delta \vedkj &\rh \nabla \Delta \ved
      &&\qquad \text{in $L_{\loc}^2(\Ombar \times [0, \infty); \R^n)$}, \label{eq:k_to_infty:nabla_delta_l2} \\
    \uedkjt &\rh \uedt &\quad &\text{and} \quad & \vedkjt &\rh \vedt
      &&\qquad \text{in $L_{\loc}^2([0, \infty); \dual{\sob12})$} \label{eq:k_to_infty:time_l2}
  \end{alignat}
  as $j \ra \infty$.
\end{lemma}
\begin{proof}
  As the claims for the second solution component can be shown analogously,
  it suffices to prove \eqref{eq:k_to_infty:pw}--\eqref{eq:k_to_infty:time_l2} for the first one.
  According to \eqref{eq:dk_est:l2}--\eqref{eq:dk_est:l2_w32_m}, \eqref{eq:uedk_time_est:statement_1} and \eqref{eq:uedk_time_est:statement_2},
  the sequence $(\uedk)_{k \in \N}$ is bounded in the space $W_{\loc}^{1, 2}([0, \infty); \sob22, \sob12)$
  so that by a diagonalization argument,
  we obtain a sequence $(k_j)_{j \in \N} \subset \N$ with $k_j \ra \infty$
  and a function $\ued \in W_{\loc}^{1, 2}([0, \infty); \sob22, \sob12)$ such that
  \begin{align*}
    \uedkj \rh \ued \qquad \text{in $W_{\loc}^{1, 2}([0, \infty); \sobn22, \sob12)$ as $j \ra \infty$},
  \end{align*}
  which directly implies \eqref{eq:k_to_infty:nabla_delta_l2} and \eqref{eq:k_to_infty:time_l2}
  and together with the Aubin--Lions lemma also \eqref{eq:k_to_infty:nabla_l2}.

  Thanks to \eqref{eq:dk_est:w12} and \eqref{eq:uedk_time_est:statement_1},
  another application of the Aubin--Lions lemma yields \eqref{eq:k_to_infty:l2}
  and thus also \eqref{eq:k_to_infty:pw},
  possibly after switching to subsequences.
\end{proof}

We conclude this subsection by showing
that the pair $(\ued, \ved)$ constructed in Lemma~\ref{lm:k_to_infty} indeed solves \eqref{prob:ped} in a weak sense.
\begin{lemma}\label{lm:ued_ved_weak_sol}
  Let $\eps, \delta \in (0, 1)$.
  The tuple $(\ued, \ved)$ constructed in Lemma~\ref{lm:k_to_infty}
  is a weak solution of \eqref{prob:ped} in the sense that
  \begin{align}\label{eq:ued_ved_weak_sol:init}
    \ued(\cdot, 0) = u_0
    \quad \text{as well as} \quad
    \ved(\cdot, 0) = v_0
    \qquad \text{hold a.e.\ in $\Omega \times (0, \infty)$},
  \end{align}
  and, for all $T \in (0, \infty)$ and $\varphi \in L^2((0, T); \sob12)$, we have
  \begin{align}\label{eq:ued_ved_weak_sol:u_eq}
          \intntom \uedt \varphi
    &=    \eps \intntom \Sod(\ued) \nabla \Delta \ued \cdot \nabla \varphi
          - \intntom D_1(|\ued|) \nabla \ued \cdot \nabla \varphi \notag \\
    &\pe  + \intntom \Sod(\ued) \nabla \ved \cdot \nabla \varphi
          + \intntom \fod(\ued, \ved) \varphi
  \intertext{as well as}\label{eq:ued_ved_weak_sol:v_eq}
          \intntom \vedt \varphi
    &=    \eps \intntom \Std(\ved) \nabla \Delta \ved \cdot \nabla \varphi
          - \intntom D_2(|\ved|) \nabla \ved \cdot \nabla \varphi \notag \\
    &\pe  - \intntom \Std(\ved) \nabla \ued \cdot \nabla \varphi
          + \intntom \ftd(\ued, \ved) \varphi.
  \end{align}
\end{lemma}
\begin{proof}
  We fix $T \in (0, \infty)$ as well as $\varphi \in L^2((0, T); \sob12)$,
  denote the orthogonal projection on $X_k$ by $P_k$ (where $X_k$ is as in Lemma~\ref{lm:ex_uedk})
  and set $(P_k \varphi)(x, t) \defs (P_k \varphi(\cdot, t))(x)$ for $(x, t) \in \Omega \times (0, T)$.
  Moreover, let $(\ued, \vedk)$ and $(k_j)_{j \in \N}$ be as given by Lemma~\ref{lm:k_to_infty}.
  According to Lemma~\ref{lm:ex_uedk} and Lemma~\ref{lm:dk_est}, we then have
  \begin{align*}
          \intntom \uedkt P_k \varphi
    &=    \eps \intntom \Sod(\uedk) \nabla \Delta \uedk \cdot \nabla P_k \varphi 
          - \intntom D_1(|\uedk|) \nabla \uedk \cdot \nabla P_k \varphi \\
    &\pe  + \intntom \Sod(\uedk) \nabla \vedk \cdot \nabla P_k \varphi
          + \intntom \fod(\uedk, \vedk) P_k \varphi
  \end{align*}
  for all $k \in \N$.
  Since $P_k \varphi \ra \varphi$ in $L^2((0, T); \sob12)$ for $k \ra \infty$,
  we infer
  \begin{align*}
    \lim_{j \ra \infty} \intntom \uedkjt P_{k_j} \varphi = \intntom \uedt \varphi
  \end{align*}
  from \eqref{eq:k_to_infty:time_l2}.
  Moreover, as $\fod$ is bounded, \eqref{eq:k_to_infty:pw}
  asserts $\fod(\uedkj, \vedkj) \ra \fod(\ued, \ved)$ in $L^2(\Omega \times (0, T))$ as $j \ra \infty$ and hence
  \begin{align*}
    \lim_{j \ra \infty} \intntom \fod(\uedkj, \vedkj) P_{k_j} \varphi = \intntom \fod(\ued, \ved) \varphi.
  \end{align*}
  Boundedness of $\Sod$, \eqref{eq:k_to_infty:pw} and Lebesgue's theorem 
  imply
  \begin{align*}
    &\pe  \|\Sod(\uedkj) \nabla P_{k_j} \varphi - \Sod(\uedkj) \nabla \varphi\|_{L^2(\Omega \times (0, T))} \\
    &\le  \|[\Sod(\uedkj) - \Sod(\ued)] \nabla \varphi\|_{L^2(\Omega \times (0, T))}
          + \|\Sod(\uedkj) \nabla [P_{k_j} \varphi - \varphi]\|_{L^2(\Omega \times (0, T))}
    \ra   0
  \end{align*}
  as $j \ra \infty$
  and hence
  \begin{align*}
      \lim_{j \ra \infty} \intntom \Sod(\uedkj) \nabla \Delta \uedkj \cdot \nabla P_{k_j} \varphi
    = \intntom \Sod(\ued) \nabla \Delta \ued \cdot \nabla \varphi
  \end{align*}
  due to \eqref{eq:k_to_infty:nabla_delta_l2}.
  A similar reasoning, relying on \eqref{eq:k_to_infty:nabla_l2} instead of \eqref{eq:k_to_infty:nabla_delta_l2}, gives
  \begin{align*}
        \lim_{j \ra \infty} \intntom \Sod(\uedkj) \nabla \uedkj \cdot \nabla P_{k_j} \varphi
    &=  \intntom \Sod(\ued) \nabla \ued \cdot \nabla \varphi
  \intertext{and}
        \lim_{j \ra \infty} \intntom D_1(|\uedkj|) \nabla \uedkj \cdot \nabla P_{k_j} \varphi
    &=  \intntom D_1(|\uedkj|) \nabla \ued \cdot \nabla \varphi
  \end{align*}
  so that indeed \eqref{eq:ued_ved_weak_sol:u_eq} holds,
  while \eqref{eq:ued_ved_weak_sol:v_eq} can be derived analogously.

  Finally, we note that \eqref{eq:k_to_infty:l2} implies $\uedkj(\cdot, 0) \ra \ued(\cdot, 0)$ in $\leb2$ as $j \ra \infty$
  so that \eqref{eq:ex_uedk:init} asserts
  \begin{align*}
      \intom \ued(\cdot, 0) \psi
    = \lim_{j \ra \infty} \intom \uedkj(\cdot, 0) P_{k_j} \psi
    = \intom u_0 \psi
    \qquad \text{for all $\psi \in \leb2$}.
  \end{align*}
  This implies $\ued(\cdot, 0) = u_0$ a.e.\
  and, by combining this with an analogous argument for the second solution component, we arrive at \eqref{eq:ued_ved_weak_sol:init}.
\end{proof}

\subsection{The limit process \tops{$\delta \sea 0$}{delta to 0}: guaranteeing nonnegativity}\label{sec:delta_sea_0}
As opposed to the problem solved by $(\uedk, \vedk)$ for $k \in \N$,
where \eqref{eq:ex_uedk:u_eq} and \eqref{eq:ex_uedk:v_eq} require that $\varphi(\cdot, t) \in X_k$ for all $t \in (0, \infty)$,
in the weak formulation for the problem \eqref{prob:ped}, \eqref{eq:ued_ved_weak_sol:u_eq} and \eqref{eq:ued_ved_weak_sol:v_eq},
all $\varphi \in L_{\loc}^2([0, \infty); \sob12)$ are admissible test functions.
In particular, we may now test with anti-derivatives of $\frac1{\Sod(\ued)}$ and $\frac1{\Std(\ved)}$,
allowing us to obtain estimates independent of both $\eps$ and $\delta$ in Lemma~\ref{lm:test_g}.
These bounds not only form the basis for the limit processes $\delta \sea 0$ and $\eps \sea 0$
(which are finally performed in Lemma~\ref{lm:delta_sea_0} and Lemma~\ref{lm:eps_sea_0}, respectively)
but are also important for showing that the later obtained limit functions $\ue, \ve$ are nonnegative (see Lemma~\ref{lm:ue_ge_0}).

To further prepare these testing procedures, we state the following lemma which should essentially be well-known.
\begin{lemma}\label{lm:test_weak}
  Let $T \in (0, \infty)$, $w, z \in \Wt$ and $\varphi \in C^1([0, T])$.

  \begin{enumerate}
    \item[(i)]
      For $H \in C^2(\R^2)$ with $D^2 H \in L^\infty(\R^2; \R^{2\times2})$,
      $H_w(w, z)$ and $H_z(w, z)$ belong to $L^2((0, T); \sob12)$ and
      \begin{align}\label{eq:test_weak:statement_i}
        &\pe  \intntom w_t H_w(w, z) \varphi + \intntom z_t H_z(w, z) \varphi \notag \\
        &=    - \intntom H(w, z) \varphi_t + \intom H(w(\cdot, T), z(\cdot, T)) \varphi(\cdot, T) - \intom H(w(\cdot, 0), z(\cdot, 0)) \varphi(\cdot, 0) 
      \end{align}
      holds.
      (We remark that $w(\cdot, 0)$, $w(\cdot, T)$, $z(\cdot, 0)$ and $z(\cdot, T)$ are (well-defined) elements of $\leb2$ since $\Wt \embed C^0([0, T]; \leb2)$;
      that is, all terms in \eqref{eq:test_weak:statement_i} are well-defined.)

    \item[(ii)]
      Let $\wt H \in C^2(\R)$ with $\wt H'' \in L^\infty(\R)$.
      Then $\wt H'(w) \in L^2((0, T); \sob12)$ and
      \begin{align*}
          \intntom w_t \wt H'(w) \varphi
        = - \intntom \wt H(w) \varphi_t
          + \intom \wt H(w(\cdot, T))
          - \intom \wt H(w(\cdot, 0)).
      \end{align*}
  \end{enumerate}
\end{lemma}
\begin{proof}
  We first fix $(w_\ell)_{\ell \in \N}, (z_\ell)_{\ell \in \N} \subset C^\infty(\Ombar \times [0, T])$
  with $w_\ell \ra w$ and $z_\ell \ra z$ in $\Wt$ as $j \ra \infty$.
  Hence, for $X \defs L^2((0, T); \sob12)$ and thus $\dualn X = L^2((0, T); \dual{\sob12})$, we have
  $w_\ell \ra w$ and $z_\ell \ra z$ in $X$,
  $w_{\ell t} \ra w_t$ and $z_{\ell t} \ra z_t$ in $\dualn X$ as well as
  $w_\ell \ra w$ and $z_\ell \ra z$ in $C^0([0, T]; \leb2)$.

  Then
  \begin{align*}
    &\pe  - \intntom H(w_\ell, z_\ell) \varphi_t + \left[ \intom H(w_\ell(\cdot, t), z_\ell(\cdot, t)) \varphi(\cdot, t) \right]_{t=0}^{t=T} \\
    &=    \intntom [H(w_\ell, z_\ell)]_t \varphi
    =     \intntom w_{\ell t} H_w(w_\ell, z_\ell) \varphi + \intntom z_{\ell t} H_z(w_\ell, z_\ell) \varphi.
  \end{align*}

  By Taylor's theorem for multivariate functions and Young's inequality, we obtain
  \begingroup \allowdisplaybreaks
  \begin{align*}
    &\pe  \left| \intom H(w(\cdot, t), z(\cdot, t)) - \intom H(w_\ell(\cdot, t), z_\ell(\cdot, t)) \right| \\
    &\le  \sum_{|\alpha|=1} 
            \intom \frac{|D^\alpha H(w(\cdot, t), z(\cdot, t))|}{\alpha!} [w(\cdot, t) - w_\ell(\cdot, t), z(\cdot, t) - z_\ell(\cdot, t)]^\alpha \\
    &\pe  + \sum_{|\alpha|=2} \frac{\max_{|\beta|=|\alpha|} \|D^\beta H\|_{L^\infty(\R^2)}}{\alpha!}
            \intom [w(\cdot, t) - w_\ell(\cdot, t), z(\cdot, t) - z_\ell(\cdot, t)]^\alpha \\
    &\le  |H_w(w(\cdot, t), z(\cdot, t))| \intom |w(\cdot, t) - w_\ell(\cdot, t)|
          + |H_z(w(\cdot, t), z(\cdot, t))| \intom |z(\cdot, t) - z_\ell(\cdot, t)| \\
    &\pe  + \|D^2 H\|_{L^\infty(\R^{2\times2})} \intom (w(\cdot, t) - z(\cdot, t))^2
          + \|D^2 H\|_{L^\infty(\R^{2\times2})} \intom (z(\cdot, t) - z_\ell(\cdot, t))^2 \\
    &\ra  0
    \qquad \text{as $\ell \ra \infty$ for all $t \in [0, T]$}.
  \end{align*}
  \endgroup

  Since moreover
  \begin{align}\label{eq:test_weak:l2_conv}
    &\pe  \|H_w(w, z) - H_w(w_\ell, z_\ell)\|_{L^2(\Omega \times (0, T))} \notag \\
    &\le  \|H_w(w, z_\ell) - H_w(w_\ell, z_\ell)\|_{L^2(\Omega \times (0, T))}
          + \|H_w(w, z) - H_w(w, z_\ell)\|_{L^2(\Omega \times (0, T))} \notag \\
    &\le  \|H_{ww}\|_{L^\infty(\R^2)} \|w - w_\ell\|_{L^2(\Omega \times (0, T))}
          + \|H_{wz}\|_{L^\infty(\R^2)} \|z - z_\ell\|_{L^2(\Omega \times (0, T))} \notag \\
    &\ra  0
    \qquad \text{as $\ell \ra \infty$}
  \end{align}
  by the mean value theorem and
  \begin{align*}
          \sup_{\ell \in \N} \intntom |\nabla H_w(w_\ell, z_\ell)|^2
    &=    \sup_{\ell \in \N} \intntom \left| H_{ww}(w_\ell, z_\ell) \nabla w_\ell + H_{wz}(w_\ell, z_\ell) \nabla z_\ell \right|^2 \\
    &\le  \sup_{\ell \in \N} \left(
            2 \|H_{ww}\|_{L^\infty(\R^2)}^2 \intntom |\nabla w_\ell|^2
            + 2 \|H_{wz}\|_{L^\infty(\R^2)}^2 \intntom |\nabla z_\ell|^2
          \right)
     \lt  \infty
  \end{align*}
  by the chain rule, we conclude $\sup_{\ell \in \N} \|H_w(w_\ell, z_\ell)\|_X^2 \lt \infty$.
  Therefore, after switching to subsequences if necessary, we have
  \begin{align}\label{eq:test_weak:hw_conv}
    H_w(w_\ell, z_\ell) \rh \tilde w \qquad \text{in $X$ as $\ell \ra \infty$}
  \end{align}
  for some $\tilde w \in X$.
  From \eqref{eq:test_weak:l2_conv}, we infer $\tilde w = H_w(w, z)$
  so that \eqref{eq:test_weak:hw_conv} and the convergence $w_{\ell t} \ra w_t$ in $\dualn X$ imply
  \begin{align*}
        \intntom w_{\ell t} H_w(w_\ell, z_\ell) \varphi
    \ra \intntom w_t H_w(w, z) \varphi
    \qquad \text{as $\ell \ra \infty$}.
  \end{align*}
  Likewise, we obtain
  \begin{align*}
        \intntom z_{\ell t} H_z(w_\ell, z_\ell) \varphi
    \ra \intntom z_t H_z(w, z) \varphi
    \qquad \text{as $\ell \ra \infty$}
  \end{align*}
  and thus \eqref{eq:test_weak:statement_i}.

  Finally, the second part follows from the first one by setting $H(\rho, \sigma) = \wt H(\rho)$ for $\rho, \sigma \in \R$.
\end{proof}

With Lemma~\ref{lm:test_weak} at hand, we are now able to prove an analogue to the entropy-like inequality \eqref{eq:approx_entropy:entropy}.
\begin{lemma}\label{lm:test_g}
  Let $\eps, \delta \in (0, 1)$ and $\ued, \ved$ be as in Lemma~\ref{lm:k_to_infty}.
  Set moreover
  \begin{align*}
    \Gid(s) \defs \int_1^s \int_1^\rho \frac{1}{\Sid(\sigma)} \dsigma \drho
    \qquad \text{for $i \in \{1, 2\}$}
  \end{align*}
  as well as
  \begin{align*}
            \Eed(t)
    &\defs  \intom \God(\ued(\cdot, t))
            + \intom \Gtd(\ved(\cdot, t)), \\
            \Ded(t)
    &\defs  \eps \intom |\Delta \ued(\cdot, t)|^2
            + \eps \intom |\Delta \ved(\cdot, t)|^2 \\
    &\pe    + \intom \frac{D_1(|\ued(\cdot, t)|)}{\Sod(\ued(\cdot, t))} |\nabla \ued(\cdot, t)|^2 
            + \intom \frac{D_2(|\ved(\cdot, t)|)}{\Std(\ved(\cdot, t))} |\nabla \ved(\cdot, t)|^2
    \quad \text{and} \\
            \Red(t)
    &\defs  \intom \God'(\ued(\cdot, t)) \fod(\ued(\cdot, t), \ved(\cdot, t))
            + \intom \Gtd'(\ved(\cdot, t)) \ftd(\ued(\cdot, t), \ved(\cdot, t))
  \end{align*}
  for $t \in [0, \infty)$.
  (Here, similarly as on Theorem~\ref{th:approx_entropy},
  $\Ded$ and $\Red$ are to be understood as functions in $L^0((0, \infty))$.)
  Then
  \begin{align}\label{eq:test_g:entropy}
          \Eed(T) \zeta(T)
          + \intnt \Ded(t) \zeta(t) \dt
    &\le  \Eed(0) \zeta(0)
          + \intnt \Red(t) \zeta(t) \dt
          + \intnt \Eed(t) \zeta'(t) \dt
  \end{align}
  holds for any $T \in (0, \infty)$ and $0 \le \zeta \in C^\infty([0, T])$.
\end{lemma}
\begin{proof}
  As $\frac1{\Sod}$ is continuous, positive and bounded,
  we may apply Lemma~\ref{lm:test_weak}~(ii) and Lemma~\ref{lm:ued_ved_weak_sol} to obtain
  \begin{align}\label{eq:test_g:test}
    &\pe  \intom \God(\ued(\cdot, T)) \zeta(T)
          - \intom \God(u_0) \zeta(0)
          - \intntom \God(\ued) \zeta' \notag \\
    &=    \intntom \uedt \God'(\ued) \zeta \notag \\
    &=    - \eps \intntom |\Delta \ued|^2 \zeta
          - \intntom \frac{D_1(|\ued|)}{\Sod(\ued)} |\nabla \ued|^2 \zeta \notag \\
    &\pe  + \intntom \nabla \ued \cdot \nabla \ved \zeta
          + \intntom \God'(\ued) \fod(\ued, \ved) \zeta
  \end{align}
  for all $\eps, \delta \in (0, 1)$.
  Since the signs of the cross-diffusive terms in the first two equations in \eqref{prob:ped} are opposite,
  \eqref{eq:test_g:test} and a corresponding identity for the second solution component already yield \eqref{eq:test_g:entropy}.
\end{proof}

Aiming to derive $(\eps, \delta)$-independent a~priori estimates from \eqref{eq:test_g:entropy} with $\zeta \equiv 1$,
we next estimate the right-hand side therein and obtain
\begin{lemma}\label{lm:est_g}
  Let $T \in (0, \infty)$ and $\Gid$, $\delta \in (0, 1)$ $i \in \{1, 2\}$ be as in Lemma~\ref{lm:test_g}.
  Then there is $C \gt 0$ such that
  \begin{align}
   \label{eq:est_g:g}
     &\sup_{t \in (0, T)} \left( \intom \God(\ued(\cdot, t)) + \intom \Gtd(\ved(\cdot, t)) \right) \le C, \\
   \label{eq:est_g:w22_eps}
     &\eps \intntom |\Delta \ued|^2 + \eps \intntom |\Delta \ved|^2 \le C, \\
   \label{eq:est_g:w12}
     &\intntom |\nabla \ued|^2 + \intntom |\nabla \ved|^2 \le C \quad \text{and} \\
   \label{eq:est_g:l2}
     &\intntom \ued^2 + \intntom \ved^2 \le C
  \end{align}
  for all $\eps, \delta \in (0, 1)$, where $\ued$ and $\ved$ are as in Lemma~\ref{lm:k_to_infty}.
\end{lemma}
\begin{proof}
  Since the definition of $\ul S_1$ entails that 
  \begin{align*}
        \Sod(s)
    \ge S_1(|s|) + \delta
    \ge \begin{cases}
          \ul S_1 s, & |s| \lt 1, \\
          \ul S_1, & |s| \ge 1
        \end{cases}
    \qquad \text{for all $s \ge0$ and $\delta \in (0, 1)$},
  \end{align*}
  we may estimate
  \begin{alignat*}{2}
          \left| \God'(\ued) \right|
    &=    \int_{\ued}^1 \frac{\dsigma}{\Sod(\sigma)}
    \le   \frac{|\ln \ued|}{\ul S_1}
    &&\qquad \text{in $\{0 \lt \ued \le 1\}$ for all $\eps, \delta \in (0, 1)$}
  \intertext{and}
          \left| \God'(\ued) \right|
    &=    \int_1^{\ued} \frac{\dsigma}{\Sod(\sigma)}
    \le   \frac{\ued-1}{\ul S_1}
    &&\qquad \text{in $\{1 \lt \ued\}$ for all $\eps, \delta \in (0, 1)$}.
  \end{alignat*}
  Due to $\fod(\ued, \ved) = 0$ in $\{\ued \le 0\}$ and because of \eqref{eq:weak_sol_nonlin:spaces_f} and \eqref{eq:weak_sol_nonlin:cond_f},
  we thus obtain $c_1 \gt 0$ such that
  \begin{align*}
          \intnstom \fod(\ued, \ved) \God'(\ued)
    \le   c_1 \intnstom (1 + \ued^2)
    \qquad \text{for all $t \in (0, T)$ and $\eps, \delta \in (0, 1)$}.
  \end{align*}
  
  Moreover, positivity of $u_0$ and $v_0$ implies finiteness of
  \begin{align*}
    \sup_{\delta \in (0, 1)} \left( \intom \God(u_0) + \intom \Gtd(v_0) \right).
  \end{align*}
  As $\frac{D_i(|s|)}{\Sid(s)} \ge \frac{\ul D_i}{\ol S_i}$, $i \in \{1, 2\}$, for all $s \in \R$,
  along with an analogous computation for the second solution component
  and choosing $\zeta \equiv 1$ in \eqref{eq:test_g:entropy},
  we obtain $c_2 \gt 0$ such that
  \begin{align}\label{eq:test_g:odi}
    &\pe  \intom \God(\ued(\cdot, t)) + \intom \Gtd(\ved(\cdot, t)) \notag \\
    &+    \eps \intnstom |\Delta \ued|^2 + \eps \intnstom |\Delta \ved|^2
     +    \frac{\ul D_1}{\ol S_1} \intnstom |\nabla \ued|^2 + \frac{\ul D_2}{\ol S_2} \intnstom |\nabla \ved|^2 \notag \\
    &\le  c_2 + c_2 \intnstom \ued^2 + c_2 \intnstom \ved^2
    \qquad \text{for all $t \in (0, T)$ and $\eps, \delta \in (0, 1)$}.
  \end{align}

  Since
  \begin{align*}
          \God(\ued)
    &=    \int_1^{\ued} \int_1^\rho \frac{1}{\Sod(\sigma)} \dsigma \drho 
     \ge  \frac1{\ol S_1} \int_1^{\ued} (\rho - 1) \drho
     =    \frac1{\ol S_1} \left( \frac12 \ued^2 - \frac12 - (\ued - 1) \right) 
     \ge  \frac1{\ol S_1} \left( \frac14 \ued^2 - \frac12 \right)
  \end{align*}
  in $\Omega \times (0, T)$ for all $\eps, \delta \in (0, 1)$
  and hence
  \begin{align}\label{eq:test_g:l2_g}
    \intnstom \ued^2 \le 4 \ol S_1 \intnstom \God(\ued) + 2 |\Omega| T
    \qquad \text{for all $t \in (0, T)$ and $\eps, \delta \in (0, 1)$},
  \end{align}
  a consequence of \eqref{eq:test_g:odi} is
  \begin{align*}
          \intom \God(\ued(\cdot, t)) + \intom \Gtd(\ved(\cdot, t))
    &\le  c_3 + 4 c_2 \max \{\ol S_1, \ol S_2\} \intnst \left( \intom \God(\ued) + \intom \Gtd(\ved) \right)
  \end{align*}
  for all $t \in (0, T)$, $\eps, \delta \in (0, 1)$ and $c_3 \defs c_2 + 4 c_2 |\Omega| T$.
  Grönwall's inequality thus asserts that
  \begin{align*}
          \intom \God(\ued(\cdot, t)) + \intom \God(\ved(\cdot, t))
    \le   c_3 \ure^{4 c_2 \max \{\ol S_1, \ol S_2\} T}
    \qquad \text{holds for all $t \in (0, T)$ and $\eps, \delta \in (0, 1)$},
  \end{align*}
  implying \eqref{eq:est_g:g}.
  Finally, \eqref{eq:est_g:w22_eps}--\eqref{eq:est_g:l2}
  follow from \eqref{eq:test_g:odi}, \eqref{eq:test_g:l2_g} and \eqref{eq:est_g:g}.
\end{proof}

Again seeking to apply the Aubin--Lions lemma,
we complement the bounds \eqref{eq:est_g:g}--\eqref{eq:est_g:w12} by estimates for the time derivatives
in the next two lemmata.
However, in contrast to Lemma~\ref{lm:uedk_time_est} and owing to the fourth-order regularization terms,
we have to settle for bounds in $L^2((0, T); \dual{\sob{n+1}2})$ instead of $L^2((0, T); \dual{\sob12})$.
\begin{lemma}\label{lm:ued_ved_nearly_weak_sol}
  For $T \in (0, \infty)$, there exists $C \gt 0$ such that
  \begin{align}\label{eq:ued_ved_nearly_weak_sol:statement_u}
    &\pe  \left| \intntom \uedt \varphi
            + \intntom D_1(|\ued|) \nabla \ued \cdot \nabla \varphi
            - \intntom S_1(\ued) \nabla \ved \cdot \nabla \varphi
            - \intntom \fo(\ued, \ved) \varphi
          \right| \notag \\
    &\le  C \eps^\frac12 \|\varphi\|_{L^2((0, T); \sob{n+1}2)}
  \intertext{and}\label{eq:ued_ved_nearly_weak_sol:statement_v}
    &\pe  \left| \intntom \vedt \varphi
            + \intntom D_2(|\ved|) \nabla \ved \cdot \nabla \varphi
            + \intntom S_2(\ved) \nabla \ued \cdot \nabla \varphi
            - \intntom \ft(\ued, \ved) \varphi
          \right| \notag \\
    &\le  C \eps^\frac12 \|\varphi\|_{L^2((0, T); \sob{n+1}2)}
  \end{align}  
  for all $\eps, \delta \in (0, 1)$ and  $\varphi \in L^2((0, T); \sob{n+1}2)$,
  where $\ued$ and $\ved$ are as in Lemma~\ref{lm:k_to_infty}.
\end{lemma}
\begin{proof}
  Since $n+1 \gt \frac{n}{2} + 1$,
  Sobolev's embedding theorem allows us to fix $c_1 \gt 0$ with
  \begin{align*}
    \|\nabla \varphi\|_{\leb\infty} \le c_1 \|\varphi\|_{\sob{n+1}{2}}
    \qquad \text{for all $\varphi \in \sob{n+1}{2}$}.
  \end{align*}
  Moreover, we fix $T \in (0, \infty)$ and choose $c_2 \gt 0$ such that \eqref{eq:est_g:w22_eps} and \eqref{eq:est_g:w12} hold (with $C$ replaced by $c_2^2$).
  Then
  \begin{align*}
    &\pe  \left| \intntom \Sod(\ued) \nabla \Delta \ued \cdot \nabla \varphi \right| \\
    &\le  \left| \intntom \Sod'(\ued) \Delta \ued \nabla \ued \cdot \nabla \varphi \right|
          + \left| \intntom \Sod(\ued) \Delta \ued \Delta \varphi \right| \\
    &\le  \|\Delta \ued\|_{L^2(\Omega \times (0, T))} \left(
            \ol S_1' \|\nabla \ued\|_{L^2(\Omega \times (0, T))} \|\nabla \varphi\|_{L^\infty(\Omega \times (0, T))}
            + \ol S_1 \|\Delta \varphi\|_{L^2(\Omega \times (0, T))}
          \right) \\
    &\le  \eps^{-\frac12} \cdot c_2 (c_1 c_2 \ol S_1' + \ol S_1) \|\varphi\|_{L^2((0, T); \sob{n+1}2)}
    \qquad \text{for all $\eps, \delta \in (0, 1)$}.
  \end{align*}
  Combined with \eqref{eq:ued_ved_weak_sol:u_eq}, this already implies \eqref{eq:ued_ved_nearly_weak_sol:statement_u},
  while \eqref{eq:ued_ved_nearly_weak_sol:statement_v} can be shown analogously.
\end{proof}

\begin{lemma}\label{lm:ued_time_est}
  Let $\eps \in (0, 1)$, $T \in (0, \infty)$ and $(\ued, \ved)$ be as in Lemma~\ref{lm:k_to_infty} for $\delta \in (0, 1)$.
  Then there exists $C \gt 0$ such that
  \begin{align}\label{eq:ued_time_est:statement}
    \|\uedt\|_{L^2((0, T); \dual{\sob{n+1}2})} + \|\vedt\|_{L^2((0, T); \dual{\sob{n+1}2})} \le C
    \qquad \text{for all $\delta \in (0, 1)$}.
  \end{align}
\end{lemma}
\begin{proof}
  This immediately follows from Lemma~\ref{lm:ued_ved_nearly_weak_sol} and the bounds provided by Lemma~\ref{lm:est_g}. 
\end{proof}

With the estimates above at hand,
we are now able to obtain convergence of certain subsequences of $(\ued, \ved)$.
\begin{lemma}\label{lm:delta_sea_0}
  Let $\eps \in (0, 1)$.
  For $\delta \in (0, 1)$, let $\ued, \ved$ be as given by Lemma~\ref{lm:k_to_infty}.
  There are functions
  \begin{align*}
    \ue, \ve \in L_{\loc}^2([0, \infty); \sob12) \quad \text{with} \quad \uet, \vet \in L_{\loc}^2([0, \infty); \dual{\sob{n+1}2})
  \end{align*}
  and a null sequence $(\delta_j)_{j \in \N} \subset (0, 1)$ along which
  \begin{alignat}{4}
    \uedj &\ra \ue &\quad &\text{and} \quad & \vedj &\ra \ve
      &&\qquad \text{pointwise a.e.}, \label{eq:delta_sea_0:pw} \\
    \uedj &\ra \ue &\quad &\text{and} \quad & \vedj &\ra \ve
      &&\qquad \text{in $L_{\loc}^2(\Ombar \times [0, \infty))$}, \label{eq:delta_sea_0:l2} \\
    \uej(\cdot, t) &\ra u(\cdot, t) &\quad &\text{and} \quad & \vej(\cdot, t) &\ra v(\cdot, t)
      &&\qquad \text{in $\leb2$ for a.e.\ $t \in (0, \infty)$}, \label{eq:delta_sea_0:l2_space} \\
    \nabla \uedj &\rh \nabla \ue &\quad &\text{and} \quad & \nabla \vedj &\rh \nabla \ve
      &&\qquad \text{in $L_{\loc}^2(\Ombar \times [0, \infty); \R^n)$}, \label{eq:delta_sea_0:nabla_l2} \\
    \uedjt &\rh \uet &\quad &\text{and} \quad & \vedjt &\rh \vet
      &&\qquad \text{in $L_{\loc}^2([0, \infty); \dual{\sob{n+1}2})$}, \label{eq:delta_sea_0:time_l2}
  \end{alignat}
  as $j \ra \infty$.
\end{lemma}
\begin{proof}
  Due to the bounds in \eqref{eq:est_g:w12}, \eqref{eq:est_g:l2} and \eqref{eq:ued_time_est:statement},
  by means of the Aubin--Lions lemma and a diagonalization argument,
  we can obtain a null sequence $(\delta_j)_{j \in \N} \subset (0, 1)$ and functions $\ue, \ve \colon \Omega \times (0, \infty) \ra \R$
  such that \eqref{eq:delta_sea_0:l2}, \eqref{eq:delta_sea_0:nabla_l2} and \eqref{eq:delta_sea_0:time_l2} hold.
  Upon switching to subsequences, if necessary, \eqref{eq:delta_sea_0:pw} and \eqref{eq:delta_sea_0:l2_space} follow then from \eqref{eq:delta_sea_0:l2}.
\end{proof}

As already alluded to, the main reason for introducing the parameter $\delta$ in \eqref{prob:ped}
is to be able to establish a.e.\ nonnegativity of the functions $\ue$ and $\ve$ constructed in Lemma~\ref{lm:delta_sea_0}.
This will inter alia assure that each component of the solution $(u, v)$ to \eqref{prob:nonlin} obtained in Subsection~\ref{sec:eps_sea_0} below
is nonnegative and hence may be interpreted as a population density.
\begin{lemma}\label{lm:ue_ge_0}
  For all $\eps \in (0, 1)$, $\ue \ge 0$ and $\ve \ge 0$ a.e.\ in $\Omega \times (0, \infty)$,
  where $\ue$ and $\ve$ are given by Lemma~\ref{lm:delta_sea_0}.
\end{lemma}
\begin{proof}
  This can be shown similarly as in \cite[pages 554--555]{GrunDegenerateParabolicDifferential1995}.
  However, since the solutions considered there fulfill regularity properties going beyond those stated in Lemma~\ref{lm:delta_sea_0},
  we give a (slightly different) proof here.

  Let us fix $\eps \in (0, 1)$ as well as $T \in (0, \infty)$
  and for the sake of contradiction assume that (a henceforth fixed representative of) $\ue$ is not nonnegative a.e.\ in $\Omega \times (0, T)$.
  Then $|\{\ue \lt 0\}| > 0$ and, by the sigma additivity of the Lebesgue measure,
  there is $\eta \gt 0$ such that $A \defs \{\, (x, t) \in \Omega \times (0, T) : \ue(x, t) \le -\eta\,\}$ has positive measure.
  
  For $\delta \in (0, 1)$, we now let $\ued$ and $\God$ be as in Lemma~\ref{lm:k_to_infty} and Lemma~\ref{lm:test_g}, respectively,
  and denote by $(\delta_j)_{j \in \N}$ the sequence given by Lemma~\ref{lm:delta_sea_0}.
  Thanks to \eqref{eq:delta_sea_0:pw} and Egorov's theorem,
  we then obtain a measurable $A' \subset A$ with $|A \setminus A'| \lt \frac{|A|}{2}$
  such that $\uedj \ra \ue$ uniformly in $A'$ as $j \ra \infty$;
  in particular, there is $j_0 \in \N$ with $\uedj(x, t) \le -\frac{\eta}{2}$ for all $(x, t) \in A'$ and $j \ge j_0$.

  Thanks to nonnegativity of $\Sod$,
  since $S_1(|s|) \le -\ol S_1' s$ for $s \le 0$ (due to the mean value theorem and as $S_1(0) = 0$ by \eqref{eq:weak_sol_nonlin:cond_d_s})
  and by Fatou's lemma
  (we note that $\lim_{\delta \sea 0} (- \ln \delta + \ln(-\rho + \delta)) = \infty$ for all $\rho \lt 0$),
  we then have
  \begin{align*}
          \liminf_{j \ra \infty} \intntom \Godj(\uedj)
    &\ge  \liminf_{j \ra \infty} \int_{A'} \int_1^{\uedj(x, t)} \int_1^\rho \frac1{\Sodj(\sigma)} \dsigma \drho \diff(x,t) \\
    &\ge  \liminf_{j \ra \infty} |A'| \int_{-\frac{\eta}{2}}^1 \int_{\rho}^1 \frac{1}{S_1(|\sigma|) + \delta_j} \dsigma \drho \\
    &\ge  \liminf_{j \ra \infty} \frac{|A'|}{\max\{\ol S'_1, 1\}} \int_{-\frac{\eta}{2}}^0 \int_{\rho}^0 \frac{1}{-\sigma + \delta_j} \dsigma \drho \\
    &=    \liminf_{j \ra \infty} \frac{|A'|}{\max\{\ol S'_1, 1\}} \int_{-\frac{\eta}{2}}^0 ( - \ln \delta_j + \ln(-\rho + \delta_j) ) \drho
     =     \infty,
  \end{align*}
  contradicting \eqref{eq:est_g:g}.
  The same argument is also applicable for the second solution component.
\end{proof}

Let us close this subsection by discussing in which way the pair $(\ue, \ve)$ obtained in Lemma~\ref{lm:delta_sea_0}
can be seen as a solution to the problem obtained by formally setting $\delta=0$ in \eqref{prob:ped}.
Within a similar context, in \cite[pages~552--553]{GrunDegenerateParabolicDifferential1995}
it is shown that the limit functions solve the corresponding problem in a certain generalized sense.
However, as already remarked in the preceding subsection, due to the possibly nonlinear diffusion terms $D_1$ and $D_2$,
the convergences obtained in Lemma~\ref{lm:delta_sea_0} are slightly weaker than those established in \cite{GrunDegenerateParabolicDifferential1995};
that is, the methods developed in \cite{GrunDegenerateParabolicDifferential1995} are not directly applicable to our situation.

Nonetheless, we are able to prove that $(\ue, \ve)$ is up to an error term of order $\eps^\frac12$ a \emph{weak} solution of that problem,
which, having the limit process $\eps \sea 0$ in mind, turns out to be more convenient for our purposes in any case.
\begin{lemma}\label{lm:ue_ve_nearly_weak_sol}
  Let $\eps \in (0, 1)$, $\ue, \ve$ be as in Lemma~\ref{lm:delta_sea_0} and $T \in (0, \infty)$.
  Then there is $C \gt 0$ such that
  \begin{align}\label{eq:ue_ve_nearly_weak_sol:statement_u}
    &\pe  \left|
            - \intntom \ue \varphi_t
            - \intom u_0 \varphi(\cdot, 0)
            + \intntom D_1(\ue) \nabla \ue \cdot \nabla \varphi
            - \intntom S_1(\ue) \nabla \ve \cdot \nabla \varphi
            - \intntom \fo(\ue, \ve) \varphi
          \right| \notag \\
    &\le  C \eps^\frac12 \|\varphi\|_{L^2((0, T); \sob{n+1}2)}
  \intertext{and}\label{eq:ue_ve_nearly_weak_sol:statement_v}
    &\pe  \left|
            - \intntom \ve \varphi_t
            - \intom v_0 \varphi(\cdot, 0)
            + \intntom D_2(\ve) \nabla \ve \cdot \nabla \varphi
            + \intntom S_2(\ve) \nabla \ue \cdot \nabla \varphi
            - \intntom \ft(\ue, \ve) \varphi
          \right| \notag \\
    &\le  C \eps^\frac12 \|\varphi\|_{L^2((0, T); \sob{n+1}2)}
  \end{align}  
  for all $\varphi \in C_c^\infty(\Ombar \times [0, \infty))$.
\end{lemma}
\begin{proof}
  For $\delta \in (0, 1)$, we let $(\ued, \ved)$ be as in Lemma~\ref{lm:k_to_infty}
  and we denote the null sequence given by Lemma~\ref{lm:delta_sea_0} by $(\delta_j)_{j \in \N}$.
  Since $D_1$, $S_1$ and $f_1$ are continuous by \eqref{eq:weak_sol_nonlin:spaces_d}--\eqref{eq:weak_sol_nonlin:spaces_f},
  the convergences \eqref{eq:delta_sea_0:l2}, \eqref{eq:delta_sea_0:nabla_l2} and \eqref{eq:delta_sea_0:pw} imply that
  \begin{align*}
    &\pe  \left|
            - \intntom \ue \varphi_t
            - \intom u_0 \varphi(\cdot, 0)
            + \intntom D_1(\ue) \nabla \ue \cdot \nabla \varphi
            - \intntom S_1(\ue) \nabla \ve \cdot \nabla \varphi
            - \intntom \fo(\ue, \ve) \varphi
          \right| \\
    &=    \lim_{j \ra \infty} \left|
            - \intntom \uedj \varphi_t
            - \intom u_0 \varphi(\cdot, 0)
            + \intntom D_1(|\uedj|) \nabla \uedj \cdot \nabla \varphi \right. \\
    &\pe    - \left. \intntom \Sodj(\uedj) \nabla \vedj \cdot \nabla \varphi
            - \intntom \fodj(\uedj, \vedj) \varphi
          \right|
  \end{align*}  
  for all $\varphi \in C_c^\infty(\Ombar \times [0, \infty))$.
  As Lemma~\ref{lm:test_weak}~(ii) and \eqref{eq:ued_ved_weak_sol:init} assert
  \begin{align*}
      - \intntom \uedj \varphi_t
      - \intom u_0 \varphi(\cdot, 0)
    = \intntom \uedjt \varphi
    \qquad \text{for all $\varphi \in C_c^\infty(\Ombar \times [0, \infty))$ and $j \in \N$},
  \end{align*}
  we see that \eqref{eq:ue_ve_nearly_weak_sol:statement_u} (with $C$ as in Lemma~\ref{lm:ued_ved_nearly_weak_sol})
  follows from \eqref{eq:ued_ved_nearly_weak_sol:statement_u}.
  An analogous argumentation yields \eqref{eq:ue_ve_nearly_weak_sol:statement_v}.
\end{proof}

\subsection{The limit process \tops{$\eps \sea 0$}{epsilon to 0}: proofs of Theorem~\ref{th:weak_sol_nonlin} and Theorem~\ref{th:approx_entropy}}\label{sec:eps_sea_0}
Since Lemma~\ref{lm:est_g} and Lemma~\ref{lm:ued_time_est} already contain $\eps$-independent estimates,
there are no further preparations necessary in order to undertake the final limit process of this section, namely $\eps \sea 0$.
\begin{lemma}\label{lm:eps_sea_0}
  Let $\ue, \ve$ be as in Lemma~\ref{lm:delta_sea_0}.
  There are nonnegative functions $u, v \in L_{\loc}^2([0, \infty); \sob12)$ and a null sequence $(\eps_j)_{j \in \N} \subset (0, 1)$ such that
  \begin{alignat}{4}
    \uej &\ra u &\quad &\text{and} \quad & \vej &\ra v
      &&\qquad \text{pointwise a.e.}, \label{eq:eps_sea_0:pw} \\
    \uej &\ra u &\quad &\text{and} \quad & \vej &\ra v
      &&\qquad \text{in $L_{\loc}^2(\Ombar \times [0, \infty))$}, \label{eq:eps_sea_0:l2} \\
    \uej(\cdot, t) &\ra u(\cdot, t) &\quad &\text{and} \quad & \vej(\cdot, t) &\ra v(\cdot, t)
      &&\qquad \text{in $\leb2$ for a.e.\ $t \in (0, \infty)$}, \label{eq:eps_sea_0:l2_space} \\
    \nabla \uej &\rh \nabla u &\quad &\text{and} \quad & \nabla \vej &\rh \nabla v
      &&\qquad \text{in $L_{\loc}^2(\Ombar \times [0, \infty); \R^n)$} \label{eq:eps_sea_0:nabla_l2}
  \end{alignat}
  as $j \ra \infty$.
\end{lemma}
\begin{proof}
  As the estimates \eqref{eq:est_g:w12} and \eqref{eq:est_g:l2} do not depend on $\eps$,
  the right-hand sides in \eqref{eq:ued_ved_nearly_weak_sol:statement_u} and \eqref{eq:ued_ved_nearly_weak_sol:statement_v} are bounded in $\eps$
  and Lemma~\ref{lm:delta_sea_0} assures that these bounds also survive the limit $\delta \sea 0$ (along a certain null sequence),
  the existence of $u, v \in L_{\loc}^2([0, \infty); \sob12)$ and a null sequence $(\eps_j)_{j \in \N} \subset (0, 1)$ such that
  \eqref{eq:eps_sea_0:pw}--\eqref{eq:eps_sea_0:nabla_l2} hold can be shown as in Lemma~\ref{lm:delta_sea_0}.
  Moreover, nonnegativity of $u$ and $v$ follow from Lemma~\ref{lm:ue_ge_0} and \eqref{eq:eps_sea_0:pw}.
\end{proof}

Next, we show that the convergences asserted by Lemma~\ref{lm:eps_sea_0} are sufficiently strong
to imply that the pair $(u, v)$ constructed in that lemma at least solves \eqref{prob:nonlin} in the following sense,
which is yet somewhat weaker than the solution concept imposed by Theorem~\ref{th:weak_sol_nonlin}.
\begin{lemma}\label{lm:u_v_weak_sol_nonlin}
  The pair $(u, v)$ constructed in Lemma~\ref{lm:eps_sea_0} fulfills
  \begin{align}\label{eq:u_v_weak_sol_nonlin:u_sol}
        - \intninfom u \varphi_t - \intom u_0 \varphi(\cdot, 0)
    &=  - \intninfom D_1(u) \nabla u \cdot \nabla \varphi
        + \intninfom S_1(u) \nabla v \cdot \nabla \varphi
        + \intninfom \fo(u, v) \varphi \\
    \intertext{and}\label{eq:u_v_weak_sol_nonlin:v_sol}
        - \intninfom v \varphi_t - \intom v_0 \varphi(\cdot, 0)
    &=  - \intninfom D_2(u) \nabla v \cdot \nabla \varphi
        - \intninfom S_2(u) \nabla u \cdot \nabla \varphi
        + \intninfom \ft(u, v) \varphi
  \end{align}
  for all $\varphi \in C_c^\infty(\Ombar \times [0, \infty))$.
\end{lemma}
\begin{proof}
  Since $D_i$, $S_i$ and $f_i$, $i \in \{1, 2\}$, are bounded,
  the statement immediately follows 
  from Lemma~\ref{lm:ue_ve_nearly_weak_sol} and Lemma~\ref{lm:eps_sea_0}.
\end{proof}

In order to prove Theorem~\ref{th:weak_sol_nonlin}, in addition to Lemma~\ref{lm:u_v_weak_sol_nonlin},
we need to make sure that $u, v$ are sufficiently regular; that is, that they belong to $\Winf$.
To that end, the $\eps$-independent estimates of the time derivatives obtained in Lemma~\ref{lm:ued_time_est} are insufficient. 
However, we can obtain the desired regularity by testing directly at the $\eps=0$ level.
\begin{lemma}\label{lm:u_v_winf}
  The functions $u, v$ constructed in Lemma~\ref{lm:eps_sea_0} belong to $\Winf$
  and satisfy \eqref{eq:weak_sol_nonlin:init_cond} and \eqref{eq:weak_sol_nonlin:u_sol} as well as \eqref{eq:weak_sol_nonlin:v_sol} for all $\varphi \in L_{\loc}^2([0, \infty); \sob12)$.
\end{lemma}
\begin{proof}
  We fix $T \in (0, \infty)$.
  From Lemma~\ref{lm:u_v_weak_sol_nonlin} and Hölder's inequality, we infer that
  \begin{align*}
          \left| \intntom u_t \varphi \right|
    &=    \left| \intntom u \varphi_t \right| \\
    &\le  \left| \intntom D_1(u) \nabla u \cdot \nabla \varphi \right|
          + \left| \intntom S_1(u) \nabla v \cdot \nabla \varphi \right|
          + \left| \intntom \fo(u, v) \varphi \right| \\
    &\le  \left(
            \ol{D_1} \|\nabla u\|_{L^2(\Omega \times (0, T))}
            + \ol{S_1} \|\nabla v\|_{L^2(\Omega \times (0, T))}
            + \|\fo\|_{L^\infty([0, \infty)^2)} (|\Omega| T)^\frac12
          \right) \|\varphi\|_{L^2((0, T); \sob12)}
  \end{align*}
  for all $\varphi \in C_c^\infty(\Ombar \times (0, T))$,
  so that since $u, v \in L^2((0, T); \sob12)$ by Lemma~\ref{lm:eps_sea_0}
  and as $C_c^\infty(\Ombar \times (0, T))$ is dense in $L^2((0, T); \sob12)$,
  we can conclude $u_t \in \dual{L^2((0, T); \sob12} = L^2((0, T); \dual{\sob12})$.
  Thus, $u$, and by the same reasoning also $v$, belongs to $\Wt$.

  As therefore
  \begin{align}\label{eq:u_v_winf:int_part}
      \int_0^1 \intom u_t \varphi
    = - \int_0^1 \intom u \varphi_t - \intom u(\cdot, 0) \varphi(\cdot, 0)
    \qquad \text{for all $\varphi \in C_c^\infty(\Ombar \times [0, 1))$}
  \end{align}
  by Lemma~\ref{lm:test_weak}~(ii),
  we infer from \eqref{eq:u_v_weak_sol_nonlin:u_sol} and the regularity of $u$ and $v$ that there is $c_1 \gt 0$ such that
  \begin{align}\label{eq:u_v_winf:c1}
    &\pe \left| \intom (u(\cdot, 0) - u_0) \varphi(\cdot, 0) \right| \notag \\
    &\le \left(
          \|u_t\|_{L^2((0, 1); \dual{\sob12})}
          + \|D_1(u) \nabla u - S_1(u) \nabla v\|_{L^2(\Omega \times (0, 1))}
          + \|f_1(u, v)\|_{L^2(\Omega \times (0, 1))}
        \right) \|\varphi\|_{L^2((0, 1); \sob12)} \notag \\[0.3em]
    &\le c_1 \|\varphi\|_{L^2((0, 1); \sob12)}
    \qquad \text{for all $\varphi \in C_c^\infty(\Ombar \times [0, 1))$}.
  \end{align}
  We now fix $\zeta \in C^\infty(\R)$ with $0 \le \zeta \le 1$, $\zeta(s) = 1$ for $s \le 0$ and $\zeta(s) = 0$ for $s \ge 1$.
  For $\psi \in \con\infty$ and $\eta \in (0, 1)$, we choose $\varphi_\eta \colon \Ombar \times [0, 1] \ni (x, t) \mapsto \psi(x) \zeta(\frac{t}{\eta})$ in \eqref{eq:u_v_winf:c1} to obtain
  \begin{align*}
    &\pe  \left| \intom (u(\cdot, 0) - u_0) \psi \right|^2
     =    \left| \intom (u(\cdot, 0) - u_0) \varphi_\eta(\cdot, 0) \right|^2 \\
    &\le  c_1^2 \|\psi\|_{\sob12}^2 \int_0^\eta \zeta^2(\tfrac{t}{\eta}) \dt
     \le  c_1^2 \|\psi\|_{\sob12}^2 \eta
    \ra 0
    \qquad \text{as $\eta \ra 0$},
  \end{align*}
  that is, $\intom (u(\cdot, 0) - u_0) \psi = 0$ for all $\psi \in \con\infty$.
  Due to density of $\con\infty$ in $\leb2$, this implies $u(\cdot, 0) = u_0$ a.e.\ in $\Omega$
  and hence the first assertion in \eqref{eq:weak_sol_nonlin:init_cond}.
  Therefore, \eqref{eq:weak_sol_nonlin:u_sol} follows from \eqref{eq:u_v_weak_sol_nonlin:u_sol} and \eqref{eq:u_v_winf:int_part};
  first for all $\varphi \in C_c^\infty(\Ombar \times [0, \infty))$ and thus by a density argument also for
  all $\varphi \in L_{\loc}^2([0, \infty); \sob12)$.
  The remaining statements for the second solution component can be derived analogously.
\end{proof}

Finally, we show that an analogue to the entropy-type inequality \eqref{eq:test_g:entropy} also holds for the limit functions $u$, $v$.

\begin{lemma}\label{lm:entropy_eps_0}
  Let $G_i$, $i \in \{1, 2\}$, $\mc E$, $\mc D$, $\mc R$ be as in Theorem~\ref{th:approx_entropy},
  $T \in (0, \infty)$ and $0 \le \zeta \in C^\infty([0, T])$.
  The functions $u$, $v$ given by Lemma~\ref{sec:eps_sea_0} then satisfy \eqref{eq:approx_entropy:entropy}.
\end{lemma}
\begin{proof}
  For $\eps, \delta \in (0, 1)$, we denote the pairs constructed in Lemma~\ref{lm:k_to_infty} and Lemma~\ref{lm:delta_sea_0}
  by $(\ued, \ved)$ and $(\ue, \ve)$, respectively,
  and let the sequences $(\eps_j)_{j \in \N}$ and $(\delta_{j'})_{j' \in \N}$ be as in Lemma~\ref{lm:delta_sea_0} and Lemma~\ref{lm:eps_sea_0}.
  Moreover, again for $\eps, \delta \in (0, 1)$, we let $\Gid$, $i \in \{1, 2\}$, $\Eed$, $\Ded$ and $\Red$ be as in Lemma~\ref{lm:test_g}.
  
  In order to prove \eqref{eq:approx_entropy:entropy}, we essentially need to ensure that the inequality \eqref{eq:test_g:entropy}
  survives the limit processes $\eps = \eps_j \sea 0$ and $\delta = \delta_j \sea 0$.
  To that end, we first note that for any $\eta \gt 0$, the family
  \begin{align*}
    \left( \frac{D_1(|\uejdj|)}{\Sodjp(\uejdj) + \eta} \zeta \right)_{j, j' \in \N}
  \end{align*}
  is bounded in $L^\infty(\Omega \times (0, T))$
  and, as first $j \ra \infty$ and then $j' \ra \infty$,
  converges a.e.\ in $\Omega \times (0, T)$ to $\frac{D_1(u)}{S_1(u) + \eta} \zeta$,
  thanks to \eqref{eq:delta_sea_0:pw} and \eqref{eq:eps_sea_0:pw}.
  Thus, combined with \eqref{eq:delta_sea_0:nabla_l2} and \eqref{eq:eps_sea_0:nabla_l2}, we see that
  \begin{align*}
        \left( \frac{D_1(|\uejdj|)}{\Sodjp(\uejdj) + \eta} \zeta \right)^\frac12 \nabla \uejdj
    \rh \left( \frac{D_1(u)}{S_1(u) + \eta} \zeta \right)^\frac12 \nabla u
  \end{align*}
  in $L^2(\Omega \times (0, T); \R^n)$ as first $j' \ra \infty$ and then $j \ra \infty$
  for all $\eta \gt 0$.
  Consequently,
  \begin{align*}
        \liminf_{j \ra \infty} \liminf_{j' \ra \infty} \intntom \frac{D_1(|\uejdj|)}{\Sodjp(\uejdj) + \eta} |\nabla \uejdj|^2 \zeta
    \ge \intntom \frac{D_1(u)}{S_1(u) + \eta} |\nabla u|^2 \zeta
    \qquad \text{for all $\eta \gt 0$}
  \end{align*}
  by the weakly lower semicontinuity of the norm.
  Since $\eta \gt 0$ and by Fatou's lemma, we can conclude that
  \begin{align*}
        \liminf_{j \ra \infty} \liminf_{j' \ra \infty} \intntom \frac{D_1(|\uejdj|)}{\Sodjp(\uejdj)} |\nabla \uejdj|^2 \zeta
    \ge \intntom \frac{D_1(u)}{S_1(u)} |\nabla u|^2 \zeta.
  \end{align*}

  Next, we show that
  \begin{align}\label{eq:entropy_eps_0:fi_conv_int}
        \lim_{j \ra \infty} \lim_{j' \ra \infty} \intntom \Godjp'(\uejdj) \fo((\uejdj)_+, (\vejdj)_+) \zeta
    = \intntom G'(u) f_1(u, v) \zeta.
  \end{align}
  To that end, we first establish pointwise a.e.\ convergence to $0$ of the integrand; that is, we prove that
  \begin{align}\label{eq:entropy_eps_0:fi_conv}
    \lim_{j \ra \infty} \lim_{j' \ra \infty} \Godjp'(\uejdj) \fo((\uejdj)_+, (\vejdj)_+) = G'(u) f_1(u, v)
  \end{align}
  a.e.\ in $\Omega \times (0, \infty)$.
  We first prove convergence on the set
  \begin{align*}
    A \defs \left\{(x, t) \in \Omega \times (0, \infty): \lim_{j \ra \infty} \lim_{j' \ra \infty} \uejdj(x, t) = u(x, t) \gt 0\right\}.
  \end{align*}
  For $(x, t) \in A$ and arbitrary $\eta \in (0, \frac{u(x, t)}{2})$,
  there is $j_0 \in \N$ such that for $j \ge j_0$,
  we can find $j_0'(j) \in \N$ with the property that
  $|\uejdj(x, t) - u(x, t)| \lt \eta$ and hence $\uejdj(x, t) \gt \frac{u(x, t)}{2}$ for all $j' \ge j_0'(j)$ and $j \ge j_0$.
  Since $\frac1{S_1}$ is bounded on $(\frac{u(x, t)}{2}, \infty)$, Lebesgue's theorem gives
  \begin{align*}
        \lim_{j \ra \infty} \lim_{j' \ra \infty} \God'(\uejdj(x, t))
    &=  \lim_{j \ra \infty} \lim_{j' \ra \infty}
          \int_0^\infty \frac{\mathds 1_{(1, \uejdj(x, t))}(\sigma) - \mathds 1_{(\uejdj(x, t), 1)}(\sigma)}{S_1(\sigma) + \delta_j} \dsigma \\
    &=  \int_0^\infty \frac{\mathds 1_{(1, u(x, t))}(\sigma) - \mathds 1_{(u(x, t), 1)}(\sigma)}{S_1(\sigma)} \dsigma
    =   G_1'(u(x, t)).
  \end{align*}
  As $f_1$ is continuous and $u, v \ge 0$, we thus obtain \eqref{eq:entropy_eps_0:fi_conv} for all points in $A$.
  Next, we consider points in space-time where $u$ vanishes and set
  \begin{align*}
    B \defs \left\{(x, t) \in \Omega \times (0, \infty): \lim_{j \ra \infty} \lim_{j' \ra \infty} \uejdj(x, t) = u(x, t) = 0\right\}.
  \end{align*}
  Similarly as above, we can see that $|\uejdj| \lt 1$ for sufficiently large $j, j' \in \N$.
  Since
  \begin{align*}
        \left| \God'(\ued) \right|
    =   \left| \int_{1}^{\ued} \frac{1}{S_1(\sigma) + \delta} \dsigma \right|
    \le \frac1{\ul S_1} \int_{\ued}^1 \frac{1}{\sigma} \dsigma
    =   \frac1{\ul S_1} |\ln(\ued)|
    \qquad \text{in $\{0 \lt \ued \le 1\}$}
  \end{align*}
  for all $\eps, \delta \in (0, 1)$,
  the assumption \eqref{eq:weak_sol_nonlin:cond_f} and the fact that $f_1((\ued)_+, (\ved)_+) = 0$ in $\{\ued \le 0\}$
  imply that \eqref{eq:entropy_eps_0:fi_conv} also holds for points in $B$.
  As \eqref{eq:delta_sea_0:pw}, \eqref{eq:eps_sea_0:pw} and the nonnegativity of $u$
  assert that $(\Omega \times (0, \infty)) \setminus (A \cup B)$ is a null set,
  we indeed obtain \eqref{eq:entropy_eps_0:fi_conv} a.e.\ in $\Omega \times (0, \infty)$.
 
  Again thanks to \eqref{eq:weak_sol_nonlin:cond_f}, there is $c_1 \gt 0$ such that
  \begin{align*}
    &\pe  |\God'(\ued) \fod(\ued, \vejdj) \zeta| \\
    &\le  \frac{\|\zeta\|_{L^\infty(\Omega \times (0, T))}}{\ul S_1} \left(
            |\ln(\ued) \fo((\ued)_+, (\vejdj)_+)| \mathds 1_{\{0 \lt \ued \le 1\}}
            + \|f_1\|_{L^\infty([0, \infty)^2)} (\ued - 1) \mathds 1_{\{1 \lt \ued\}}
          \right) \\
    &\le  c_1 (1 + |\ued|)
    \qquad \text{in $\Omega \times (0, T)$ for all $\eps, \delta \in (0, 1)$}
  \end{align*}
  so that \eqref{eq:entropy_eps_0:fi_conv}, Vitali's theorem
  as well as the bound \eqref{eq:est_g:l2} assert \eqref{eq:entropy_eps_0:fi_conv_int}.

  As moreover $0 \le \God(\ued) \le c_2 (1 + \ued^2)$ in $\Omega \times (0, T))$ for all $\eps, \delta \in (0, 1)$ and some $c_2 \gt 0$
  and since $\lim_{j \ra \infty} \lim_{j' \ra \infty} (1 + u_{\eps_j \delta_{j'}}^2) = (1 + u^2)$ in $L^1(\Omega \times (0, T))$
  is contained in \eqref{eq:delta_sea_0:l2} and \eqref{eq:eps_sea_0:l2},
  Pratt's lemma asserts that
  \begin{align*}
        \lim_{j \ra \infty} \lim_{j' \ra \infty} \intntom G_{1 \delta_{j'}}(\uejdj) \zeta'
    =   \intntom G_1(u) \zeta'
    \qquad \text{for all $T \in (0, \infty)$}.
  \end{align*}
  Likewise, now relying on \eqref{eq:delta_sea_0:l2_space} and \eqref{eq:eps_sea_0:l2_space}
  instead of \eqref{eq:delta_sea_0:l2} and \eqref{eq:eps_sea_0:l2},
  we also obtain
  \begin{align*}
        \lim_{j \ra \infty} \lim_{j' \ra \infty} \intom G_{1 \delta_{j'}}(\uejdj(\cdot, T)) \zeta(\cdot, T)
    =   \intom G_1(u(\cdot, T)) \zeta(\cdot, T)
    \qquad \text{for a.e.\ $T \in (0, \infty)$}.
  \end{align*}

  Finally,
  \begin{align*}
        \God(u_0)
    =   \int_0^{u_0} \int_0^\rho \frac{1}{S_1(\sigma) + \delta} \dsigma \drho
    \ra \int_0^{u_0} \int_0^\rho \frac{1}{S_1(\sigma)} \dsigma \drho
    =   G_1(u_0)
    \qquad \text{as $\delta \sea 0$}
  \end{align*}
  by Beppo Levi's theorem so that according to Lebesgue's theorem,
  \begin{align*}
        \intom \God(u_0) \zeta(0)
    \ra \intom G_1(u_0) \zeta(0)
    \qquad \text{as $\delta \sea 0$}.
  \end{align*}

  Combined with analogous arguments for the second solution component,
  these convergences show that \eqref{eq:approx_entropy:entropy} holds for a.e.\ $T \in (0, \infty)$.
  Since $u, v \in C^0([0, \infty); \leb2) \cap L_{\loc}^2([0, \infty); \sob12)$ by Lemma~\ref{lm:u_v_winf},
  the inequality \eqref{eq:approx_entropy:entropy} holds indeed for all $T \in (0, \infty)$.
\end{proof}

Finally, we note that the previous two lemmata already contain the main results of this section.
\begin{proof}[Proof of Theorem~\ref{th:weak_sol_nonlin} and Theorem~\ref{th:approx_entropy}]
  Theorem~\ref{th:weak_sol_nonlin} and Theorem~\ref{th:approx_entropy}
  are direct consequences of Lem\-ma~\ref{lm:u_v_winf} and Lemma~\ref{lm:entropy_eps_0}, respectively.
\end{proof}

\section{Approximative solutions to \eqref{prob:nonlin}}\label{sec:approx_main}
In the remainder of the article, we will construct global weak solutions (in the sense of Definition~\ref{def:weak_sol_main} below)
of \eqref{prob:nonlin}.
To that end, we henceforth suppose that \eqref{eq:intro:di_chii},
either \eqref{eq:intro:h1} or \eqref{eq:intro:h2},
\eqref{eq:intro:cond_f1} or \eqref{eq:intro:cond_f2},
\eqref{eq:intro:mi_qi},
\eqref{eq:intro:main_cond} (with $p_i$ and $r_i$, $i \in \{1, 2\}$, as in \eqref{eq:intro:pi} and \eqref{eq:intro:ri})
as well as \eqref{eq:intro:cond_init} hold
and that $D_i, S_i, f_i$, $i \in \{1, 2\}$ are as in \eqref{eq:intro:Di_Si} and \eqref{eq:intro:fi}.

Sections~\ref{sec:approx_main}--\ref{sec:proof_11} are organized as follows.
In the present section, we will define approximations of $D_i, S_i, f_i$, $i \in \{1, 2\}$ as well as of $u_0$ and $v_0$
so that Theorem~\ref{th:weak_sol_nonlin}, which has been proven in the preceding section,
becomes applicable and thus provides us with global weak solutions $(\ua, \va)$, $\alpha \in (0, 1)$, to the corresponding approximative problems.

The main part of Section~\ref{sec:lim_alpha_sea_0} then consists of 
deriving $\alpha$-independent bounds from the entropy-like inequality given by Theorem~\ref{th:approx_entropy}.
This will then allow us to obtain solution candidates $(u, v)$ of \eqref{prob:nonlin} in Lemma~\ref{lm:alpha_sea_0}.
Finally, in Section~\ref{sec:proof_11}, we show that under the hypotheses of Theorem~\ref{th:ex_weak_nonlin},
these convergences are sufficiently strong to conclude that $(u, v)$ is indeed a global weak solution of \eqref{prob:nonlin}.

Having an application of Theorem~\ref{th:weak_sol_nonlin} in mind, we now define approximative functions for each henceforth fixed $\alpha \in (0, 1)$.
We begin by setting
\begin{align*}
  \Dia(s) \defs d_i \left( \frac{s+1}{1 + \alpha (s+1)} \right)^{m_i-1} + \alpha
  \quad \text{and} \quad
  \Sia(s) \defs \frac{\chi_i s (s + 1)^{q_i-1}}{(1 + \alpha (s+1))^{q_i}}
  \qquad \text{for $s \ge 0$ and $i \in \{1, 2\}$}.
\end{align*}

We also fix $\xi \in C^\infty(\R)$ with $\xi(s) = 1$ for $s \le 0$ and $\xi(s) = 0$ for $s \ge 1$
and set
\begin{align*}
  \fia(s_1, s_2) \defs f_i(s_1, s_2) \xioa(s_1) \xita(s_2)
\end{align*}
where $\xiia(s) \defs \xi(\alpha^{\frac1{4-\min\{q_1, q_2\}}} s - 1)$ for $s \in \R$;
in particular,
\begin{align}\label{eq:approx_main:xiia}
  \xiia(s) = 
  \begin{cases}
    1, & s \le \alpha^{-\frac1{4-\min\{q_1, q_2\}}}, \\
    0, & s \ge 2 \alpha^{-\frac1{4-\min\{q_1, q_2\}}}
  \end{cases}
  \qquad \text{for all $\alpha \in (0, 1)$ and $i \in \{1, 2\}$}.
\end{align}

As a last yet undefined component,
let us construct initial data $\una, \vna$ approximating $u_0, v_0$ in a suitable sense as $\alpha \sea 0$.
\begin{lemma}\label{lm:def_una_vna}
  There are families $(\una)_{\alpha \in (0, 1)}, (\vna)_{\alpha \in (0, 1)} \subset \con\infty$
  such that $\una \gt 0$ and $\vna \gt 0$ in $\Ombar$ for all $\alpha \in (0, 1)$,
  $(\intom u_0) (\intom \una) = ( \intom u_0 )^2$ and $(\intom v_0) (\intom \vna) = (\intom v_0)^2$ for all $\alpha \in (0, 1)$,
  \begin{align}\label{eq:def_una_vna:conv_lebl}
    (\una, \vna) \ra (u_0, v_0)
    \qquad \text{a.e.\ and in $X_1 \times X_2$ as $\alpha \sea 0$},
  \end{align}
  where $X_i \defs \leb{2-q_i}$ if $q_i \lt 1$ and $X_i \defs \lebl$ if $q_i = 1$ for $i \in \{1, 2\}$,
  as well as
  \begin{align}\label{eq:def_una_vna:eps_l2}
    \lim_{\alpha \sea 0} \alpha \|\una\|_{\leb p}^p = 0
    \quad \text{and} \quad
    \lim_{\alpha \sea 0} \alpha \|\vna\|_{\leb p}^p = 0,
    \qquad \text{where $p \defs {3-\min\{q_1, q_2\}}$}.
  \end{align}
\end{lemma}
\begin{proof}
  As $\con\infty$ is dense in $X_1$ (cf.\ \cite[Theorem~8.21]{AdamsFournierSobolevSpaces2003} for $X_1 = \lebl$),
  and since $u_0$ belongs to $X_1$ and is nonnegative by \eqref{eq:intro:cond_init},
  there is a sequence of nonnegative functions $(\tilde u_{0j})_{j \in \N} \subset \con\infty$ with $\tilde u_{0j} \ra u_0$ in $X_1$ as $j \ra \infty$.
  Since we may without loss of generality assume that $u_0 \not\equiv 0$,
  $\gamma_j \defs (\intom u_0) (\intom (\tilde u_{0j} + \frac1j))^{-1}$ is positive for all $j \in \N$
  so that the functions $u_{0j} \defs \gamma_j (\tilde u_{0j} + \frac1j)$
  not only fulfill $u_{0j} \ra u_0$ in $X_1$ as $j \ra \infty$
  but also $\intom u_{0j} = \intom u_0$ and $u_{0j} \ge \frac{\gamma_j}{j} \gt 0$ for $j \in \N$.
  Since $X_1 \embed \leb1$, after switching to a subsequence if necessary,
  we may without loss of generality also assume that $\tilde u_{0j} \ra u_0$ a.e.\ as $j \ra \infty$.

  For $\alpha \in (0, 1)$, we observe then that
  \begin{align*}
    A_\alpha \defs \left\{\, j \in \N : j \le \frac{1}{\alpha} \text{ and } \|u_{0j}\|_{\leb p}^{p+1} \le \frac1\alpha \,\right\} \cup \{1\}
  \end{align*}
  is nonempty and finite,
  so that
  \begin{align*}
    j_\alpha \defs \max A_\alpha
    \quad \text{and} \quad
    \una \defs u_{0 j_\alpha},
    \qquad \alpha \in (0, 1),
  \end{align*}
  are well-defined.
  Because $j_\alpha \ra \infty$ as $\alpha \sea 0$
  and $\alpha \|u_{0 j_\alpha}\|_{\leb p}^{p} \le \alpha^{1-\frac{p}{p+1}}$ for all $\alpha \in (0, 1)$ with $j_\alpha \gt 1$,
  we obtain the statement given an analogous definition of and argumentation for $(\vna)_{\alpha \in (0, 1)}$.
\end{proof}

With these preparations at hand, we are now able to apply Theorem~\ref{th:weak_sol_nonlin}
to obtain global weak $W^{1, 2}$-solutions of certain approximative problems.
\begin{lemma}\label{lm:ex_ua_va}
  Let $\alpha \in (0, 1)$,
  $\Dia, \Sia, \fia$, $i \in \{1, 2\}$ be as defined above
  and $\una, \vna$ be as given by Lemma~\ref{lm:def_una_vna}.
  Then there exists a global nonnegative weak $W^{1,2}$-solution (in the sense of Theorem~\ref{th:weak_sol_nonlin}) $(\ua, \va)$ belonging to $(\Winf)^2$ of
  \begin{align}\label{prob:alpha}
    \begin{cases}
      \uat = \nabla \cdot (\Doa(\ua) \nabla \ua - \Soa(\ua) \nabla \va) + \foa(\ua, \va) & \text{in $\Omega \times (0, \infty)$} \\
      \vat = \nabla \cdot (\Dta(\va) \nabla \va + \Sta(\va) \nabla \ua) + \fta(\ua, \va) & \text{in $\Omega \times (0, \infty)$} \\
      \partial_\nu \ua = \partial_\nu \va = 0 & \text{on $\partial \Omega \times (0, \infty)$} \\
      \ua(\cdot, 0) = \una, \va(\cdot, 0) = \vna & \text{in $\Omega$}.
    \end{cases}
  \end{align}
  Setting 
  \begin{align}\label{eq:ex_ua_va:gia}
    \Gia(s) \defs \int_1^s \int_1^{\rho} \frac1{\Sia(\sigma)} \dsigma \drho 
    \qquad \text{for $s \ge 0$ and $i \in \{1, 2\}$},
  \end{align}
  this solution moreover satisfies
  \begin{align}\label{eq:alpha_entropy:first_entropy}
    &\pe  \intom \Goa(\ua(\cdot, T))
          + \intom \Gta(\va(\cdot, T))
          + \intntom \frac{\Doa(\ua)}{\Soa(\ua)} |\nabla \ua|^2
          + \intntom \frac{\Dta(\ua)}{\Sta(\ua)} |\nabla \va|^2 \notag \\
    &\le  \intom \Goa(\una)
          + \intom \Gta(\vna)
          + \intntom \Goa'(\ua) \foa(\ua, \va)
          + \intntom \Gta'(\va) \fta(\ua, \va)
  \end{align}
  for all $T \in (0, \infty)$.
\end{lemma}
\begin{proof}
  As $\una, \vna$ belong to $\con\infty$ and are positive in $\Ombar$ by Lemma~\ref{lm:def_una_vna},
  the statement follows from Theorem~\ref{th:weak_sol_nonlin} and Theorem~\ref{th:approx_entropy} (with $\zeta \equiv 1$)
  once we have shown that \eqref{eq:weak_sol_nonlin:spaces_d}--\eqref{eq:weak_sol_nonlin:cond_f}
  hold for $D_i, S_i, f_i$ replaced by $\Dia, \Sia, \fia$, $i \in \{1, 2\}$.

  Indeed, by definition $\Dia, \Sia$ belong to $C^\infty([0, \infty))$ with
  \begin{align*}
        \alpha
    \le \Dia(s)
    \le d_i \begin{cases}
          \alpha^{1-m_i} + \alpha, & m_i \gt 1, \\ 
          (\frac{1}{1 + \alpha})^{m_i-1} + \alpha, & m_i \le 1
        \end{cases}
    \quad \text{and} \quad
        0
    \le \Sia(s)
    \le \chi_i \begin{cases}
          \alpha^{-q_i}, & q_i \gt 0, \\ 
          (\frac{1}{1 + \alpha})^{q_i} , & q_i \le 0
        \end{cases}
  \end{align*}
  as well as
  \begin{align*}
          \frac{|\Sia'(s)|}{\chi_i}
    &\le  \frac{(1 + |q_i-1|) (s+1)^{q_i-1}}{(1 + \alpha (s+1))^{q_i}}
          + \frac{|q_i| \alpha (s+1)^{q_i}}{(1 + \alpha (s+1))^{q_i+1}}
     \le  (1 + |q_i-1| + |q_i| \alpha)  \begin{cases}
            \alpha^{-q_i}, & q_i \gt 0, \\ 
            (\frac{1}{1 + \alpha})^{q_i} , & q_i \le 0
          \end{cases}
  \end{align*}
  for $s \ge 0$ and $i \in \{1, 2\}$.
  Also, for $i \in \{1, 2\}$, the function
  \begin{align*}
    [0, 1] \ni s \mapsto \frac{\Sia(s)}{s} = \frac{\chi_i (s + 1)^{q_i-1}}{(1 + \alpha (s+1))^{q_i}}
  \end{align*}
  is continuous and positive
  and, as $s \ge \frac{s+1}{2}$ for all $s \ge 1$,
  \begin{align*}
        \inf_{s \ge 1} \Sia(s)
    \ge \frac{\chi_i}{2} \inf_{s \ge 1} \left( \frac{s+1}{1 + \alpha(s+1)} \right)^{q_i}
    \ge \frac{\chi_i}{2} \begin{cases}
          (\frac{2}{1+2\alpha})^{q_i}, & q_i \gt 0, \\
          \alpha^{-q_i}, & q_i \le 0
        \end{cases}
    \qquad \text{for $i \in \{1, 2\}$}.
  \end{align*}
  That is, \eqref{eq:weak_sol_nonlin:spaces_d}, \eqref{eq:weak_sol_nonlin:spaces_s} and \eqref{eq:weak_sol_nonlin:cond_d_s} hold.

  As $\fia$ is continuous with $\supp \fia \subset [0, 2 \alpha^{-\frac1{{4-\min\{q_1, q_2\}}}}]^2 \sfed K$,
  $\|\fia\|_{L^\infty((0, \infty)^2)} = \|\fia\|_{C^0(K)}$ is finite
  and thus \eqref{eq:weak_sol_nonlin:spaces_f} is fulfilled for $i \in \{1, 2\}$.
  Moreover, the definitions of $f_1$ and $\foa$ entail that
  $[0, \infty)^2 \ni (s_1, s_2) \mapsto \frac{\foa(s_1, s_2)}{s_1}$ is also continuous and supported in $K$,
  implying $\lim_{s_1 \sea 0} \sup_{s_2 \ge 0} |\foa(s_1, s_2) \ln s_1| = 0$.
  The second statement in \eqref{eq:weak_sol_nonlin:cond_f} follows analogously.
\end{proof}

\section{The limit process \tops{$\alpha \sea 0$}{alpha to 0}: obtaining solution candidates}\label{sec:lim_alpha_sea_0}
Apart from assumptions made at the beginning of the preceding section,
throughout this section, for $\alpha \in (0, 1)$,
we also let $\Dia, \Sia, \xiia$, $i \in \{1, 2\}$, as introduced in Section~\ref{sec:approx_main},
$\una, \vna$ as well as $\ua, \va$ as given by Lemma~\ref{lm:def_una_vna} and Lemma~\ref{lm:ex_ua_va}, respectively,
and $\Gia$, $i \in \{1, 2\}$, as in \eqref{eq:ex_ua_va:gia}.

In order to prepare taking the limit $\alpha \sea 0$, we collect several a~priori estimates.
As already alluded to in the introduction, the main ingredient will be an entropy-like inequality;
that is, we will heavily rely on \eqref{eq:alpha_entropy:first_entropy}.

\subsection{Preliminary observations}
To streamline later arguments, in this subsection
we first collect several elementary statements regarding the parameters and nonlinearities involved in Theorem~\ref{th:ex_weak_nonlin}.

\begin{lemma}\label{lm:params}
  Set $\beta_i \defs m_i - q_i - 1$ for $i \in \{1, 2\}$.
  Then the inequalities 
  \begin{align}\label{eq:params:statement}
    2 \left( m_i - 1 - \frac{\beta_i}{2} \right) \lt p_i, \quad
    \beta_i \gt - 2
    \quad \text{and} \quad
    p_i \gt 0
  \end{align}
  hold.
\end{lemma}
\begin{proof}
  Recalling that $p_i \ge \beta_i + 2 + \frac{2(2-q_i)}{n}$ by \eqref{eq:intro:pi} and $q_i \le 1 \lt 2$ by \eqref{eq:intro:di_chii}, we have
  \begin{align*}
        2 \left(m_i - 1 - \frac{\beta_i}{2}\right)
    =   \beta_i + 2 q_i
    \lt \beta_i + 2 q_i + \frac{2(2-q_i)}{n}
    \le p_i
    \qquad \text{for $i \in \{1, 2\}$},
  \end{align*}
  which shows that the first inequality in \eqref{eq:params:statement} is fulfilled.
  The second one therein is equivalent to the assumption \eqref{eq:intro:mi_qi},
  upon which the third one follows by the definition of $p_i$ as $q_i \lt 2$.
\end{proof}

As further preparation, we estimate the functions $\Gia$ defined in \eqref{eq:ex_ua_va:gia} and their derivatives
both from above and from below.
\begin{lemma}\label{lm:g_el_est}
  Set
  \begin{align}\label{eq:g_el_est:def_l_g}
    L_{q}(s) \defs 
    \begin{cases}
      1, & q \lt 1, \\
      \ln s, & q = 1
    \end{cases}
    \qquad \text{for $s \ge 0$ and $q \le 1$.}
  \end{align}
  and let $\Gia$ be as in \eqref{eq:ex_ua_va:gia} for $i \in \{1, 2\}$ and $\alpha \in (0, 1)$.
  Then there are $C_1, C_2, C_3, C_4, \gt 0$ such that
  \begin{alignat}{2}\label{eq:g_el_est:est_g}
    \Gia(s) &\begin{cases}
      \ge C_1 \left( \frac{s+1}{1+\alpha(s+1)} \right)^{2-q_i} L_{q_i}(s + \ure) - C_2 \\
      \le  C_2 s^{2 - q_i} L_{q_i}(s) + C_2 \alpha s^{3 - q_i} + C_2
    \end{cases}
    &&\qquad \text{for $s \ge 0$}, \\
  \label{eq:g_el_est:est_g'_lt_1}
    \Gia'(s) &\begin{cases}
      \ge C_3 \ln s \\
      \le  0
    \end{cases}
    &&\qquad \text{for $s \in (0, 1)$} \quad \text{and} \\
  \label{eq:g_el_est:est_g'_ge_1}
    \Gia'(s) &\begin{cases}
      \ge G_i'(s) - C_4 \alpha s^{2-q_i} \\
      \le  G_i'(s) + C_4 \alpha s^{2-q_i}
    \end{cases}
    &&\qquad \text{for $s \ge 1$}
  \end{alignat}
  for $\alpha \in (0, 1)$ and $i \in \{1,2\}$.
\end{lemma}
\begin{proof}
  We fix $i \in \{1, 2\}$.
  Since
  \begin{align*}
        \left| \frac{\partial}{\partial \alpha} [1 + \alpha (s+1)]^{q_i} \right|
    &=  |q_i| [1 + \alpha (s+1)]^{q_i-1} (s+1)
    \le |q_i| (s+1)
    \qquad \text{for all $s \ge 0$ and $\alpha \in (0, 1)$},
  \end{align*}
  the mean value theorem implies that
  \begin{align*}
          \chi_i |\Gia'(s) - G_i'(s)|
    &=    \sign(s-1) \int_1^s \frac{[1 + \alpha (\sigma+1)]^{q_i} - 1}{\sigma (\sigma+1)^{q_i-1}} \dsigma
     \le  \alpha |q_i| \sign(1-s) \int_s^1 \frac{(\sigma+1)^{2-q_i}}{\sigma} \dsigma
  \end{align*}
  for all $s \gt 0$ and $\alpha \in (0, 1)$.
  Estimating here $\sigma + 1 \le 2$ and $\sigma + 1 \le 2\sigma$ for $\sigma \in (0, 1)$ and $\sigma \ge 1$, respectively,
  we obtain
  \begin{alignat*}{2}
          \chi_i |\Gia'(s) - G_i'(s)|
    &\le  2^{2-q_i} \alpha |q_i| \int_s^1 \frac{1}{\sigma} \dsigma
     =    2^{2-q_i} \alpha |q_i| |\ln s|
    &&\qquad \text{for all $s \in (0, 1)$ and $\alpha \in (0, 1)$}
  \intertext{and}
          \chi_i |\Gia'(s) - G_i'(s)|
    &\le  2^{2-q_i} \alpha |q_i| \int_1^s \sigma^{1-q_i} \dsigma
     \le  \frac{2^{2-q_i} \alpha |q_i|}{2-q_i} s^{2-q_i}
    &&\qquad \text{for all $s \ge 1$ and $\alpha \in (0, 1)$}.
  \end{alignat*}
  As moreover $\Gia'(s) \le 0$ for $s \in (0, 1)$ and $\alpha \in (0, 1)$
  and
  \begin{align*}
          \chi_i |G_i'(s)|
    &=    \int_s^1 \frac{(\sigma+1)^{1-q_i}}{\sigma} \dsigma
     \le  2^{1-q_i} |\ln s|
    \qquad \text{for all $s \in (0, 1)$ and $\alpha \in (0, 1)$},
  \end{align*}
  consequences thereof are \eqref{eq:g_el_est:est_g'_lt_1} and \eqref{eq:g_el_est:est_g'_ge_1} for a certain $C_3, C_4 \gt 0$.

  Furthermore, again making use of the fact that $s+1 \le 2s$ for $s \ge 1$,
  a direct computation shows that
  \begin{align}\label{eq:g_el_est:est_g_ge_1}
          \chi_i G_i(s)
    &=    \int_1^s \int_1^\rho \frac{(\sigma+1)^{1-q_i}}{\sigma} \dsigma \drho \notag \\
    &\le  2^{1-q_i} \int_1^s \int_1^\rho \sigma^{-q_i} \dsigma \drho \notag \\
    &\le  \frac{2^{1-q_i}}{1 - q_i \mathds 1_{\{q_i \lt 1\}}} \int_1^s \rho^{1-q_i} L_{q_i}(\rho) \drho \notag \\
    &\le  \frac{2^{1-q_i}}{(2 - q_i)(1 - q_i \mathds 1_{\{q_i \lt 1\}})} s^{2-q_i} L_{q_i}(s)
    \qquad \text{for $s \ge 1$}.
  \end{align}
  In a similar vein, we obtain $c_1, c_2 \gt 0$ such that
  \begin{align*}
          \chi_i \Gia(s)
    &=    \int_1^s \int_1^\rho \frac{(1 + \alpha(\sigma+1))^{q_i}}{\sigma (\sigma+1)^{q_i-1}} \dsigma \drho \\
    &\ge  (1 + \alpha (s + 1))^{\min\{q_i, 0\}} \int_1^s \int_1^\rho \sigma^{-q_i} \dsigma \drho \\
    &\ge  \frac{c_1 s^{2-q_i} L_{q_i}(s)}{(1 + \alpha (s + 1))^{\max\{-q_i, 0\}}} - c_2 s \\
    &\ge  \frac{c_1 s^{2-q_i} L_{q_i}(s+\ure)}{(1 + \alpha (s + 1))^{\max\{-q_i, 0\}}}
          -  \frac{c_1 \ln(1+\ure) \mathds 1_{\{q_i \lt 1\}} s^{2-q_i}}{(1 + \alpha (s + 1))^{\max\{-q_i, 0\}}}
          - c_2 s
    \qquad \text{for $s \ge 1$ and $\alpha \in (0, 1)$},
  \end{align*}
  where in the last step we have made use of the fact that $\ln(s + \ure) - \ln s = \ln \frac{s + \ure}{s} \le \ln (1 + \ure)$ for $s \ge 1$.
  Since the first term on the right-hand side herein grows faster than the other two, there is moreover $c_3 \gt 0$ such that
  \begin{align*}
          \chi_i \Gia(s)
    &\ge  \frac{c_1 s^{2-q_i} L_{q_i}(s+\ure)}{2(1 + \alpha (s + 1))^{\max\{-q_i, 0\}}}
          - c_3 \\
    &\ge  \frac{c_1}{2} \left( \frac{s}{1+\alpha(s+1)} \right)^{2-q_i} L_{q_i}(s + \ure) 
          - c_3 \\
    &\ge  \frac{c_1}{2^{3-q_i}} \left( \frac{s+1}{1+\alpha(s+1)} \right)^{2-q_i} L_{q_i}(s + \ure) 
          - c_3
    \qquad \text{for $s \ge 1$ and $\alpha \in (0, 1)$},
  \end{align*}
  which, when combined with \eqref{eq:g_el_est:est_g'_lt_1}--\eqref{eq:g_el_est:est_g_ge_1},
  implies the existence of $C_1, C_2 \gt 0$
  such that \eqref{eq:g_el_est:est_g} holds for all $s \ge 1$ and $\alpha \in (0, 1)$.
  
  Finally, by integrating \eqref{eq:g_el_est:est_g'_lt_1},
  we see that there is $c_4 \gt 0$ such that $-c_4 \le \Gia(s) \le c_4$ for all $s \in [0, 1)$ and $\alpha \in (0, 1)$
  so that, possibly after enlarging $C_1$ and $C_2$,
  \eqref{eq:g_el_est:est_g} is indeed valid for all $s \ge 0$ and $\alpha \in (0, 1)$.
\end{proof}

The estimates obtained in Lemma~\ref{lm:g_el_est} and the definitions of $\Dia$ and $\Sia$ now allow us to
infer the following from the entropy-like inequality \eqref{eq:alpha_entropy:first_entropy}.
\begin{lemma}\label{lm:alpha_entropy}
  Let $T \in (0, \infty)$.
  Then there exists $C_1, C_2 \gt 0$ such that
  \begin{align}\label{eq:alpha_entropy:statement}
    &\pe    C_1 \intom \Ba^{2 - q_1}(\ua(\cdot, t)) L_{q_1}(\ua(\cdot, t) + \ure)
          + C_1 \intom \Ba^{2 - q_2}(\va(\cdot, t)) L_{q_2}(\va(\cdot, t) + \ure) \notag \\
    &\pe  + \frac{d_1}{\chi_1} \intnstom \Ba^{m_1 - q_1 - 1}(\ua) |\nabla \ua|^2
          + \frac{d_2}{\chi_2} \intnstom \Ba^{m_2 - q_2 - 1}(\va) |\nabla \va|^2 \notag \\
    &\le  C_2 
          + \intnstom \Goa'(\ua) \foa(\ua, \va)
          + \intnstom \Gta'(\va) \fta(\ua, \va)
  \end{align}
  for all $t \in (0, T)$ and all $\alpha \in (0, 1)$, where $L_{q_i}$ is as in \eqref{eq:g_el_est:def_l_g} and
  \begin{align}\label{eq:alpha_entropy:ba}
    \Ba(s) \defs \frac{s+1}{1 + \alpha(s+1)}, \qquad s \ge 0, \alpha \in (0, 1).
  \end{align}
\end{lemma}
\begin{proof}
  As according to \eqref{eq:def_una_vna:conv_lebl} and \eqref{eq:def_una_vna:eps_l2},
  there is $c_1 \gt 0$ such that
  \begin{align*}
            \intom \una^{2 - q_1} L_{q_1}(\una) + \alpha \intom \una^{3-q_1}
          + \intom \vna^{2 - q_2} L_{q_2}(\vna) + \alpha \intom \vna^{3-q_2}
    &\le  c_1
    \qquad \text{for all $\alpha \in (0, 1)$},
  \end{align*}
  an application of \eqref{eq:g_el_est:est_g} gives $c_2 \gt 0$ such that
  \begin{align*}
            \intom \Goa(\una)
          + \intom \Gta(\vna)
    &\le  c_2
    \qquad \text{for all $\alpha \in (0, 1)$},
  \end{align*}
  Moreover,
  \begin{align*}
    \Dia(s) \ge d_i \Ba^{m_i-1}(s)
    \quad \text{and} \quad
    \Sia(s) \le \chi_i \Ba^{q_i}(s)
  \end{align*}
  and hence $\frac{\Dia(s)}{\Sia(s)} \ge \frac{d_i}{\chi_i} \Ba^{m_i - q_i - 1}(s)$ for $s \ge 0$, $\alpha \in (0, 1)$ and $i \in \{1, 2\}$.
  Also making use of the first inequality in \eqref{eq:g_el_est:est_g},
  we can then infer \eqref{eq:alpha_entropy:statement} from \eqref{eq:alpha_entropy:first_entropy}
  for certain $C_1, C_2 \gt 0$.
\end{proof}

\subsection{Controlling the right-hand side of \eqref{eq:alpha_entropy:statement}}
In order to obtain $\alpha$-independent a priori estimates from \eqref{eq:alpha_entropy:statement},
we need to obtain an upper bound for the terms on the right-hand side therein.
Restricted to the set where $\ua$ and $\va$ are at least $1$,
we will bound the corresponding integrand using one of the assumptions \eqref{eq:intro:cond_f1} and \eqref{eq:intro:cond_f2}.
This is complemented by the following observation essentially showing we may indeed focus on that regime.
\begin{lemma}\label{lm:entropy_rhs_basic}
  There is $C \gt 0$ such that 
  \begin{align}\label{eq:entropy_rhs_basic:statement}
    &\pe  \Goa'(\ua) \foa(\ua, \va) + \Gta'(\va) \fta(\ua, \va) \notag \\
    &\le  C +
          \left( G_1'(\ua) \fo(\ua, \va) + G_2'(\va) \ft(\ua, \va) \right) \xioa(\ua) \xita(\va)
            \mathds 1_{\{\ua \ge 1\} \cap \{\va \ge 1\}}
  \end{align}
  a.e.\ in $\Omega \times (0, \infty)$ for all $\alpha \in (0, 1)$.
\end{lemma}
\begin{proof}
  For $\alpha \in (0, 1)$, we fix representatives of $\ua$ and $\va$ in $L_{\loc}^1(\Ombar \times [0, \infty))$
  so that sets such as $\{\ua \lt 1\}$ or $\{\va \lt 1\}$ are well-defined.

  According to \eqref{eq:g_el_est:est_g'_lt_1}, there is $c_1 \gt 0$ such that
  \begin{align*}
    c_1 \ln s \le \Gia'(s) \le 0
    \qquad \text{for all $s \in (0, 1)$, $\alpha \in (0, 1)$ and $i \in \{1, 2\}$}.
  \end{align*}
  Recalling the definition of $\fia$ and that $\ua, \va$ are nonnegative, this implies
  \begin{alignat*}{2}
          \Goa'(\ua) \foa(\ua, \va)
    &\le  c_1 |\ln \ua| \mu_1 \ua^2 \xioa(\ua) \xita(\va) 
    &&\qquad \text{in $\{\ua \lt 1\}$} \quad \text{and} \\
          \Gta'(\va) \fta(\ua, \va)
    &\le  c_1 |\ln \va| (\mu_2 \va^2 + a_2 \ua \va) \xioa(\ua) \xita(\va)
    &&\qquad \text{in $\{\va \lt 1\}$}
  \end{alignat*}
  for all $\alpha \in (0, 1)$.
  Since $(0, 1) \ni s \mapsto s \ln s$ is bounded, there is $c_2 \gt 0$ such that
  \begin{alignat}{2}\label{eq:entropy_rhs_basic:ua_lt_1}
          \Goa'(\ua) \foa(\ua, \va)
    &\le  c_2
    &&\qquad \text{in $\{\ua \lt 1\}$} \quad \text{and} \\
    \label{eq:entropy_rhs_basic:va_lt_1}
          \Gta'(\va) \fta(\ua, \va)
    &\le  c_2 + c_2 \ua \xioa(\ua) \xita(\va)
    &&\qquad \text{in $\{\va \lt 1\}$}.
  \end{alignat}
  for all $\alpha \in (0, 1)$
  and thus \eqref{eq:entropy_rhs_basic:statement} holds on the set $\{\ua \lt 1\} \cap \{\va \lt 1\}$ for some $C \gt 0$.

  Moreover, by \eqref{eq:g_el_est:est_g'_ge_1}, there is $c_3 \gt 0$ such that
  \begin{align*}
    |\Gia'(s) - G_i'(s)| \le c_3 \alpha s^{2-q_i}
    \qquad \text{for all $s \ge 1$, $\alpha \in (0, 1)$ and $i \in \{1, 2\}$}.
  \end{align*}
  As \eqref{eq:approx_main:xiia} entails that
  $\ua$ and $\va$ are bounded by $2 \alpha^{-\frac1{4-\min\{q_1, q_2\}}}$ on $\supp \xioa(\ua)$ and $\supp \xita(\va)$, respectively,
  and hence
  \begin{align*}
          \alpha \ua^{2-q_1} |\foa(\ua, \va)|
    &\le  \alpha \ua^{2-q_1} (\lambda_1 \ua + \mu_1 \ua^2 + a_1 \ua \va) \xioa(\ua) \xita(\va) \\
    &\le  2^{4-q_1} (\lambda_1 + \mu_1 + a_1)
     \sfed c_4
    \qquad \text{in $\{\ua \ge 1\}$ for all $\alpha \in (0, 1)$},
  \end{align*}
  we can conclude
  \begin{align}\label{eq:entropy_rhs_basic:ua_ge_1}
          \Goa'(\ua) \foa(\ua, \va)
    &\le  G_1'(\ua) \fo(\ua, \va) \xi_\alpha(\va) \xi_\alpha(\va)
          + c_3 c_4
    \qquad \text{in $\{\ua \ge 1\}$ for all $\alpha \in (0, 1)$}.
  \end{align}
  Likewise, there is $c_5 \gt 0$ such that
  \begin{align}\label{eq:entropy_rhs_basic:va_ge_1}
          \Gta'(\va) \fta(\ua, \va)
    &\le  G_2'(\va) \ft(\ua, \va) \xi_\alpha(\va) \xi_\alpha(\va)
          + c_3 c_5
    \qquad \text{in $\{\va \ge 1\}$ for all $\alpha \in (0, 1)$}.
  \end{align}
  Therefore, after enlarging $C$ if necessary,
  \eqref{eq:entropy_rhs_basic:statement} holds also in the regime $\{\ua \ge 1\} \cap \{\va \ge 1\}$.

  Furthermore,
  \begin{align*}
        \fo(\ua, \va)
    \le \ua \left(\lambda_1 - \frac{\mu_1}{2} \ua + a_1 \right) - \frac{\mu_1}{2} \ua^2
    \le - \frac{\mu_1}{2} \ua^2
    \qquad \text{in } \left\{\ua \ge \frac{2(\lambda_1 + a_1)}{\mu_1}\right\} \cap \{\va \lt 1\} \text{ for all $\alpha \in (0, 1)$}
  \end{align*}
  so that since $G_1'(s) \ge 0$ for $s \ge 1$, we have
  \begin{align*}
        G_1'(\ua) \fo(\ua, \va) \xioa(\ua) \xita(\va)
    \le c_6
        - \frac{\mu_1}{2} \ua^2 \xioa(\ua) \xita(\va)
    \qquad \text{in $\{\ua \ge 1\} \cap \{\va \lt 1\}$ for all $\alpha \in (0, 1)$},
  \end{align*}
  wherein
  $c_6 \defs \|G_1' \fo(\cdot, 1)\|_{L^\infty(0, \frac{2(\lambda_1 + a_1)}{\mu_1})}$ is finite as $G_1' \fo(\cdot, 1)$ is continuous on $[0, \infty)$.
  Combined with \eqref{eq:entropy_rhs_basic:va_lt_1} and \eqref{eq:entropy_rhs_basic:ua_ge_1},
  and possibly after enlarging $C$,
  this shows that \eqref{eq:entropy_rhs_basic:statement} holds also on the set $\{\ua \ge 1\} \cap \{\va \lt 1\}$,
  
  Finally, for the remaining subset $\{\ua \lt 1\} \cap \{\va \ge 1\}$ of $\Omega \times (0, \infty)$, we can argue similarly as above.
\end{proof}

If \eqref{eq:intro:cond_f1} holds,
the preceding lemma immediately allows us to bound the integrands on the right-hand side of \eqref{eq:alpha_entropy:statement}.
\begin{lemma}\label{lm:rhs_bdd_f1}
  Let $T \in (0, \infty)$ and suppose that \eqref{eq:intro:cond_f1} holds.
  Then we can find $C_1, C_2 \gt 0$ such that
  \begin{align}\label{eq:rhs_bdd_f1:statement}
    &\pe  \intom \Goa'(\ua) \foa(\ua, \va) + \intom \Gta'(\va) \fta(\ua, \va) \notag \\
    &\le  C_1 - C_2 \intntom \ua^2 \ln \ua - C_2 \intntom \va^2 \ln \va
    \qquad \text{a.e.\ in $\Omega \times (0, \infty)$ for all $\alpha \in (0, 1)$}.
  \end{align}
\end{lemma}
\begin{proof}
  This directly follows from combining \eqref{eq:intro:cond_f1}, \eqref{eq:alpha_entropy:statement} and \eqref{eq:entropy_rhs_basic:statement}.
\end{proof}


In the majority of the remainder of this subsection,
we will show that \eqref{eq:rhs_bdd_f1:statement} also holds if we assume \eqref{eq:intro:cond_f2} instead of \eqref{eq:intro:cond_f1}.
To that end, we may assume that \eqref{eq:intro:h2} holds
since  the right-hand side of \eqref{eq:alpha_entropy:statement} is trivially bounded in the case of \eqref{eq:intro:h1}.
The key ingredient to the corresponding proof will be the Gagliardo--Nirenberg inequality 
whose application we prepare by obtaining locally uniform-in-time $\leb1$ bounds in the following
\begin{lemma}\label{lm:ua_va_l1}
  Let $T \in (0, \infty)$ and suppose that \eqref{eq:intro:h2} holds. There is $C \gt 0$ such that
  \begin{align}\label{eq:ua_va_l1:statement}
    \intom \ua(\cdot, t) + \intom \va(\cdot, t) \le C
    \qquad \text{for all $t \in (0, T)$ and $\alpha \in (0, 1)$}.
  \end{align}
\end{lemma}
\begin{proof}
  Testing the first equation in \eqref{prob:alpha} with the constant function $a_2 \gt 0$,
  recalling the definition of $\foa$
  and applying Young's inequality give
  \begin{align*}
    &\pe  a_2 \intom \ua(\cdot, t) - a_2 \intom \una
     =    a_2 \intnstom \uat \\
    &=    a_2 \lambda_1 \intnstom \ua \xioa(\ua) \xita(\va)
          - a_2 \mu_1 \intnstom \ua^2 \xioa(\ua) \xita(\va)
          + a_1 a_2 \intnstom \ua \va \xioa(\ua) \xita(\va) \\
    &\le  \frac{a_2 \lambda_1^2}{4 \mu_1} |\Omega| T
          + a_1 a_2 \intnstom \ua \va \xioa(\ua) \xita(\va)
    \qquad \text{for $t \in (0, T)$ and $\alpha \in (0, 1)$}.
  \end{align*}
  As likewise
  \begin{align*}
          a_1 \intom \va(\cdot, t) - a_1 \intom \vna
    &\le  \frac{a_1 \lambda_2^2}{4 \mu_2} |\Omega| T
          - a_1 a_2 \intnstom \ua \va \xioa(\ua) \xita(\va)
    \qquad \text{for $t \in (0, T)$ and $\alpha \in (0, 1)$},
  \end{align*}
  we conclude
  \begin{align*}
          a_2 \intom \ua(\cdot, t)
          + a_1 \intom \va(\cdot, t) 
    &\le  a_2 \intom \una 
          + a_1 \intom \una
          + \frac{a_2 \lambda_1^2}{4 \mu_1} |\Omega| T
          + \frac{a_1 \lambda_2^2}{4 \mu_2} |\Omega| T
    \qquad \text{for $t \in (0, T)$ and $\alpha \in (0, 1)$}.
  \end{align*}
  In view of \eqref{eq:def_una_vna:conv_lebl}, this implies \eqref{eq:ua_va_l1:statement} for a certain $C \gt 0$.
\end{proof}

\begin{lemma}\label{lm:gni_f1}
  Let $T \in (0, \infty)$, $\eta \gt 0$, $\beta_i \defs m_i - q_i - 1$ for $i \in \{1, 2\}$ and suppose that \eqref{eq:intro:h2} holds.
  For $p \in (0, \frac{(\beta_1+2)n + 2}{n})$, there is $C_1 \gt 0$ such that
  \begin{alignat}{2}
    \label{eq:gni_f1:u}
          \intom \ua^{p}(\cdot, t) \xioa(\ua(\cdot, t))
    &\le  \eta \intom \Ba^{\beta_1}(\ua(\cdot, t)) |\nabla \ua(\cdot, t)|^2 + C_1
    &&\qquad \text{for all $t \in (0, T)$ and $\alpha \in (0, 1)$}
  \intertext{and, for $p \in (0, \frac{(\beta_2+2)n + 2}{n})$, there is $C_2 \gt 0$ such that}
    \label{eq:gni_f1:v}
          \intom \va^{p}(\cdot, t) \xita(\va(\cdot, t))
    &\le  \eta \intom \Ba^{\beta_2}(\va(\cdot, t)) |\nabla \va(\cdot, t)|^2 + C_2
    &&\qquad \text{for all $t \in (0, T)$ and $\alpha \in (0, 1)$},
  \end{alignat}
  where $\Ba$ is as in \eqref{eq:alpha_entropy:ba};
  that is, $\Ba(s) = \frac{s+1}{1 + \alpha(s+1)}$ for $s \ge 0$ and $\alpha \in (0, 1)$.
\end{lemma}
\begin{proof}
  As $\Ba(s) \le s+1$ for all $s \ge 0$ and $\alpha \in (0, 1)$,
  Lemma~\ref{lm:ua_va_l1} allows us to fix $c_1 \gt 0$ such that
  \begin{align*}
    \intom \Ba(\ua(\cdot, t)) + \intom \Ba(\va(\cdot, t)) \le c_1
    \qquad \text{for all $t \in (0, T)$ and $\alpha \in (0, 1)$}.
  \end{align*}

  The definitions $\tilde p_1 \defs  \frac{2((\beta_1 + 2)n + 2)}{(\beta_1 + 2)n}$ and
  $\tilde q_1 \defs \frac{2}{\beta_1 + 2}$ imply
  \begin{align*}
      b
    &\defs \frac{\frac1{\tilde q_1} - \frac1{\tilde p_1}}{\frac1{\tilde q_1} - \frac{n-2}{2n}}
     =     \frac{(\beta_1 + 2) ((\beta_1 + 2) n + 2)n - (\beta_1 + 2) n^2}{(\beta_1 + 2) ((\beta_1 + 2) n + 2)n - (n-2) ((\beta_1 + 2) n + 2)} \\
    &=     \frac{(\beta_1 + 2) n ((\beta_1 + 1) n + 2)}{((\beta_1 + 1)n + 2))((\beta_1 + 2) n + 2)}
     =     \frac{(\beta_1 + 2) n}{(\beta_1 + 2)n + 2}
    \in (0, 1).
  \end{align*}
  Since $\frac{\tilde p_1 b}{2} = 1$,
  an application of the Gagliardo--Nirenberg inequality
  (cf.\ \cite[Lemma~2.3]{LiLankeitBoundednessChemotaxisHaptotaxis2016} for a version allowing for merely positive $\tilde q_1$)
  gives $c_2 \gt 0$ such that
  \begin{align*}
          \intom \varphi^{\tilde p_1}
    &\le  c_2 \left( \intom |\nabla \varphi|^2 \right)
            \left( \intom \varphi^{\frac{2}{\beta_1+2}} \right)^{\frac{\tilde p_1 (1-b)}{\tilde q_1}}
          + c_2 \left( \intom \varphi^{\frac{2}{\beta_1+2}} \right)^{\frac{\tilde p_1}{\tilde q_1}}
    \qquad \text{for all $\varphi \in \sob12$}.
  \end{align*}
  Thus, setting $c_3 \defs \frac{(\beta+1)^2}{4} c_1^{\frac{\tilde p_1( 1-b)}{\tilde q_1}} c_2$
  and noting that $|\Ba'|(s) = \frac{1}{(1 + \alpha (s+1))^2} \le 1$ for $s \ge 0$ and $\alpha \in (0, 1)$, we conclude
  \begin{align*}
          \intom \Ba^\frac{(\beta_1 + 2)n + 2}{n}(\ua(\cdot, t))
    &=    \intom \left( \Ba^\frac{\beta_1+2}{2}(\ua(\cdot, t)) \right)^{\tilde p_1} \\
    &\le  c_2 \intom \left| \nabla \Ba^\frac{\beta_1+2}{2}(\ua(\cdot, t)) \right|^2
          \left( \intom \Ba(\ua(\cdot, t)) \right)^{\frac{\tilde p_1 (1-b)}{\tilde q_1}}
          + c_2 \left( \intom \Ba(\ua(\cdot, t)) \right)^{\frac{\tilde p_1}{\tilde q_1}} \\
    &\le  c_3 \intom \Ba^{\beta_1}(\ua(\cdot, t)) |\Ba'(\ua(\cdot, t))|^2 |\nabla \ua(\cdot, t)|^2
          + c_1^{\frac{\tilde p_1}{\tilde q_1}} c_2 \\
    &\le  c_3 \intom \Ba^{\beta_1}(\ua(\cdot, t)) |\nabla \ua(\cdot, t)|^2
          + c_1^{\frac{\tilde p_1}{\tilde q_1}} c_2
    \qquad \text{for all $t \in (0, T)$ and $\alpha \in (0, 1)$}.
  \end{align*}

  We now fix $\eta \gt 0$ and $p \in (0, \frac{(\beta_1+2)n + 2}{n})$.
  By Young's inequality, we then obtain $c_4 \gt 0$ such that
  \begin{align}\label{eq:gni_f1:c3}
          \intom \Ba^{p}(\ua(\cdot, t))
    &\le  \frac{\eta}{4^{p}} \intom \Ba^{\beta_1}(\ua(\cdot, t)) |\nabla \ua(\cdot, t)|^2
          + c_4
    \qquad \text{for all $t \in (0, T)$ and $\alpha \in (0, 1)$}.
  \end{align}
  For $\alpha \in (0, 1)$ and $s \in \supp \xioa \subset [0, 2 \alpha^{-1/(4-\min\{q_1, q_2\})}] \subset [0, 2 \alpha^{-1}]$, we have
  \begin{align*}
        s
    =   (1 + \alpha(s+1)) \frac{s}{1 + \alpha(s+1)} 
    \le 4 \Ba(s)
  \end{align*}
  so that the monotonicity of $[0, \infty) \ni s \mapsto s^{p}$ asserts
  \begin{align*}
          \intom \left(\frac{\ua(\cdot, t)}{4}\right)^{p} \xioa(\ua(\cdot, t))
    &\le  \intom \Ba^{p}(\ua(\cdot, t))
    \qquad \text{for all $t \in (0, T)$ and $\alpha \in (0, 1)$}.
  \end{align*}
  Together with \eqref{eq:gni_f1:c3},
  this implies \eqref{eq:gni_f1:u} for $C_1 \defs 4^{p} c_4$.
  By an analogous argumentation, we also obtain the corresponding statement for the second solution component.
\end{proof}

If $\beta_1$ and $\beta_2$ are sufficiently large compared to $q_1$ and $q_2$,
one might hope that the estimates obtained in Lemma~\ref{lm:gni_f1}
are strong enough to control the right-hand side of \eqref{eq:alpha_entropy:statement}.
This idea can be quantified as follows.
\begin{lemma}\label{lm:rhs_bdd_f2}
  Let $T \in (0, \infty)$ and suppose that \eqref{eq:intro:cond_f2} and \eqref{eq:intro:h2} hold.
  Then there are $C_1, C_2 \gt 0$
  such that \eqref{eq:rhs_bdd_f1:statement} holds.
\end{lemma}
\begin{proof}
  We will crucially rely on the assumption \eqref{eq:intro:cond_f2} which asserts that $m_1 \gt \ul m_1$ or $m_2 \gt \ul m_2$,
  where
  \begin{align*}
    \ul m_1 \defs \frac{2n-2}{n} + \frac{(3-q_2)(2-q_1) - (3-q_1)(2-q_2)}{2-q_2}
    \quad \text{and} \quad
    \ul m_2 \defs  \frac{2n-2}{n} + (q_2 - q_1).
  \end{align*}
  Setting again $\beta_i \defs m_i - q_i - 1$ for $i \in \{1, 2\}$, these definitions imply
  \begin{alignat*}{2}
          \frac{(\beta_1 + 2)n + 2}{n}
    &\gt  \ul m_1 - q_1 + 1 + \frac2n
    =     \frac{(3-q_2)(2-q_1)}{2-q_2}
    &&\qquad \text{if $m_1 \gt \ul m_1$} \quad \text{and} \\
          \frac{(\beta_2 + 2)n + 2}{n}
    &\gt  \ul m_2 - q_2 + 1 + \frac2n
    =     3 - q_1
    &&\qquad \text{if $m_2 \gt \ul m_2$},
  \end{alignat*}
  whence there is $\eta \in (0, 1)$ such that still
  \begin{alignat}{2}\label{eq:engery_rhs_f2:m1}
    \frac{(\beta_1 + 2)n + 2}{n} &\gt \frac{(3-q_2)(2-q_1+\eta)}{2-q_2} &&\qquad \text{if $m_1 \gt \ul m_1$} \quad \text{and} \\
  \label{eq:engery_rhs_f2:m2}
    \frac{(\beta_2 + 2)n + 2}{n} &\gt \frac{3 - q_1}{1-\eta} &&\qquad \text{if $m_2 \gt \ul m_2$}.
  \end{alignat}
  
  For $s \ge 1$,
  we have $\frac{s+1}{s} \in [1, 2]$
  and hence $s^{1-q_i} \le (s+1)^{1-q_i} \le 2^{1-q_i} s^{1-q_i} $ for $i \in \{1, 2\}$
  which due to $\chi_i G_i'(s) = \int_1^s \frac{(\sigma+1)^{1-q_i}}{\sigma} \dsigma$ for $s \ge 1$ and $i \in \{1, 2\}$
  implies that
  \begin{align*}
        \frac{s^{1-q_i} L_{q_i}(s)}{1 - q_i \mathds 1_{\{q_i \lt 1\}}}
    \le \chi_i G_i'(s)
    \le \frac{2^{1-q_i} s^{1-q_i} L_{q_i}(s)}{1 - q_i \mathds 1_{\{q_i \lt 1\}}}
    \qquad \text{for $s \ge 1$ and $i \in \{1, 2\}$}.
  \end{align*}
  (We recall that $L_{q}(s) = \mathds 1_{\{q \lt 1\}} + \mathds 1_{\{q = 1\}} \ln s$ for $s \ge 1$ and $q \le 1$ by \eqref{eq:g_el_est:def_l_g}.)
  Combined with the facts that $\ln(s+\ure) - \ln s = \ln \frac{s+\ure}{s} \le \ln(1 + \ure)$ and $\ln s \le s^\eta$ for $s \ge 1$
  and Young's inequality, we thus obtain $c_1, c_2 \gt 0$ such that
  \begin{align*}
    &\pe  \left[ G_1'(\ua) \fo(\ua, \va) + G_2'(\va) \ft(\ua, \va) \right] \xioa(\ua) \xita(\va) \\
    &\le  \left[
            c_1 \ua^{2-q_1+\eta} \va
            - 2c_2 \ua^{3-q_1} L_{q_1}(\ua + \ure)
            - 2c_2 \va^{3-q_2} L_{q_2}(\va + \ure)
          \right] \xioa(\ua) \xita(\va)
          + c_1
  \end{align*}
  in $\{\ua \ge 1\} \cap \{\va \ge 1\}$ for all $\alpha \in (0, 1)$.

  We now distinguish between the cases $m_1 \gt \ul m_1$ and $m_2 \gt \ul m_2$.
  In the former one, we first employ Young's inequality to obtain $c_3 \gt 0$ such that
  \begin{align*}
    &\pe  c_1 \ua^{2-q_1 + \eta} \va \xioa(\ua) \xita(\va) \\
    &\le  c_3 \ua^{\frac{(3-q_2)(2-q_1 + \eta)}{2-q_2}} \xioa(\ua)
          + c_2 \va^{3-q_2} \xioa(\ua) \xita(\va)
    \qquad \text{in $\Omega \times (0, T)$ for all $\alpha \in (0, 1)$}
  \end{align*}
  and then make use of the assumption $m_1 \gt \ul m_1$
  which allows us to apply \eqref{eq:engery_rhs_f2:m1} and Lemma~\ref{lm:gni_f1}
  to obtain $c_4 \gt 0$ such that
  \begin{align*}
        c_3 \intntom \ua^{\frac{(3-q_2)(2-q_1 + \eta)}{2-q_2}} \xi_\alpha(\ua)
    \le \frac{d_1}{2 \chi_1} \intntom \Ba^{\beta_1}(\ua) |\nabla \ua|^2
        + c_4
    \qquad \text{for all $\alpha \in (0, 1)$},
  \end{align*}

  If on the other hand $m_2 \gt \ul m_2$, then we again make first use of Young's inequality to obtain $c_5 \gt 0$ such that
  \begin{align*}
        c_1 \ua^{2 - q_1 + \eta} \va \xioa(\ua) \xita(\va)
    \le c_2 \ua^{3 - q_1} \xioa(\ua) \xita(\va)
        + c_5 \va^{\frac{3-q_1}{1-\eta}} \xita(\va)
    \qquad \text{in $\Omega \times (0, T)$ for all $\alpha \in (0, 1)$}.
  \end{align*}
  According to Lemma~\ref{lm:gni_f1} (which is applicable thanks to \eqref{eq:engery_rhs_f2:m2}),
  there is then $c_6 \gt 0$ such that
  \begin{align*}
        c_5 \intntom  \va^{\frac{3-q_1}{1-\eta}} \xita(\va)
    \le \frac{d_2}{2 \chi_2} \intntom \Ba^{\beta_2}(\va) |\nabla \va|^2
        + c_6
    \qquad \text{for all $\alpha \in (0, 1)$}.
  \end{align*}
  In both cases $m_1 \gt \ul m_1$ and $m_2 \gt \ul m_2$,
  we then conclude from the estimates above that there is $c_7 \gt 0$ such that
  \begin{align*}
    &\pe \intntom \left[ G_1'(\ua) \fo(\ua, \va) + G_2'(\va) \ft(\ua, \va) \right] \xioa(\ua) \xita(\va)
            \mathds 1_{\{\ua \ge 1\} \cap \{\va \ge 1\}} \\
    &\le  \frac{d_1}{2 \chi_1} \intntom \Ba^{\beta_1}(\ua) |\nabla \ua|^2
          + \frac{d_2}{2 \chi_2} \intntom \Ba^{\beta_2}(\va) |\nabla \va|^2 \\
    &\pe  - c_2 \intntom \ua^{3-q_1} L_{q_1}(\ua + \ure) \xioa(\ua) \xita(\va)
          - c_2 \intntom \va^{3-q_2} L_{q_2}(\ua + \ure) \xioa(\ua) \xita(\va)
          + c_7
  \end{align*}
  for all $\alpha \in (0, 1)$,
  which in conjunction with \eqref{eq:alpha_entropy:statement} and \eqref{eq:entropy_rhs_basic:statement} gives the claim.
\end{proof} 

This concludes our journey of controlling the right-hand side in \eqref{eq:alpha_entropy:statement}.
As a consequence, we obtain the following a~priori bounds.
\begin{lemma}\label{lm:entropy_est_alpha}
  Let $T \in (0, \infty)$.
  There is $C_1 \gt 0$ such that
  \begin{align}
    &\pe  \sup_{t \in (0, T)} \left(
            \intom \Ba^{2 - q_1}(\ua(\cdot, t)) L_{q_1}(\ua(\cdot, t) + \ure)
            + \intom \Ba^{2 - q_1}(\va(\cdot, t)) L_{q_2}(\va(\cdot, t) + \ure)
          \right) \le C_1 \quad \text{and} \label{eq:entropy_est_alpha:linfty_lp} \\
    &\pe  \intntom \Ba^{\beta_1}(\ua) |\nabla \ua|^2
          + \intntom \Ba^{\beta_2}(\va) |\nabla \va|^2
          \le C_1  \label{eq:entropy_est_alpha:l2_l2}
  \end{align}
  for all $\alpha \in (0, 1)$,
  where again $\beta_i \defs m_i - q_i - i$ for $i \in \{1, 2\}$,
  and $L_{q_i}$ and $\Ba$ are as in \eqref{eq:g_el_est:def_l_g} and \eqref{eq:alpha_entropy:ba}, respectively.
  Moreover, if \eqref{eq:intro:h2} holds, then we can find $C_2 \gt 0$ with the property that
  \begin{align}\label{eq:entropy_est_alpha:f_l1}
          \intntom \ua^{2} \ln(\ua + \ure)
          + \intntom \va^{2} \ln(\va + \ure)
    &\le  C_2 
    \qquad \text{for all $\alpha \in (0, 1)$}.
  \end{align}
\end{lemma}
\begin{proof}
  According to Lemma~\ref{lm:alpha_entropy}, Lemma~\ref{lm:rhs_bdd_f1} and Lemma~\ref{lm:rhs_bdd_f2},
  there are $c_1, c_2 \gt 0$ and $c_3 \ge 0$
  such that $c_3$ is positive if \eqref{eq:intro:h2} holds
  and
  \begin{align*}
    &\pe    c_1 \intom \Ba^{2 - q_1}(\ua(\cdot, t)) L_{q_1}(\ua(\cdot, t) + \ure)
          + c_1 \intom \Ba^{2 - q_2}(\va(\cdot, t)) L_{q_2}(\va(\cdot, t) + \ure) \notag \\
    &\pe  + \frac{d_1}{2 \chi_1} \intntom \Ba^{\beta_1}(\ua) |\nabla \ua|^2
          + \frac{d_2}{2 \chi_2} \intntom \Ba^{\beta_2}(\va) |\nabla \va|^2 \notag \\
    &\le  c_2
          - c_3 \intntom \ua^2 \ln(\ua + \ure)
          - c_3 \intntom \va^2 \ln(\va + \ure)
    \qquad \text{for $t \in (0, T)$ and $\alpha \in (0, 1)$},
  \end{align*}
  as desired.
\end{proof}

\subsection{Space-time bounds and the limit process}
As a next step, we derive further space-time bounds from \eqref{eq:entropy_est_alpha:linfty_lp} and \eqref{eq:entropy_est_alpha:l2_l2}.
To that end, we make use of the following interpolation inequality
which is both a refinement and a consequence of the Gagliardo--Nirenberg inequality
and has been proven by Tao and Winkler in \cite{TaoWinklerFullyCrossdiffusiveTwocomponent2021}.
\begin{lemma}\label{lm:gni_tao_winkler}
  Let $0 \lt q \lt p \lt \frac{2n}{(n-2)_+}$
  and suppose that $\Lambda \in C^0(\R)$ fulfills $\Lambda \ge 1$ on $\R$.
  Then there exist $C \gt 0$ and $\theta \in (0, 1]$ such that
  \begin{align*}
        \intom |\varphi|^p \Lambda^\theta(\varphi)
    \le C \left( \intom |\nabla \varphi|^2 \right)^\frac{pb}{2}
        \left( \intom |\varphi|^q \Lambda(\varphi) \right)^\frac{p(1-b)}{q}
        + C \left( \intom |\varphi|^q \Lambda(\varphi) \right)^\frac{p}{q}
    \qquad \text{for all $\varphi \in \sob12$},
  \end{align*}
  where
  \begin{align*}
    b \defs \frac{\frac1q - \frac1p}{\frac1q + \frac1n - \frac12} \in (0, 1).
  \end{align*}
\end{lemma}
\begin{proof}
  This is a direct consequence of \cite[Lemma~7.5]{TaoWinklerFullyCrossdiffusiveTwocomponent2021}.
\end{proof}

\begin{lemma}\label{lm:space_time_bdds_alpha}
  For all $T \in (0, \infty)$, there are $C \gt 0$ and $\theta_1, \theta_2 \in (0, 1]$ such that
  \begin{align}\label{eq:space_time_bdds_alpha:lp_lp}
    &\intntom \Ba^{p_1}(\ua) L_{q_1}^{\theta_1}(\Ba(\ua) + \ure) + \intntom \Ba^{p_2}(\va) L_{q_2}^{\theta_2}(\Ba(\va) + \ure) \le C
    \qquad \text{for all $\alpha \in (0, 1)$},
  \end{align}
  where $p_1$ and $p_2$ are as in \eqref{eq:intro:pi},
  and $L_{q_i}$ and $\Ba$ are as in \eqref{eq:g_el_est:def_l_g} and \eqref{eq:alpha_entropy:ba}, respectively.
\end{lemma}
\begin{proof}
  We fix $T \in (0, \infty)$. As usual, it suffices to show the statement for the first solution component.

  Let us first assume $p_1 = 3-q_1$ and that \eqref{eq:intro:h2} holds.
  Then \eqref{eq:entropy_est_alpha:f_l1} already contains \eqref{eq:space_time_bdds_alpha:lp_lp}.
  Moreover, if $p_1 = 2-q_1$, then \eqref{eq:entropy_est_alpha:linfty_lp}
  and an integration in time also show \eqref{eq:space_time_bdds_alpha:lp_lp}.
  According to \eqref{eq:intro:pi}, it remains to be shown that \eqref{eq:space_time_bdds_alpha:lp_lp}
  also holds for $2-q_1 \lt p_1 = \beta_1 + 2 + \frac{2(2-q_1)}{n}$, where again $\beta_1 \defs m_1 - q_1 - 1$.
  As already alluded to,
  the main ingredients for this proof are \eqref{eq:entropy_est_alpha:linfty_lp} and \eqref{eq:entropy_est_alpha:l2_l2}
  which assert that there are $c_1, c_2 \gt 0$ such that  
  \begin{align*}
    \sup_{t \in (0, T)} \intom \Ba^{2-q_1}(\ua(\cdot, t)) L_{q_1}(\ua(\cdot, t) + \ure) \le c_1
    \quad \text{and} \quad
    \intntom \Ba^{\beta_1}(\ua) |\nabla \ua|^2 \le c_2.
  \end{align*}
  Preparing an application of Lemma~\ref{lm:gni_tao_winkler},
  we set
  $\tilde q_1 \defs \frac{2(2-q_1)}{\beta_1+2}$,
  which is positive as $\beta_1 \gt -2$ is contained in \eqref{eq:params:statement}.
  Moreover,
  \begin{align*}
      \tilde p_1
    \defs \frac{2(n+\tilde q_1)}{n}
    = 2\left(1+\frac{\tilde q_1}{n}\right)
    = \frac{2 (\beta_1+2 + \frac{2(2-q_1)}{n})}{\beta_1+2}
    = \frac{2 p_1}{\beta_1+2}
    \gt \tilde q_1
  \end{align*}
  thanks to $p_1 \gt 2-q_1$.
  Thus, $\tilde p_1 \lt \frac{2(n + \tilde p_1)}{n}$ and hence $\frac{n-2}{n} \tilde p_1 \lt 2$
  which in turn implies $\tilde p_1 \lt \frac{2n}{(n-2)_+}$.
  Therefore, we may indeed apply Lemma~\ref{lm:gni_tao_winkler} to obtain $c_3 \gt 0$, $\theta_1 \in (0, 1]$ and $b \in (0, 1)$ such that
  with $\Lambda(s) \defs L_{q_1}(s^\frac{2}{\beta_1+2} + \ure)$, $s \ge 0$,
  \begin{align*}
          \intom \varphi^{\tilde p_1} \Lambda^{\theta_1}(\varphi)
    &\le  c_3 \left( \intom |\nabla \varphi|^2 \right)^{\frac{\tilde p_1 b}{2}}
            \left( \intom \varphi^{\frac{2 (2-q_1)}{\beta_1+2}} \Lambda(\varphi) \right)^{\frac{\tilde p_1 (1-b)}{\tilde q_1}}
          + c_3 \left( \intom \varphi^{\frac{2 (2-q_1)}{\beta_1+2}} \Lambda(\varphi) \right)^{\frac{\tilde p_1}{\tilde q_1}}
  \end{align*}
  for all nonnegative $\varphi \in \sob12$.
  Taking here $\varphi = \Ba^{p_1}(\ua(\cdot, t))$, $t \in (0, T)$, and integrating in time yield
  \begin{align*}
    &\pe  \intntom \Ba^{p_1}(\ua) L_{q_1}^{\theta_1}(\Ba(\ua) + \ure) \\
    &=    \intntom \left( \Ba^\frac{\beta_1+2}{2}(\ua) \right)^\frac{2p_1}{\beta_1+2} \Lambda^{\theta_1}(\Ba^\frac{\beta_1+2}{2}(\ua)) \\
    &\le  c_3 \intnt \left( \intom \left| \nabla \Ba^\frac{\beta_1+2}{2}(\ua) \right|^2 \right)^\frac{\tilde p_1 b}{2}
            \left( \intom \Ba^{2-q_1}(\ua) L_{q_1}(\ua + \ure) \right)^{\frac{\tilde p_1(1-b)}{\tilde q_1}} \\
    &\pe  + c_3 \intnt \left( \intom \Ba^{2-q_1}(\ua) L_{q_1}(\ua + \ure) \right)^{\frac{\tilde p_1}{\tilde q_1}} \\
    &\le  \frac{c_1^{\tilde p_1/\tilde q_1} c_3 (\beta_1+2)^2}{4} \intntom \Ba^{\beta_1}(\ua) |B'(\ua)|^2 |\nabla \ua|^2
          + T c_1^{\tilde p_1 / \tilde q_1} c_3 \\
    &\le  \frac{c_1^{\tilde p_1/\tilde q_1} c_2 c_3 (\beta_1+2)^2}{4}
          + T c_1^{\tilde p_1 / \tilde q_1} c_3
    \qquad \text{for all $\alpha \in (0, 1)$},
  \end{align*}
  where in the last step we have used that $|\Ba'(s)| = \frac{1}{(1 + \alpha (s+1))^2} \le 1$ for $s \ge 0$ and $\alpha \in (0, 1)$.
  Thus, \eqref{eq:space_time_bdds_alpha:lp_lp} indeed holds in all cases treated by this lemma.
\end{proof}

As an application of Young's inequality reveals, 
\eqref{eq:entropy_est_alpha:l2_l2} and \eqref{eq:space_time_bdds_alpha:lp_lp} allow us to also obtain gradient space-time bounds.
\begin{lemma}\label{lm:gradient_space_time_bdds_alpha}
  Let $T \in (0, \infty)$ and $r_1, r_2$ be as in \eqref{eq:intro:ri}.
  Then there is $C \gt 0$ such that
  \begin{align}\label{eq:gradient_space_time_bdds_alpha:lr_w1r}
    &\intntom |\nabla \ua|^{r_1} + \intntom |\nabla \va|^{r_2} \le C
    \qquad \text{for all $\alpha \in (0, 1)$}.
  \end{align}
\end{lemma}
\begin{proof}
  Again, it suffices to prove the bound only for $\ua$, $\alpha \in (0, 1)$.
  We first assume that $r_1 \lt 2$
  and hence $r_1 = \frac{2p_1}{p_1 - \beta_1}$ by \eqref{eq:intro:ri}, where $\beta_1 \defs m_1 - q_1 - 1$.
  With $\Ba$ as in \eqref{eq:alpha_entropy:ba}, we then make use of Young's inequality to obtain
  \begin{align*}
          \intntom |\nabla \ua|^{r_1}
    &=    \intntom \Ba^\frac{\beta_1 r_1}{2}(\ua) |\nabla \ua|^{r_1} \Ba^{-\frac{\beta_1 r_1}{2}}(\ua) \\
    &\le  \frac{r_1}{2} \intntom \Ba^{\beta_1}(\ua) |\nabla \ua|^2
          + \frac{2-r_1}{2} \intntom \Ba^{-\frac{\beta_1 r_1}{2-r_1}}(\ua)
  \end{align*}
  for all $\alpha \in (0, 1)$ which due to \eqref{eq:space_time_bdds_alpha:lp_lp} and
  \begin{align*}
      -\frac{\beta_1 r_1}{2-r_1}
    = \frac{-\beta_1}{\frac2{r_1}-1}
    = \frac{-\beta_1}{\frac{p_1-\beta_1}{p_1}-1}
    = \frac{-\beta_1}{\frac{-\beta_1}{p_1}}
    = p_1
  \end{align*}
  implies \eqref{eq:gradient_space_time_bdds_alpha:lr_w1r} for some $C \gt 0$.

  If, on the other hand $r_1 \ge 2$ and hence $r_1 = 2 \le \frac{2 p_1}{p_1 - \beta_1}$ by \eqref{eq:intro:ri},
  then $\beta_1 \ge 0$ since positivity of $p_1$ is contained in \eqref{eq:params:statement}.
  Thus, in this case the estimate \eqref{eq:entropy_est_alpha:l2_l2} directly implies \eqref{eq:gradient_space_time_bdds_alpha:lr_w1r}.
\end{proof}

As a last preparation before obtaining limit functions $u$ and $v$ by applying several compactness theorems---in particular, the Aubin--Lions lemma---,
we derive estimates for the time derivatives $\uat$ and $\vat$, $\alpha \in (0, 1)$.
\begin{lemma}\label{lm:uat_vat_bdd}
  Let $T \in (0, \infty)$. Then there exists $C \gt 0$ such that
  \begin{align}\label{eq:uat_vat_bdd:statement}
    \|\uat\|_{L^1((0, T); \dual{\sob{n+1}{2}})}
    + \|\vat\|_{L^1((0, T); \dual{\sob{n+1}{2}})}
    \le C
    \qquad \text{for all $\alpha \in (0, 1)$}.
  \end{align}
\end{lemma}
\begin{proof}
  Since $\ua \in L^2((0, T); \sob12)$ by Lemma~\ref{lm:ex_ua_va},
  the weak formulation \eqref{eq:weak_sol_nonlin:u_sol} entails that
  \begin{align*}
      \intom \uat(\cdot, t) \psi
    = - \intom \Doa(\ua(\cdot, t)) \nabla \ua(\cdot, t) \cdot \nabla \psi
      + \intom \Soa(\ua(\cdot, t)) \nabla \va(\cdot, t) \cdot \nabla \psi
      + \intom \foa(\ua(\cdot, t), \va(\cdot, t)) \psi
  \end{align*}
  for a.e.\ $t \in (0, T)$, all $\psi \in \sob12$ and all $\alpha \in (0, 1)$.
  Thus, recalling that $\Doa(\ua) \le d_1 \Ba^{m_1-1}(\ua) + 1$ and $\Soa(\ua) \le \chi_1 \Ba^{q_1}(\ua)$ for $\alpha \in (0, 1)$
  if $\Ba$ as in \eqref{eq:alpha_entropy:ba},
  we may estimate
  \begin{align*}
    &\pe  \left| \intom \uat(\cdot, t) \psi \right| \\
    &\le  \left| \intom (\Doa(\ua(\cdot, t)) \nabla \ua(\cdot, t) \cdot \nabla \psi \right|
          + \left| \intom \Soa(\ua(\cdot, t)) \nabla \va(\cdot, t) \cdot \nabla \psi \right|
          + \left| \intom \foa(\ua(\cdot, t), \va(\cdot, t)) \psi \right| \\
    &\le  d_1 \left(
            \intom \left( \Ba^{m_1-1-\frac{\beta_1}{2}}(\ua(\cdot, t)) + 1 \right)^2
            + \intom \left( \left( \Ba^{\frac{\beta_1}{2}}(\ua(\cdot, t)) + 1 \right) |\nabla \ua(\cdot, t)| \right)^2
          \right) \|\nabla \psi\|_{\leb\infty} \\
    &\pe  + \chi_1 \left(
            \intom \left( \Ba^{q_1}(\ua(\cdot, t)) \right)^{\frac{r_2}{r_2-1}}
            + \intom |\nabla \va(\cdot, t)|^{r_2}
          \right) \|\nabla \psi\|_{\leb\infty} \\
    &\pe  + \left( \intom |\foa(\ua(\cdot, t), \va(\cdot, t))| \right)
          \| \psi\|_{\leb\infty}
    \qquad \text{for a.e.\ $t \in (0, T)$, all $\psi \in \sob1\infty$ and all $\alpha \in (0, 1)$},
  \end{align*}
  wherein as usual $\beta_1 \defs m_1-q_1-1$.
  As according to \eqref{eq:params:statement} and \eqref{eq:intro:main_cond},
  both $2(m_1-1-\frac{\beta_1}{2})$ and $\frac{\max\{q_1, 0\} r_2}{r_2-1}$ are at most $p_1$,
  the bounds \eqref{eq:space_time_bdds_alpha:lp_lp}, \eqref{eq:entropy_est_alpha:l2_l2},
  \eqref{eq:gradient_space_time_bdds_alpha:lr_w1r} and \eqref{eq:entropy_est_alpha:f_l1}
  along with the embeddings $\sob{n+1}{2} \embed \sob1\infty \embed \leb\infty$
  and an integration in time
  yield $c_1 \gt 0$ such that
  \begin{align*}
          \intnt \sup_{\substack{\psi \in \sob{n+1}{2} \\ \|\psi\|_{\sob{n+1}{2}} \le 1}} \left| \intom \uat \psi \right|
    &\le  c_1
    \qquad \text{for all $\alpha \in (0, 1)$},
  \end{align*}
  which together with analogous considerations regarding $\vat$ implies \eqref{eq:uat_vat_bdd:statement}.
\end{proof}

The a~priori bounds gained in the lemmata above now allow us to conclude
that $(\ua, \va)$ converge in certain spaces along some null sequence $(\alpha_j)_{j \in \N}$.
\begin{lemma}\label{lm:alpha_sea_0}
  Set
  \begin{align*}
    \mc P_i \defs
    \begin{cases}
      [1, p_i), & q_i \lt 1, \\
      [1, p_i], & q_i = 1.
    \end{cases}
  \end{align*}
  Then there exists a null sequence $(\alpha_j)_{j \in \N} \subset (0, 1)$ and nonnegative $u, v \in L_{\loc}^1(\Ombar \times [0, \infty))$
  such that
  \begin{alignat}{2}
    \uaj &\ra u
      &&\qquad \text{pointwise a.e.}, \label{eq:alpha_sea_0:u_pw} \\
    \vaj &\ra v
      &&\qquad \text{pointwise a.e.}, \label{eq:alpha_sea_0:v_pw} \\
    \Ba(\uaj) &\ra u + 1
      &&\qquad \text{in $L_{\loc}^{p}(\Ombar \times [0, \infty))$ for all $p \in \mc P_1$}, \label{eq:alpha_sea_0:u_lp} \\
    \Ba(\vaj) &\ra v + 1
      &&\qquad \text{in $L_{\loc}^{p}(\Ombar \times [0, \infty))$ for all $p \in \mc P_2$}, \label{eq:alpha_sea_0:v_lp} \\
    \uaj &\rh u
      &&\qquad \text{in $L_{\loc}^{r_1}([0, \infty); \sob{1}{r_1})$}, \label{eq:alpha_sea_0:u_nabla_lr} \\
    \vaj &\rh v
      &&\qquad \text{in $L_{\loc}^{r_1}([0, \infty); \sob{1}{r_1})$}, \label{eq:alpha_sea_0:v_nabla_lr} \\
    \foa(\ua, \va) &\ra \fo(u, v)
      &&\qquad \text{in $L_{\loc}^{1}(\Ombar \times [0, \infty))$} \quad \text{and} \label{eq:alpha_sea_0:f1_l1} \\
    \fta(\ua, \va) &\ra \ft(u, v)
      &&\qquad \text{in $L_{\loc}^{1}(\Ombar \times [0, \infty))$} \label{eq:alpha_sea_0:f2_l1}
  \end{alignat}
  as $j \ra 0$, where $\Ba$ is as in $\eqref{eq:alpha_entropy:ba}$ for $\alpha \in (0, 1)$.
\end{lemma}
\begin{proof}
  Thanks to \eqref{eq:gradient_space_time_bdds_alpha:lr_w1r} and \eqref{eq:uat_vat_bdd:statement},
  the Aubin--Lions lemma (along with a diagonalization argument) provides us with
  a null sequence $(\alpha_j)_{j \in \N} \subset (0, 1)$ and functions $u, v \in L_{\loc}^1(\Ombar \times [0, \infty))$
  such that $\uaj \ra u$ and $\vaj \ra v$ in $L_{\loc}^1(\Ombar \times [0, \infty))$ as $j \ra \infty$.
  After switching to a subsequence, if necessary, we may thus assume that \eqref{eq:alpha_sea_0:u_pw} and \eqref{eq:alpha_sea_0:v_pw} hold.
  Thus, nonnegativity of $u$ and $v$ is inherited from nonnegativity of $\uaj$ and $\vaj$, $j \in \N$,
  which in turn is asserted by Lemma~\ref{lm:ex_ua_va}.
  Due to the bound \eqref{eq:space_time_bdds_alpha:lp_lp},
  and because $\Ba(\ua) \ra \ua+1$ and $\Ba(\va) \ra \va+1$ pointwise a.e.\ as $\alpha \sea 0$ by \eqref{eq:alpha_sea_0:u_pw} and \eqref{eq:alpha_sea_0:v_pw},
  Vitali's theorem asserts that \eqref{eq:alpha_sea_0:u_lp} and \eqref{eq:alpha_sea_0:v_lp} hold.

  Moreover, possibly after switching to further subsequences,
  \eqref{eq:alpha_sea_0:u_nabla_lr} and \eqref{eq:alpha_sea_0:v_nabla_lr} follow from \eqref{eq:gradient_space_time_bdds_alpha:lr_w1r}.
  (We note that \eqref{eq:alpha_sea_0:u_pw} and \eqref{eq:alpha_sea_0:v_pw} guarantee that the corresponding limit functions coincide.)

  Finally, additional consequences of \eqref{eq:alpha_sea_0:u_pw} and \eqref{eq:alpha_sea_0:v_pw}
  are \eqref{eq:alpha_sea_0:f1_l1} and \eqref{eq:alpha_sea_0:f2_l1}:
  For fixed $T \in (0, \infty)$, the complement of
  \begin{align*}
    A \defs
    \left\{\,
      (x, t) \in \Omega \times (0, T) :
      \max\{u(x, t), v(x, t)\} \lt \infty \text{ and } (\uaj, \vaj)(x, t) \ra (u, v)(x, t) \text{ as } j \ra \infty
    \,\right\}
  \end{align*}
  in $\Omega \times (0, T)$ is a null set (since the inclusions $u, v \in L^1(\Omega \times (0, T))$ imply $u, v \lt \infty$ a.e.).
  Given $(x, t) \in A$, there is $M \gt 0$ with $\max\{u(x, t), v(x, t)\} \lt M$.
  Thus, we can find $j_1 \in \N$ such that $\max\{\uaj(x, t), \vaj(x, t)\} \lt 2M$ for all $j \ge j_1$.
  Taking moreover $j_2 \in \N$ so large that $2M \le \alpha_j^{-1/(4-\min\{q_1, q_2\})}$ for all $j \ge j_2$, we see that
  $\xi_{\alpha_j}(u(x, t)) = \xi_{\alpha_j}(v(x, t)) = 1$ and hence
  $\foaj(\uaj(x, t), \vaj(x, t)) = \fo(\uaj(x, t), \vaj(x, t))$ for all $j \ge \max\{j_1, j_2\}$
  so that $\foaj(\uaj(x, t), \vaj(x, t)) \ra \fo(u(x, t), v(x, t))$ as $j \ra \infty$ by the continuity of $\fo$.
  Since $(x, t) \in A$ was arbitrary, $\foaj(\uaj, \vaj) \ra \fo(u, v)$ a.e.\ as $j \ra \infty$.
  In the case of \eqref{eq:intro:h1}, \eqref{eq:alpha_sea_0:f1_l1} is trivially true
  while for \eqref{eq:intro:h2}, we make first use of Young's inequality to obtain $c_1 \gt 0$ such that
  $|\foa(s_1, s_2)| \le c_1 (s_1^2 + s_2^2 + 1)$ for all $s_1, s_2 \ge 0$ and~$\alpha \in (0, 1)$
  and then employ Vitali's theorem along with \eqref{eq:entropy_est_alpha:f_l1} and the just obtained pointwise convergence of $\foa$
  to also obtain \eqref{eq:alpha_sea_0:f1_l1} in that case.
  As usual, \eqref{eq:alpha_sea_0:f2_l1} can be shown analogously.
\end{proof}

\section{Existence of global weak solutions to \eqref{prob:nonlin}: proof of Theorem~\ref{th:ex_weak_nonlin}}\label{sec:proof_11}
In this final section, we show that the pair $(u, v)$ constructed in Lemma~\ref{lm:alpha_sea_0}
is a solution to \eqref{prob:nonlin} in the following sense.
\begin{definition}\label{def:weak_sol_main}
  A pair $(u, v) \in L_{\loc}^1(\Ombar \times [0, \infty))$ is called a \emph{global nonnegative weak solution} of \eqref{prob:nonlin}
  if $u, v \ge 0$,
  \begin{align*}
    D_1(u) \nabla u,
    S_1(u) \nabla v,
    D_2(u) \nabla v,
    S_2(v) \nabla u,
    f_1,
    f_2
    \in L_{\loc}^1(\Ombar \times [0, \infty))
  \end{align*}
  and 
  \begin{align}\label{eq:weak_sol_main:u_sol}
        - \intninfom u \varphi_t - \intom u_0 \varphi(\cdot, 0)
    &=  - \intninfom D_1(u) \nabla u \cdot \nabla \varphi
        + \intninfom S_1(u) \nabla v \cdot \nabla \varphi
        + \intninfom \fo(u, v) \varphi \\
    \intertext{as well as}\label{eq:weak_sol_main:v_sol}
        - \intninfom v \varphi_t - \intom v_0 \varphi(\cdot, 0)
    &=  - \intninfom D_2(u) \nabla v \cdot \nabla \varphi
        - \intninfom S_2(u) \nabla u \cdot \nabla \varphi
        + \intninfom \ft(u, v) \varphi
  \end{align}
  hold for all $\varphi \in C_c^\infty(\Ombar \times [0, \infty))$.
\end{definition}

\begin{lemma}\label{lm:u_v_weak_sol_main}
  The tuple $(u, v)$ constructed in Lemma~\ref{lm:alpha_sea_0}
  is a weak solution of \eqref{prob:nonlin} in the sense of Definition~\ref{def:weak_sol_main}.
\end{lemma}
\begin{proof}
  Both the required regularity and nonnegativity of $u$ and $v$ are contained in Lemma~\ref{lm:alpha_sea_0}.

  In order show that \eqref{eq:weak_sol_main:u_sol} holds,
  we first fix $\varphi \in C_c^\infty(\Ombar \times [0, \infty))$.
  For all $\alpha \in (0, 1)$,
  the pair $(\ua, \va)$ given by Lemma~\ref{lm:ex_ua_va} solves \eqref{prob:alpha} weakly
  so that by \eqref{eq:weak_sol_nonlin:u_sol} and an integration by parts,
  \begin{align}\label{eq:u_v_weak_sol_main:ua_va_weak_sol}
        I_{1 \alpha} + I_{2 \alpha}
    &\defs - \intninfom \ua \varphi_t
        - \intom \una \varphi(\cdot, 0) \notag \\
    &=  - \intninfom \Doa(\ua) \nabla \ua \cdot \nabla \varphi
        + \intninfom \Soa(\ua) \nabla \va \cdot \nabla \varphi
        + \intninfom \foa(\ua, \va) \varphi \notag \\
    &\sfed  I_{3 \alpha} + I_{4 \alpha} + I_{5 \alpha}
    \qquad \text{for all $\alpha \in (0, 1)$}.
  \end{align}
  Mainly relying on the convergences provided by Lemma~\ref{lm:alpha_sea_0}, we now take the limit $\alpha = \alpha_j \sea 0$ in each term herein.
  First, 
  \begin{align*}
    I_{2 \alpha_j} \ra - \intom u_0 \varphi(\cdot, 0)
    \quad \text{and} \quad
    I_{5 \alpha_j} \ra \intninfom \fo(u, v) \varphi
    \qquad \text{as $j \ra \infty$}
  \end{align*}
  are direct consequences of \eqref{eq:def_una_vna:conv_lebl} and \eqref{eq:alpha_sea_0:f1_l1}.
  Moreover, as $r_1 \gt 1$ by \eqref{eq:intro:main_cond},
  we infer from \eqref{eq:alpha_sea_0:u_nabla_lr} that $\uaj \ra u$ in $L_{\loc}^1(\Ombar \times [0, \infty))$ and thus
  \begin{align*}
    I_{1 \alpha_j} \ra - \intninfom u \varphi_t
    \qquad \text{as $j \ra \infty$}.
  \end{align*}
  Regarding $I_{3 \alpha}$, we first note that in the case of $m_1 \le 1$,
  \begin{align}\label{eq:u_v_weak_sol_main:i3_m1_lt_1}
    \Baj^{m_1 - 1}(\uaj) \ra (u+1)^{m_1 - 1}
    \qquad \text{in $L_{\loc}^\frac{r_1}{r_1-1}(\Ombar \times [0, \infty))$ as $j \ra \infty$}
  \end{align}
  by Lebesgue's theorem and \eqref{eq:alpha_sea_0:u_pw},
  where $\Ba$ is as in \eqref{eq:alpha_entropy:ba} for $\alpha \in (0, 1)$.
  We now show that \eqref{eq:u_v_weak_sol_main:i3_m1_lt_1} also holds for $m_1 \gt 1$.
  If additionally $r_1 = \frac{2p_1}{p_1-\beta_1}$ with $\beta_1 \defs m_1-q_1-1$,
  then $(m_1-1) \frac{r_1}{r_1-1} = (m_1-1) \frac{2p_1}{p_1+\beta_1} \lt p_1$
  since $0 \lt 2(m_1-1) \lt p_1+\beta_1$ is entailed in \eqref{eq:params:statement}.
  If on the other hand ($m_1 \gt 1$ and) $r_1 \neq \frac{2p_1}{p_1-\beta_1}$ and thus $r_1 = 2 \gt \frac{2p_1}{p_1-\beta_1}$ by \eqref{eq:intro:ri},
  then $\beta_1 \lt 0$ so that \eqref{eq:params:statement} asserts $2(m_1 - 1) \lt p_1$
  and hence also $(m_1-1) \frac{r_1}{r_1-1} \lt p_1$.
  Therefore, \eqref{eq:alpha_sea_0:u_lp} asserts that \eqref{eq:u_v_weak_sol_main:i3_m1_lt_1} indeed also holds for $m_1 \gt 1$.
  Combined with \eqref{eq:alpha_sea_0:u_nabla_lr}, \eqref{eq:u_v_weak_sol_main:i3_m1_lt_1} then implies
  \begin{align*}
        \intninfom \Baj^{m_1 - 1}(\uaj) \nabla \uaj \cdot \nabla \varphi
    \ra \intninfom (u+1)^{m_1 - 1} \nabla u \cdot \nabla \varphi
    \qquad \text{as $j \ra \infty$},
  \end{align*}
  and since additionally $\alpha_j \intninfom \nabla \uaj \cdot \nabla \varphi \ra 0$ as $j \ra \infty$ by \eqref{eq:alpha_sea_0:u_nabla_lr},
  we conclude
  \begin{align*}
    I_{3 \alpha_j} \ra - \intninfom D_1(u) \nabla u \cdot \nabla \varphi
    \qquad \text{as $j \ra \infty$}.
  \end{align*}
  Finally, we concern ourselves with the term stemming from the cross diffusion:
  Precisely due to our main condition \eqref{eq:intro:main_cond},
  we can choose $p \gt 1$ such that
  \begin{align*}
    \frac1p + \frac1{r_1} = 1
    \quad \text{and} \quad
    p \in \begin{cases}
      [1, \infty), & q_1 \le 0, \\
      [1, \frac{p_1}{q_1}), & 0 \lt q_1 \lt 1, \\
      [1, p_1], & q_1 = 1.
    \end{cases}
  \end{align*}
  As also 
  $0 \le \Soa(s) \le \chi_1 \Ba^{q_1}(s)$ for all $s \ge 0$ and $\alpha \in (0, 1)$
  as well as $\Soaj(\uaj) \ra S_1(\uaj)$ a.e.\ as $\alpha \sea 0$,
  Pratt's lemma and \eqref{eq:alpha_sea_0:u_lp} assert that $\Soaj^p(\uaj) \ra S_1^p(u)$
  in $L_{\loc}^1(\Ombar \times [0, \infty))$ as $j \ra \infty$, provided that $q_1 \ge 0$.
  For $q_1 \lt 0$, the same conclusion can be reached by Lebesgue's theorem.
  Combined with \eqref{eq:alpha_sea_0:v_nabla_lr}, this entails that
  $\Soaj(\uaj) \nabla \vaj \rh S_1(u) \nabla v$ in $L_{\loc}^1(\Ombar \times [0, \infty))$ as $j \ra \infty$ and thus
  \begin{align*}
    I_{4 \alpha_j} \ra \intninfom S_1(u) \nabla v \cdot \nabla \varphi
    \qquad \text{as $j \ra \infty$}.
  \end{align*}
  In combination, these convergences and \eqref{eq:u_v_weak_sol_main:ua_va_weak_sol} prove \eqref{eq:weak_sol_main:u_sol},
  and since \eqref{eq:weak_sol_main:v_sol} can be shown analogously,
  $(u, v)$ is indeed a weak solution of \eqref{prob:nonlin}.
\end{proof}

This lemma already contains our main theorem.
\begin{proof}[Proof of Theorem~\ref{th:ex_weak_nonlin}]
  All claims have been proven in Lemma~\ref{lm:u_v_weak_sol_main}.
\end{proof}

\section*{Acknowledgments}
The author is partially supported by the German Academic Scholarship Foundation
and by the Deutsche Forschungsgemeinschaft within the project \emph{Emergence of structures and advantages in
cross-diffusion systems}, project number 411007140.

\end{document}